\documentclass[a4paper,10pt]{article}
\usepackage{authblk}

\usepackage[ngerman,english]{babel}
\usepackage[latin1]{inputenc}
\usepackage{csquotes}
\usepackage[normalem]{ulem}
\usepackage{amsfonts,amsmath,amsthm}
\usepackage{empheq}
\usepackage{cancel}

\usepackage[normalem]{ulem}
\newcommand\cmgout{\bgroup\markoverwith
{\textcolor{magenta}{\rule[.5ex]{2pt}{0.4pt}}}\ULon}
\usepackage[colorinlistoftodos,textsize=small]{todonotes}

\usepackage{cases}
\usepackage{mathabx}
\usepackage{xfrac}
\usepackage{fancyhdr}
\usepackage{color}
\usepackage[colorlinks]{hyperref}
\definecolor{blue75}{rgb}{0,0,.75}
\definecolor{green75}{rgb}{0,.75,0}
\hypersetup{colorlinks=true, urlcolor=blue75,linkcolor=blue75,citecolor=green75,pdfstartview=FitB,bookmarksopen=true,bookmarksopenlevel=1}
\usepackage[a4paper, left=2.5cm, right=2.5cm, top=2.5cm,bottom=2cm]{geometry}
\usepackage{constants}
\newcommand{\parenthezises}[1]{\arabic{#1}}
\newconstantfamily{C}{
symbol=C,
format=\parenthezises,
reset={section}
}
\newconstantfamily{M}{
symbol=M,
format=\parenthezises,
reset={section}
}
\newconstantfamily{a}{
symbol=a,
format=\parenthezises,
reset={section}
}
\usepackage{enumerate}

\usepackage{graphicx}
\graphicspath{ {images/} }
\usepackage{wrapfig}
\usepackage{figbib}

\allowdisplaybreaks

\usepackage[capitalise]{cleveref}

\crefdefaultlabelformat{{\it #2#1#3}}

\crefname{equation}{}{}

\crefname{enumi}{}{}
\creflabelformat{enumi}{{(#2#1#3)}}

\crefname{section}{{\it Section}}{{\it Sections}}
\crefname{subsection}{{\it Subsection}}{{\it Subsections}}
\crefname{subsubsection}{{\it Paragraph}}{{\it Paragraphs}}
\crefname{table}{{\it Table}}{{\it Tables}}
\crefname{figure}{{\it Figure}}{{\it Figures}}
\newtheorem{Theorem}{Theorem}[section]
\crefname{Theorem}{{\it Theorem}}{{\it Theorems}}
\newtheorem{Definition}[Theorem]{Definition}
\crefname{Definition}{{\it Definition}}{{\it Definitions}}
\newtheorem{Lemma}[Theorem]{Lemma}
\crefname{Lemma}{{\it Lemma}}{{\it Lemmas}}
\newtheorem{Assumptions}[Theorem]{Assumptions}
\crefname{Assumptions}{{\it Assumptions}}{{\it Assumptions}}

\theoremstyle{definition}
\newtheorem{Remark}[Theorem]{Remark}
\crefname{Remark}{{\it Remark}}{{\it Remarks}}
\newtheorem{Notation}[Theorem]{Notation}
\crefname{Notation}{{\it Notation}}{{\it Notations}}
\newtheorem{Example}[Theorem]{Example}
\crefname{Example}{{\it Example}}{{\it Examples}}

\begin{document}

\newcommand{\FOO}{n(n+1)}
\definecolor{darkgreen}{rgb}{0.1,0.8,0.5}
\definecolor{grau}{rgb}{0.1,0.3,0.8}
%\definecolor{orange}{rgb}{1,0.5,0}
\definecolor{orange}{rgb}{1,0.2,0.4}
\definecolor{turcoaz}{rgb}{0.1,0.5,0.7}
% % \newcommand{\cb}{\color{blue}}
% \newcommand{\cb}{\color{blue}}
% % \newcommand{\cred}{\color{red}}
% %  \newcommand{\cmg}{\color{magenta}}
% \newcommand{\cmg}[1]{\textcolor{magenta}{#1}}
% \newcommand{\red}[1]{\textcolor{red}{#1}}
% \newcommand{\cg}{\color{darkgreen}}
% \newcommand{\cy}{\color{cyan}}
% \newcommand{\co}{\color{orange}}
% \newcommand{\ct}{\color{turcoaz}}
% \newcommand{\cgr}{\color{grau}}
%%%Numbers%%%
\newcommand{\D}{\mathbb{D}}
\newcommand{\R}{\mathbb{R}}
\newcommand{\N}{\mathbb{N}}
\newcommand{\F}{\mathbb{F}}
\newcommand{\Om}{\Omega }
\newcommand{\eps}{\epsilon }

%%%Operators%%%
\def\diam{\operatorname{diam}}
\def\dist{\operatorname{dist}}
\def\ess{\operatorname{ess}}
\def\inner{\operatorname{int}}
\def\osc{\operatorname{osc}}
\def\sign{\operatorname{sign}}
\def\supp{\operatorname{supp}}
%%%Spaces%%%
\newcommand{\BMO}{BMO(\Omega)}
\newcommand{\LOne}{L^{1}(\Omega)}
\newcommand{\LOnen}{(L^{1}(\Omega))^n}
\newcommand{\LTwo}{L^{2}(\Omega)}
\newcommand{\Lq}{L^{q}(\Omega)}
\newcommand{\Lp}{L^{p}(\Omega)}
\newcommand{\Lpn}{(L^{p}(\Omega))^n}
\newcommand{\Ltwon}{(L^{2}(\Omega))^n}
\newcommand{\LInf}{L^{\infty}(\Omega)}
\newcommand{\HOneO}{H^{1,0}(\Omega)}
\newcommand{\HTwoO}{H^{2,0}(\Omega)}
\newcommand{\HOne}{H^{1}(\Omega)}
\newcommand{\HTwo}{H^{2}(\Omega)}
\newcommand{\HmOne}{H^{-1}(\Omega)}
\newcommand{\HmTwo}{H^{-2}(\Omega)}

\newcommand{\LlogL}{L\log L(\Omega)}

\def\avint{\mathop{\,\rlap{-}\!\!\int}\nolimits} 

%% Functions
\newcommand{\cre}{c_{r\varepsilon}}
\newcommand{\vre}{v_{r\varepsilon}}
\newcommand{\crem}{c_{r\varepsilon m}}
\newcommand{\vrem}{v_{r\varepsilon m}}
\newcommand{\crel}{c_{r\varepsilon _m}}
\newcommand{\vrel}{v_{r\varepsilon _m}}

% %%%Theorems%%%
% \newtheorem{Theorem}{Theorem}[section]
% \newtheorem{Assumptions}[Theorem]{Assumptions}
% \newtheorem{Corollary}[Theorem]{corollary}
% \newtheorem{Convention}[Theorem]{convention}
% \newtheorem{Definition}[Theorem]{Definition}
% 
% \newtheorem{Lemma}[Theorem]{Lemma}
% \newtheorem{Notation}[Theorem]{Notation}
% \newtheorem{Remark}[Theorem]{Remark}
% \theoremstyle{definition}
% \newtheorem{Example}[Theorem]{Example}
%%%Settings%%%
\numberwithin{equation}{section}
\title{Nonlocal and local models for taxis in cell migration: a rigorous limit procedure}
% \author{Maria Krasnianski, 
% %Kevin Painter, 
% Christina Surulescu}
% \renewcommand\Affilfont{\itshape\small}
% \affil{Technische Universität Kaiserslautern, Felix-Klein-Zentrum für Mathematik\\ Paul-Ehrlich-Str. 31, 67663 Kaiserslautern, Germany\\
%   e-mail: {\{krasnian,surulescu\}@mathematik.uni-kl.de}}
%   
% 
% 
% \author{Anna~Zhigun}
% \affil{Queen's University Belfast, School of Mathematics and Physics\protect\\
% University Road, Belfast BT7 1NN, Northern Ireland, UK\protect\\
%   e-mail: {A.Zhigun@qub.ac.uk}}  
%   
% \date{}
% \maketitle

\makeatletter
\let\@fnsymbol\@alph
\makeatother

\author{
   Maria Krasnianski\thanks{Felix-Klein-Zentrum für Mathematik, Technische Universität Kaiserslautern, Paul-Ehrlich-Str. 31, 67663 Kaiserslautern, Germany, \href{mailto:krasnian@mathematik.uni-kl.de}{krasnian@mathematik.uni-kl.de}}, 
   \quad Kevin J. Painter\thanks{Department of Mathematics \& Maxwell Institute, Heriot-Watt University,
   Edinburgh, Scotland, EH14 4AS, \href{mailto:K.Painter@hw.ac.uk}{K.Painter@hw.ac.uk}},\quad Christina Surulescu\thanks{Felix-Klein-Zentrum für Mathematik, Technische Universität Kaiserslautern, Paul-Ehrlich-Str. 31, 67663 Kaiserslautern, Germany, \href{mailto:surulescu@mathematik.uni-kl.de}{surulescu@mathematik.uni-kl.de}
   },\quad and Anna Zhigun\thanks{School of Mathematics and Physics, Queen's University Belfast, University Road, Belfast BT7 1NN, Northern Ireland, UK, \href{mailto:A.Zhigun@qub.ac.uk}{A.Zhigun@qub.ac.uk}
   }
}
\date{\vspace{-5ex}}
\maketitle

\begin{abstract}
A rigorous limit procedure is presented which links nonlocal models involving adhesion or nonlocal chemotaxis to their local counterparts featuring haptotaxis and classical chemotaxis, respectively. It relies on a novel reformulation  of the involved nonlocalities in terms of  integral operators applied directly to the gradients of signal-dependent quantities.
The proposed approach handles both model types in a unified way and extends the  previous mathematical framework to settings that allow for  general solution-dependent coefficient functions. The previous forms of nonlocal operators are compared with the new ones introduced in this paper and the advantages of the latter are highlighted by concrete examples. Numerical simulations in 1D provide an illustration of some of the theoretical findings.

\end{abstract}

\section{Introduction}\label{intro}

Macroscopic equations and systems describing the evolution of populations in response to soluble and 
insoluble environmental cues have been intensively studied and the palette of such reaction-diffusion-taxis models is continuously expanding. Models of such form are motivated by problems arising in various contexts, a large part related to cell migration and proliferation connected to tumor invasion, embryonal development, wound healing, biofilm formation, insect behavior in response to chemical cues, etc. We refer, e.g. to \cite{BBTW} for a recent review also containing some deduction methods for taxis equations based on kinetic transport equations. 

\noindent
Apart from such purely local PDE systems with taxis, several spatially nonlocal models have been introduced over the last two decades and are attracting ever increasing interest. They involve integro-differential operators in one or several terms of the featured reaction-diffusion-drift equations. Their aim is to characterize interactions between individuals or signal perception  happening not only at a specific location, but over a whole set (usually a ball) containing (centered at) that location. In the context of cell populations, for instance, this seems to be a more realistic modeling assumption, as cells are able to extend various protrusions (such as lamellipodia, filopodia, cytonemes, etc.) into their surroundings, which can reach across long distances compared against cell size, see \cite{GM17,SSM17} and references therein. Moreover, the cells are able to relay signals they perceive and thus transmit them to cells with which they are not in direct contact, thereby influencing their motility, see e.g.,  \cite{eom,GaPa08}. Cell-cell and cell-tissue adhesion are essential for mutual communication, homeostasis, migration, proliferation, sorting, and many other biological processes. %, see \cite{Armstrong2006} and references therein. 
A large variety of models for adhesive behavior at the cellular level have been developed to account for the dynamics of focal contacts, e.g. \cite{Bell,BDB,ward} and to assess their influence on cytoskeleton restructuring and cell migration, e.g. \cite{DiMilla,DT93,Kuusela-Alt,Aydar}. 
Continuous, spatially nonlocal models involving adhesion were introduced more recently \cite{Armstrong2006} and are attracting increasing interest from the modeling \cite{BTCE,Butten,Carillo_etal19,DTGC14,gerisch2008,GePa10,15MuTo,PAS10}, analytical \cite{CLSW,dyson13,dyson10,SGAP09,WHP16}, and numerical \cite{Alf-num09} viewpoints. Yet more recent models \cite{DTGC17,ESuSt2016} also take into account subcellular level dynamics, thus involving further nonlocalities (besides adhesion), with respect to some structure variable referring to individual cell state. Thereby, multiscale mathematical settings are obtained, which lead to challenging problems for analysis and numerics. % of the respective PDEs.
Another essential aspect of cell migration is the directional bias in response to a diffusing signal, commonly termed chemotaxis. A model of cell migration with finite sensing radius, thus featuring nonlocal chemotaxis has been introduced in \cite{othmer-hillen2} and readdressed in \cite{HilPSchm} from the perspective of well-posedness, long time behaviour, and patterning. We also refer to \cite{loy-preziosi} for further spatially nonlocal models and their formal deduction.

\noindent
For adhesion and nonlocal chemotaxis models, a gradient of some nondiffusing or diffusing signal is replaced by a nonlocal integral term. Here we are only interested in this type of model, and refer to \cite{CPSZ,eftimie,kav-suzuki} for reviews on settings involving other types of nonlocality. Specifically, following \cite{Armstrong2006,gerisch2008,HilPSchm,othmer-hillen2}, we consider the subsequent systems, whose precise mathematical formulations will be specified further below:
\begin{enumerate}
 \item a prototypical nonlocal model for adhesion 
 \begin{subequations}\label{nlHapto}
 \begin{align}
  &\partial_t c_{r}=\nabla\cdot\left(D_c(c_{r},v_{r})\nabla c_{r}-c_{r}\chi(c_{r},v_{r}) {\cal A}_{r}(g(c_r,v_r))\right)+f_c(c_{r},v_{r}),\label{nlHaptoc}\\
  &\partial_t v_{r}=f_v(c_{r},v_{r}),\label{nlHaptov}
 \end{align}
 \end{subequations}
 where 
 \begin{align}
  {\cal A}_{r}u(x):=\frac{1}{r} \avint_{B_r}u(x+\xi)\frac{\xi}{|\xi|}F_r(|\xi|)\,d\xi\label{DefAr}
 \end{align}
 is referred to as the adhesion velocity, and the function $F_r$ describes how the magnitude of the interaction force depends on the interaction range $|\xi|$ within the sensing radius $r$. We require this function to satisfy 
 \begin{Assumptions}[Assumptions on $F_r$]\label{Assf}~
 
 \begin{enumerate}[{\it (i)}]
  \item  $(r,\rho)\mapsto F_r(\rho)$ is continuous and positive in  $\left[0,r_0\right]^2$ for some $r_0>0$;
  \item%\label{F00} 
  $F_0(0)={ n+1}$.
  \footnote{In \cref{AverOper} we will see that this is, indeed, the 'right'  normalisation. If we assume, as in \cite{Armstrong2006}, that this function is a constant involving some viscosity related proportionality,  
  		%and a factor  reflecting the strength of adhesivity between the cells and/or between cells and signal, 
  		then this choice provides the value of that constant.}
 \end{enumerate}
\end{Assumptions}
The quantity 
\begin{equation*}
\F(c_r,v_r)=c_r\chi (c_r,v_r){\cal A}_{r}(g(c_r,v_r))
\end{equation*}
is often referred to as the total adhesion flux, possibly scaled by some constant involving the typical cell size or the sensing radius, see e.g.,  \cite{Armstrong2006,Butten}. Here we also include a coefficient $\chi (c_r,v_r)$ that depends on cell and tissue (extracellular matrix, ECM) densities, 
which can be seen as characterizing the sensitivity of cells towards their neighbours and the surrounding tissue. It will, moreover, help provide in a rather general framework a unified presentation of 
this and the subsequent local and nonlocal model classes for adhesion, haptotactic, and chemotactic behavior of moving cells.

\noindent
System \cref{nlHapto} is a simplification of the integro-differential system (4) in \cite{gerisch2008}.
The main difference between the two settings is that in our case we ignore the so-called matrix-degrading enzymes (MDEs). Instead, we assume cells directly  degrade the tissue directly: this fairly standard
simplification (e.g., \cite{PAS10}) effectively assumes that proteolytic enzymes remain localised to the cells, and helps simplify the analysis. On the other hand, \cref{nlHapto} can also be viewed as 
a nonlocal version of the haptotaxis model with nonlinear diffusion: 
 \begin{subequations}\label{Hapto}
 \begin{align}
  &\partial_t c=\nabla\cdot\left(D_c(c,v)\nabla c-c\chi(c,v) \nabla g(c,v)\right)+f_c(c,v),\label{Haptoc}\\
  &\partial_t v=f_v(c,v);\label{Haptov}
 \end{align}
 \end{subequations}
 \item a  prototypical  nonlocal chemotaxis-growth model
 \begin{subequations}\label{nlChemo}
 \begin{align}
  \partial_t c_r=&\nabla\cdot\left(D_c(c_{r},v_{r})\nabla c_{r}-c_{r}\chi(c_{r},v_{r}) \mathring{\nabla}_{ r}v_{r}\right)+f_c(c_{r},v_{r}),\label{nlChemoc}\\
  \partial_t v_r=&D_v\Delta v_r+f_v(c_{r},v_{r})\label{nlChemov}
 \end{align}
\end{subequations}
with the nonlocal gradient  
\begin{align}
 \mathring{\nabla}_{r}u(x):=\frac{n}{r} \avint_{S_r}u(x+r\xi)\xi\,d\xi.\nonumber%\label{eq:nonlocal-grad}
\end{align}
System \cref{nlChemo} can be seen as  a nonlocal version of the chemotaxis-growth model
\begin{subequations}\label{Chemo}
 \begin{align}
  \partial_t c=&\nabla\cdot\left(D_c(c,v)\nabla c-c\chi(c,v) \nabla v\right)+f_c(c,v),\label{Chemoc}\\
  \partial_t v=&D_v\Delta v+f_v(c,v),\label{Chemov}
 \end{align}
\end{subequations}
where $\chi (c,v)$ is the chemotactic sensitivity function. As mentioned above, in order to have a unified description of our systems \cref{Hapto} and \cref{Chemo} and of their respective nonlocal counterparts \cref{nlHapto} and \cref{nlChemo}, we later introduce a more general version of the nonlocal chemotaxis flux, similar to the above adhesion velocity $\mathcal A_r$.  
\end{enumerate}
Here and below $B_r$ and $S_r$ denote the open $r$-ball and the $r$-sphere in $\R^n$, both centred at the origin, and  
\begin{align}
&\avint_{B_r}u(\xi)\,d\xi:=\frac{1}{|B_r|}\int_{B_r}u(\xi)\,d\xi,\nonumber\\
&\avint_{S_r}u(\xi)\,d\xi:=\frac{1}{|S_r|}\int_{S_r}u(\xi)\,d_{S_r}(\xi)\nonumber
\end{align}
are the usual mean values of a function $u$ over $B_r$ and $S_r$, respectively. The nonlocal systems \cref{Hapto} and \cref{Chemo} are stated for $$t>0,\qquad x\in\Omega\subset\R^n.$$ 
Unless the spatial domain $\Omega$ is the whole $\R^n${,}  suitable boundary conditions are required. In the latter case, usually periodicity is assumed, which is not biologically realistic in general. Still, this offers the easiest way to  properly define the output of the nonlocal operator in the boundary layer where the sensing region is not fully contained in $\Omega$. Very recently various other boundary conditions have been derived and compared in the context of a single equation modeling cell-cell adhesion in 1D \cite{HillButten2019}. 

\noindent
Few previous works focus on solvability for models with nonlocality in a taxis term. 
%  They mostly cover well-posedness.
 Some of them deal with single equations that only involve cell-cell adhesion 
%  and no interactions with any soluble or { insoluble} signals in 
 \cite{dyson10,dyson13,HillButten2019}, others study nonlocal systems of the sort considered here 
 for two  \cite{HilPSchm} or  more components \cite{ESuSt2016}. The global solvability  and boundedness study in \cite{WHP16} is obtained for the case of a nonlocal operator with integration over a set of sampling directions being an open, not necessarily strict subset of $\R^{n}$. The systems studied there include settings with a third equation for the dynamics of diffusing MDEs. Conditions which secure uniform boundedness of solutions to such cell-cell and cell-tissue adhesion models in 1D were elaborated in \cite{SGAP09}.

\noindent
Some heuristic analysis via local Taylor expansions 
was performed in \cite{gerisch2008} and \cite{Hillen2007} showing that as $r\rightarrow0$ the outputs ${\cal A}_{r}u$ and $\mathring{\nabla}_{r}u$, respectively, converge pointwise to $\nabla u$ for a fixed and sufficiently smooth $u$. In \cite{HilPSchm} it was observed that it would be interesting to study rigorously the limiting behaviour of solutions of the nonlocal problems involving $\mathring{\nabla}_{r}u$. The authors ask in which sense, if at all, do these solutions converge to solutions of the corresponding local problem  as
$r\rightarrow0$.  Numerical results appeared to confirm that, in certain cases, the answer is positive. Still, to the best of our knowledge, no rigorous analytical study of this issue has as yet been performed. Clearly, any approach based on representations using Taylor polynomials requires a rather high order regularity  of solution components and a suitable control on the approximation errors, and that uniformly in $r$. This is difficult or even impossible  to obtain in most cases, particularly when dealing with weak solutions. In this work we propose a different approach based on the representation of the input $u$ in terms of an integral of $\nabla u$ over line segments. This leads to a new description of the nonlocal operators ${\cal A}_{r}$ and $\mathring{\nabla}_{r}$  in terms of nonlocal operators applied to gradients (see  \cref{AverOper} below). Moreover, it turns out that redefining their outputs inside the vanishing boundary layer in a suitable way allows one to perform a rigorous proof of convergence: Under suitable assumptions on the system coefficients and other parameters, appropriately defined sequences of solutions to nonlocal problems involving the mentioned modified nonlocal operators converge for $r\to 0$ to those of the corresponding local models  \cref{Hapto} and \cref{Chemo}, respectively. Our convergence proof is based on estimates on $c_r$ and $v_r$ which are uniform in $r$ and on a compactness argument.
% , and the Lebesgue differentiation theorem. 
The two models \cref{nlHapto} and \cref{nlChemo} are chosen as illustrations, however our idea can be further applied to other integro-differential systems with similar properties. 

\noindent
The rest of the paper is organised as follows.  \cref{not} introduces some basic 
notations to be used throughout this paper. In  \cref{AverOper} we introduce the aforementioned  adaptations of the nonlocal operators ${\cal A}_r$ and $\mathring{\nabla}_{r}$ and study their limiting properties as $r$ becomes infinitesimally small.  This turns out to be useful for our convergence proof later.  We also establish  in  \cref{modelEq} the well-posedness   
for a certain class of equations including such operators.  In the subsequent \textit{\cref{model}}  we introduce a couple of nonlocal models that involve the previously considered averaging operators, prove the global existence of solutions of the respective systems, and investigate their limit behaviour as $r\to 0$. 
\textit{\cref{sec:numerics}} provides some numerical simulations comparing various nonlocal and local models considered in this work in the 1D case. 
Finally, \textit{\cref{sec:discussions}} contains a discussion of the results and a short outlook on open issues.
% After that we turn to specific applications, systems \cref{nlHapto} and \cref{nlChemo}, in \cref{model}.  We first addresses some  modelling issues and sets the models  into context\dots   \cref{WellP} recalls the techniques to establish well-posedness for nonlocal models.  In  \cref{apriori} we  establish uniform in $r$ estimates for solutions $(c_r,v_r)$ for \cref{nlHapto}/\cref{nlChemo} under suitable assumptions on the system coefficients. Finally, in \cref{Limit} we show how these estimates and the properties of the averaging operators can be used in order to obtain convergence of solutions of  \cref{nlHapto} and \cref{nlChemo} to solutions of \cref{Hapto} and \cref{Chemo}, respectively. 

\section{Basic notations and function spaces}\label{not}
We denote the Lebesgue measure of a set $A$ by $|A|$. Let $\Omega \subset \R^n$ be a bounded domain with smooth enough boundary.

\noindent
For a function $w:\Omega\rightarrow\R^n$ we assume, by convention, that 
\begin{align}
w:=0\qquad \text{in }\R^n\backslash\overline{\Omega}.\nonumber%\label{wout}  
\end{align}

\noindent
For  $r>0$ we introduce the following subdomain of $\Omega$ $$\Omega_r:=\{x\in\Omega\ :\  \dist(x,\partial\Omega)>r\}.$$

\noindent
Partial derivatives, in both classical and distributional sense, with respect to variables  $t$ and $x_i$, will be denoted respectively by $\partial_t$ 
and $\partial_{x_i}$. Further, $\nabla$, $\nabla\cdot$ and $\Delta$ stand for the spatial gradient, divergence and Laplace operators, respectively. $\partial_{\nu}$ is 
the derivative with respect to the outward unit normal of $\partial\Omega$. 

\noindent
We assume the reader to be familiar with the definitions and the usual properties of such spaces as: the standard Lebesgue and Sobolev spaces, spaces of functions with values in these spaces, and with anisotropic Sobolev spaces. In particular, we denote by $C_w([0,T];L^2(\Omega))$ the space of functions $u:[0,T]\rightarrow L^2(\Omega)$ which are continuous w.r.t. the weak topology of $L^2(\Omega)$.\\
Throughout the paper $\left<\cdot,\cdot\right>_{{ X^*,X}}$ denotes a duality paring between a space { $X$} and its dual { $X^*$.}
% The space may vary depending on the context and is mostly not written explicitly.

\noindent
Finally, we make the following useful convention: For all indices $i$, the quantity $C_i$ denotes a positive constant or, alternatively, a positive function of its arguments. Moreover, unless explicitly stated, these constants  {\bf do not} depend upon $r$.

\section{Operators  \texorpdfstring{${\cal A}_r$}{} and \texorpdfstring{$\mathring{\nabla}_{r}$}{} and averages of \texorpdfstring{$\nabla$}{}}\label{AverOper}
In this section we study the applications of the non-local operators ${\cal A}_r$ and $\mathring{\nabla}_{r}$ to  fixed, i.e. independent of $r$, functions $u$. Our focus is on the limiting behaviour as $r\rightarrow0$. Formal Taylor expansions performed in  \cite{gerisch2008,HilPSchm} anticipate that the limit is  the gradient operator  in both cases. This we prove here rigorously under rather mild regularity assumptions on $u$. To be more precise, we replace  ${\cal A}_r$ and $\mathring{\nabla}_{r}$  by certain  integral operators ${\cal T}_r$ and ${\cal S}_r$ (see \cref{IrTr} and \cref{JrSr} below) applied to $\nabla u$ and show that these operators are pointwise approximations of the identity operator in the $L^p$ spaces.   
 
\noindent
We start with the operator ${\cal A}_r$. For $r\in(0,r_0]$,  %{\cancel{$F_r\in C(\overline{\Omega})$}}, 
$u\in C^1(\Omega)$, and $x\in \Omega_r$ we compute that
\begin{align}
 {\cal A}_ru(x)=&\frac{1}{r} \avint_{B_r}u(x+\xi)\frac{\xi}{|\xi|}F_r(|\xi|)\,d\xi\nonumber\\
 =&\frac{1}{r}\avint_{B_r}(u(x+\xi)-u(x))\frac{\xi}{|\xi|}F_r(|\xi|)\,d\xi\nonumber\\
 =&\frac{1}{r}\avint_{B_r}\int_0^1 (\nabla u(x+s\xi)\cdot\xi)\,ds\,\frac{\xi}{|\xi|}F_r(|\xi|)\,d\xi\nonumber\\
 =&\frac{1}{r}\int_0^1\avint_{B_r} (\nabla u(x+s\xi)\cdot\xi)\frac{\xi}{|\xi|}F_r(|\xi|)\,d\xi\,ds\nonumber\\
 =&\int_0^1\avint_{B_{1}} (\nabla u(x+rsy)\cdot y)\frac{y}{|y|}F_r(r|y|)\,dyds.\label{Ar}
\end{align}
 Formula \cref{Ar} extends to arbitrary  $u\in W^{1,1}(\Omega)$ by means of a  density argument. Motivated by \cref{Ar} we introduce the averaging operator
\begin{align}
{\cal T}_r w(x):=&\int_0^1\avint_{B_{1}} (w(x+rsy)\cdot y)\frac{y}{|y|}F_r(r|y|)\,dyds.\label{IrTr}
\end{align}
In {\cref{OperProp}} we check  that ${\cal T}_rw(x)$ is  well-defined for all $w\in (L^1(\Omega))^n$ and a.a. $x\in\Omega$. 
In this notation, for all $r\in(0,r_0]$ and $u\in W^{1,1}(\Omega)$ identity \cref{Ar} takes the  form
\begin{align}
 {\cal A}_ru={\cal T}_r (\nabla u)\qquad\text{a.e. in }\Omega_r.\nonumber
\end{align}
In the limiting case $r=0$ we have that 
\begin{align}
  {\cal T}_0 w(x)=&\int_0^1\avint_{B_{1}} (w(x)\cdot y)\frac{y}{|y|}F_0(0)\,dyds,\nonumber\\
 =&F_0(0)\sum_{i,j=1}^n w_i(x)e_j\avint_{B_{1}} \frac{y_iy_j}{|y|}\,dy\nonumber\\
 =&F_0(0)\sum_{i,j=1}^n w_i(x)e_j\delta_{ij}\avint_{B_{1}} \frac{y_i^2}{|y|}\,dy\nonumber\\
 =&F_0(0)\sum_{i=1}^n w_i(x)e_i\avint_{B_{1}} \frac{y_i^2}{|y|}\,dy\nonumber\\
 =&F_0(0)\sum_{i=1}^n w_i(x)e_i\frac{1}{n}\avint_{B_{1}} \frac{|y|^2}{|y|}\,dy\nonumber\\
  =&F_0(0)\frac{1}{n}\avint_{B_{1}} |y|\,dy\, w(x)\nonumber\\
  =&w(x).\label{I0w}
\end{align}
In the final step we used {\cref{Assf}{\it (ii)} which says that $F_0(0)={ n+1}$ (this explains our choice) and the 
 trivial identity
\begin{align} \label{intby}
 \avint_{B_{1}} |y|\,dy=\frac{ n}{n+1}.
\end{align}
Thus, we have just proved the following lemma:
\begin{Lemma}[Adhesion velocity vs. ${\cal T}_r$]\label{LemAdVel}
 Let $u\in W^{1,1}(\Omega)$. Then it holds that
 \begin{align}
 {\cal A}_ru={\cal T}_r (\nabla u)\qquad\text{a.e. in }\Omega_r\qquad\text{for  }r\in(0,r_0].\label{ArTr}
\end{align}
Moreover, if $F_0(0)=n+1$, then 
\begin{align}
 \nabla u={\cal T}_0 (\nabla u)\qquad\text{in }\Omega.\label{A0T0}
\end{align}
\end{Lemma}

\noindent
In a very similar manner one can establish a representation for  $\mathring{\nabla}_{r}$. For this purpose we define the averaging operator
\begin{align}
{\cal S}_r w(x):=&n\int_0^1\avint_{S_{1}} (w(x+rsy)\cdot y)y\,d_{S_1}(y)ds\qquad\text{for }r\in(0,r_0].\label{JrSr}
\end{align}
The corresponding result then reads:
\begin{Lemma}[Non-local gradient vs. ${\cal S}_r$]\label{LemNGR}
  Let $u\in W^{1,1}(\Omega)$. Then  it holds that 
 \begin{align}
 \mathring{\nabla}_{r} u=&{\cal S}_r (\nabla u)\qquad\text{a.e. in }\Omega_r\qquad\text{for  }r\in(0,r_0],\label{DrSr}\\
 \nabla u=&{\cal S}_0 (\nabla u)\qquad\text{a.e. in }\Omega.\label{D0S0}
\end{align}
% where ${\cal S}_r$ is the averaging operator as  defined by \cref{JrSr}. 
\end{Lemma}

\noindent
The proof of {\cref{LemNGR}} is very similar to that of {\cref{LemAdVel}} and we omit it here.  

\noindent
Next, we observe  that  identity \cref{ArTr} was established for $\Omega_r$.  In the  boundary 
layer $\Omega\backslash\Omega_{r}$ the definition \cref{DefAr} of the adhesion velocity allows various extensions. For example, one could keep \cref{DefAr} by assuming (as done, e.g., in  \cite{ESuSt2016}) that $u:=0$ in $\R^n\backslash\Omega$. An alternative would be to average over the part of the $r$-ball  that lies inside the domain. Let us have a closer look at the first option (the second can be handled similarly). Consider the following example:
\begin{Example}[${\cal A}_r $ vs. ${\cal T}_r(\nabla \cdot)$ in 1D]\label{BndL}
 Let $\Omega=(-1,1)$, $r_0=1$,  $F_r\equiv 2$, and $u\equiv 1$. In this case, $u'\equiv 0$, hence $${\cal T}_{r}(u')\equiv0 \equiv u'.$$ For ${\cal A}_r$ one readily computes by assuming  $u=0$ in $\R\backslash(-1,1)$ that  for $x\in(-1,1)$
 \begin{align}
  {\cal A}_ru(x)=&\frac{2}{r}\frac{1}{2r}\int_{(-1-x,1-x)\cap(-r,r)}\sign(\xi)\,d\xi\nonumber\\
  =&
  \begin{cases}
    \frac{1}{r^2}(-1+r-x)&\text{ in }[-1,-1+r],\\
     0&\text{ in }(-1+r,1-r)=\Omega_r,\\
      \frac{1}{r^2}(1-r-x)&\text{ in }[1-r, 1],
 \end{cases}
\nonumber
 \end{align}
so that
\begin{align}
 \|{\cal A}_ru\|_{L^1(-1,1)}
=
 \|{\cal A}_ru\|_{L^1(\Omega\backslash\Omega_r)}=&\frac{1}{r^2}\int_{-1}^{-1+r}\left|-1+r-x\right|\,dx+\frac{1}{r^2}\int_{1-r}^1\left|1-r-x\right|\,dx=1,\nonumber
\end{align}
although
\begin{align}
 |\Omega\backslash\Omega_r|=2r\underset{r\rightarrow0}{\rightarrow}0.\nonumber
\end{align}
Thus, $${\cal A}_ru\underset{r\rightarrow0}{\rightarrow}0\equiv u'$$ in the measure but not in $L^1(\Omega)$. 
  
\end{Example}

\noindent
{\cref{BndL}} supports our idea to average $\nabla u$ instead of $u$ itself. The same applies to $\mathring{\nabla}_{r} u$ vs. ${\cal S}_r (\nabla u)$. 
% { Meinst Du hier etwa \cref{JrSr}?}

\noindent
Averaging w.r.t. $y\in B_1$ and then also w.r.t. $s\in(0,1)$ might appear superfluous in the definition of  the operator ${\cal T}_r$. The following example compares the effect of ${\cal T}_r$ with that of an operator which averages w.r.t. to $y$ only. 
\begin{Example}\label{Ex2}
Let $\Omega=\R^n$, $n\geq2$, and $r>0$, $F_r\equiv n+1$.  In this case
\begin{align}
{\cal T}_r w(x):=(n+1)\int_0^1\avint_{B_{1}} (w(x+rsy)\cdot y)\frac{y}{|y|}\,dyds.\nonumber%\label{T1}
\end{align}
 Consider also the operator 
 \begin{align}
{\widetilde{{\cal T}}}_r w(x):=(n+1)\avint_{B_{1}} (w(x+ry)\cdot y)\frac{y}{|y|}\,dy.\nonumber%\label{tT1}
\end{align}
It is easy to see that  both operators are well-defined, linear,  continuous,  and  self-adjoint in the space $L^2(\R^n)$. Moreover, they map the dense subspace $C_0(\R^n;\R^n)$ into itself. This suggests the following natural extension to $(C_0(\R^n;\R^n))^*$:
\begin{align}
 \left<{\cal T}_r\mu,\varphi\right>_{(C_0(\R^n;\R^n))^*,C_0(\R^n;\R^n)}:=&\left<\mu,{\cal T}_r\varphi\right>_{(C_0(\R^n;\R^n))^*,C_0(\R^n;\R^n)},\nonumber\\
 \left<\widetilde{{\cal T}}_r\mu,\varphi\right>_{(C_0(\R^n;\R^n))^*,C_0(\R^n;\R^n)}:=&\left<\mu,\widetilde{{\cal T}}_r\varphi\right>_{(C_0(\R^n;\R^n))^*,C_0(\R^n;\R^n)}.\nonumber
\end{align}
Let, for instance, $$w:=\delta_0e_1,$$ $\delta_0$ and $e_1$ mean the usual Dirac delta and the vector $(1,0,\dots,0)$, respectively. One readily computes that
\begin{align}
 {\widetilde{{\cal T}}}_r(\delta_0 e_1) (x)=\frac{n+1}{|B_{r}|}\chi_{B_{r}}(x)\frac{x_1}{r}\frac{x}{|x|},\nonumber
\end{align}
whereas 
\begin{align}
 {\cal T}_r(\delta_0 e_1)(x) =&\frac{n+1}{|B_{r}|}\int_0^1s^{-n-1}\chi_{B_{rs}}(x)\,ds\frac{x_1}{r}\frac{x}{|x|}\nonumber\\
 =&\frac{n+1}{n|B_{r}|}\left(\left(\frac{r}{|x|}\right)^n-1\right)_+\frac{x_1}{r}\frac{x}{|x|}.\nonumber
\end{align}
For $n\geq2$, the operator ${\cal T}_r$ retains the  singularity at the origin, however making it less concentrated, while $\widetilde{{\cal T}}_r$ eliminates that singularity entirely and produces instead jump discontinuities all over $S_r$.

\end{Example}

 \subsection{Properties of the averaging operators \texorpdfstring{${\cal T}_r$ and ${\cal S}_r$}{}}\label{OperProp}
\noindent 
In this section we collect some  properties of the averaging operators ${\cal T}_r$ and ${\cal S}_r$.
 \begin{Lemma}[Properties of ${\cal T}_r$]\label{TWellDef}
 Let $F_r$ satisfy  {\cref{Assf}} and let $r\in(0,r_0]$. Then:
  \begin{enumerate}[{\it (i)}]
   \item\label{PropTi} ${\cal T}_r$ is a well-defined continuous  linear operator in  $\Lpn$ for all $p\in[1,\infty]$. The corresponding operator 
   norm satisfies
  \begin{align}
  \|{\cal T}_r\|_{L(\Lpn)}\leq \Cr{CP2}(r,p),\label{Tbnd}
 \end{align}
   where 
 \begin{align}
  \Cl{CP2}(r,p):=
  \begin{cases}
\left({ n}\int \limits _0^1\rho^{n-1+p^*}(F_r(r{ \rho }))^{p^*}\,d\rho\right)^{\frac{1}{p^*}}&\text{for }p\in(1,\infty],\ \ \frac{1}{p}+\frac{1}{p^*}=1,\\
\underset{\rho\in [0,1]}{\max}\,\rho F_r(r\rho)&\text{for }p=1.
\end{cases}\nonumber%\label{TbndC}
 \end{align}
   \item\label{PropTii} Let $p,p^*\in[1,\infty]$ be such that $\frac{1}{p}+\frac{1}{p^*}=1$. For all $w_1\in \left(\Lp\right)^n$ and 
   $w_2\in \left(L^{p^*}(\Omega)\right)^n$ it holds: 
\begin{align}
\int_{\Omega}({\cal T}_rw_1(x)\cdot w_2(x))\,dx=\int_{\Omega}(w_1(x)\cdot{\cal T}_rw_2(x))\,dx.\label{sadj}
\end{align}
\item\label{PropTiii} Let $p\in[1,\infty)$. For all $w\in \Lpn$ it holds that
 \begin{align}
  {\cal T}_r w\underset{r\rightarrow0}{\rightarrow}{\cal T}_0w=w\qquad \text{in }\Lpn.\label{convT}
 \end{align}
\item\label{PropTiv} For $p =2$ it holds that
\begin{align}
\| {\cal T}_r\|_{L((L^2(\Omega))^n)} %
%\frac{1}{n} \avint_{B_{1}} |y|F_r(r|y|) \, dy 
\underset{r\rightarrow0}{\rightarrow} 1.
\end{align}

  \end{enumerate}
 \end{Lemma}
 \begin{Remark}
 Due to the assumptions on $F_r$ we have in the limit that
  \begin{align}
  \Cr{CP2}(r,p)\underset{r\rightarrow0}{\rightarrow}\Cl{CP2_}(p):=  \begin{cases}
{ (n+1)\left(\frac{n}{n+p^*}\right)^{\frac{1}{p^*}}}&\text{for }p\in(1,\infty]\backslash\{2\},\ \ \frac{1}{p}+\frac{1}{p^*}=1,\\
{ n+1}&\text{for }p=1.
\end{cases}\label{C2p}
 \end{align}
 \end{Remark}
\begin{proof}{(of {\cref{TWellDef}})}
\begin{enumerate}[{\it (i)}]
\item Since $w$ is measurable and $\rho\mapsto F_r(\rho)$, $(x,s,y)\mapsto x+rsy$, $(y,z)\mapsto(z\cdot y)\frac{y}{|y|}$ are 
continuous, we have that
\begin{align}
 (x,y,s)\mapsto (w(x+rsy)\cdot y)\frac{y}{|y|}F_r(r|y|)\nonumber
\end{align}
is well-defined a.e. in $\Omega\times B_1\times(0,1)$ and is measurable.  Let $p\in{(1},\infty)$ and $\frac{1}{p}+\frac{1}{p^*}=1$. 
 Using  H\"older's inequality, Fubini's theorem, and our convention that $w$ vanishes outside $\Omega$, we deduce for all $w\in \Lpn$ that
 \begin{align}
  \|{\cal T}_rw\|_{\Lpn}^p=&\int_{\Omega}\left|\int_0^1\avint_{B_{1}} (w(x+rsy)\cdot y)\frac{y}{|y|}F_r(r|y|)\,dyds\right|^p\,dx\nonumber\\
  \leq &\int_{\Omega}\int_0^1\avint_{B_{1}}|w(x+rsy)|^p\,dy\left(\avint_{B_{1}}\left(|y|F_r(r|y|)\right)^{p^*}\,dy\right)^{\frac{p}{p^*}}dsdx\nonumber\\
  =&\Cr{CP2}^p(r,p)\int_0^1\avint_{B_{1}}\int_{\Omega}|w(x+rsy)|^p\,dxdyds,\nonumber\\
  \leq&\Cr{CP2}^p(r,p)\int_0^1\avint_{B_{1}}\int_{\Omega}|w(z)|^p\,dzdyds\nonumber\\
  =&\Cr{CP2}^p(r,p)\|w\|_{\Lpn}^p.\nonumber%\label{est6}
 \end{align}
 This implies that for all $p\in(1,\infty)$ operator ${\cal T}_r$ is well-defined in $\Lpn$ and satisfies  \cref{Tbnd}. It is also clearly linear. Taken together we then have that ${\cal T}_r\in L(\Lpn)$ and \cref{Tbnd} holds. 
 The cases $p=1$ and $p=\infty$ can be treated similarly. 
  
\item Let $w_1\in \left(\Lp\right)^n$ and $w_2\in \left(L^{p^*}(\Omega)\right)^n$. We compute by using  Fubini's theorem, the symmetry of $B_1$, and simple variable transformations that
 \begin{align}
 &\int_{\Omega}({\cal T}_rw_1(x)\cdot w_2(x))\,dx\nonumber\\=&\int_{\Omega}\int_0^1\avint_{B_{1}} (w_1(x+rsy)\cdot y)\frac{y}{|y|}F_r(r|y|)\,dyds\cdot w_2(x)\,dx\nonumber\\
 =&\int_0^1\avint_{B_{1}}|y|F_r(r|y|)\int_{\Omega} \left(w_1(x+rsy)\cdot\frac{y}{|y|}\right)\left(w_2(x)\cdot\frac{y}{|y|}\right)\,dx\,dy\,ds\nonumber\\
 =&\int_0^1\avint_{B_{1}}|y|F_r(r|y|)\int_{\Omega\cap(-rsy+\Omega)} \left(w_1(x+rsy)\cdot\frac{y}{|y|}\right)\left(w_2(x)\cdot\frac{y}{|y|}\right)\,dx\,dy\,ds\label{sa1}\\
  =&\int_0^1\avint_{B_{1}}|y|F_r(r|y|)\int_{(rsy+\Omega)\cap\Omega} \left(w_1(z)\cdot\frac{y}{|y|}\right)\left(w_2(z-rsy)\cdot\frac{y}{|y|}\right)\,dzdyds\nonumber\\
  =&\int_0^1\avint_{B_{1}}|y|F_r(r|y|)\int_{(-rsy+\Omega)\cap\Omega} \left(w_1(z)\cdot\frac{y}{|y|}\right)\left(w_2(z+rsy)\cdot\frac{y}{|y|}\right)\,dzdyds.\label{sa2}
\end{align}
Thereby we used our convention that each function defined in $\Omega$ is assumed to be prolonged by zero outside $\Omega$. 
Comparing \cref{sa1} and \cref{sa2} we obtain \cref{sadj}.
\item We apply the  Banach-Steinhaus theorem. Due to {\it\cref{PropTi}} and \cref{C2p},  $\{{\cal T}_r\}_{r\in(0,r_0]}$ is a family of  uniformly bounded linear operators in the Banach space $\Lpn$. Thus, as $C_c(\overline{\Omega};\R^n)$ is dense in $\Lpn$ for $p<\infty$, we only need to check \cref{convT} for $w\in C_c(\overline{\Omega};\R^n)$. But for such $w$ we can directly pass to the limit under the integral and thus obtain using \cref{I0w} { and the dominated convergence theorem} that
 \begin{align} 
  {\cal T}_r w\underset{r\rightarrow0}{\rightarrow} {\cal T}_0w=w\qquad\text{for all }x\in\Omega{ \text{ and in }\Lpn}.\nonumber
 \end{align}
\item Here we make use of the  Fourier transform, which we denote by the hat symbol. A straightforward calculation shows that
\begin{align}
{\widehat{{\cal T}_r w}}= \Phi_r \widehat{w},\nonumber                                                                \end{align}
where %$\Phi_r$ is given by 
\begin{align}
\Phi_r (\xi) := \int_0^1 \avint_{B_1} \frac{yy^T}{|y|} F_r(r|y|)e^{irsy \cdot \xi} \, dy ds.\label{FT}
\end{align}
Combining \cref{FT} with the Plancherel theorem and using our convention that $w$ vanishes outside $\Omega$, we can estimate as follows:
\begin{align}
 \|{\cal T}_r\|_{L((L^2(\Omega))^n)}
 =& \sup_{\|w\|_{(L^2(\Omega))^n} = 1} \|{\cal T}_r w\|_{(L^2(\Omega))^n}\nonumber\\ 
 \leq& \sup_{\|w\|_{(L^2(\Omega))^n} = 1} \|\widehat{{\cal T}_r w}\|_{(L^2(\R^n))^n}\nonumber\\
 \leq& \| {|}\Phi_r{|}_2 \|_{L^{\infty}(\R^n)}\sup_{\|w\|_{(L^2(\Omega))^n} = 1} \|\widehat{w}\|_{(L^2(\R^n))^n}\nonumber\\
 =& \| {|}\Phi_r{|}_2 \|_{L^{\infty}(\R^n)}\sup_{\|w\|_{(L^2(\Omega))^n} = 1} \|w\|_{(L^2(\Omega))^n}\nonumber\\
 = &\| {|}\Phi_r{|}_2 \|_{L^{\infty}(\R^n)}.\label{estTr2}
\end{align}
% Set $\Omega^r := \{x\in \R^n: \text{dist}(x,\Omega) < r\}$. Using this and due to $\text{supp}({\cal T}_r w) \subseteq \Omega^r$ for $w \in (L^2(\Omega))^n$ ...
Here $|M|_2$ denotes
the spectral norm of a matrix $M\in\R^{n\times n}$.
Further, observe that
\begin{align}
\Phi_r(O \xi) = O \Phi_r (\xi) O^T \qquad\text{ for all orthogonal }O \in \R^{n\times n} \text{ and } \xi \in \R^n.                                                                                                               \end{align}
 Consequently, denoting by $e_1$ the first canonical vector of $\R^n$  and appropriately constructing an ortho\-go\-nal matrix $O$ in order for $O\xi =|\xi|e_1$ to hold, we obtain that
\begin{align} \label{frub}
{|} \Phi_r(\xi) {|}_2 = {|}\Phi_r(|\xi|e_1){|}_2\qquad\text{for all }\xi\in\R^n.% \leq \frac{1}{n} \avint_{B_{1}} |y|F_r(r|y|) \, dy,
\end{align}
Since 
\begin{align}
  \Phi_r(|\xi|e_1)=&\int_0^1 \avint_{B_1} \frac{yy^T}{|y|} F_r(r|y|)e^{irs|\xi|y_1} \, dy ds
%   \nonumber\\
%   =&\int_0^1 \avint_{B_1} \frac{yy^T}{|y|} F_r(r|y|)\cos(rs|\xi|y_1) \, dy
  \label{Phi1}
\end{align}
is a  diagonal matrix, its spectral norm is given by the spectral radius. Estimating the right-hand side of \cref{Phi1} we then conclude that
\begin{align}
 {|}\Phi_r(|\xi|e_1){|}_2\leq {\frac{1}{n}} \avint_{B_1}|y| F_r(r|y|)\, dy\underset{r\rightarrow0}{\rightarrow} 1\qquad\text{for all }\xi\in\R^n\label{ctrub}
\end{align}
due to $F_0(0) = n+1$ and \eqref{intby}.
Combining \cref{estTr2,ctrub,frub} 
% $$\|{\cal T}_r\|_{L((L^2(\Omega))^n)} \geq  \frac{1}{n} \avint_{B_{1}} |y|F_r(r|y|) \, dy.$$
we arrive at
\begin{align}
 \underset{r\rightarrow0}{\lim\sup}\, \|{\cal T}_r\|_{L((L^2(\Omega))^n)}\leq 1.
\end{align}
Finally, the pointwise convergence \cref{convT} and  the Banach-Steinhaus theorem imply that
\begin{align}
 \underset{r\rightarrow0}{\lim\inf}\, \|{\cal T}_r\|_{L((L^2(\Omega))^n)}\geq 1,\nonumber
\end{align}
concluding the proof.

\end{enumerate}
\end{proof}

\noindent
A similar result holds for ${\cal S}_r$:
\begin{Lemma}[Operator ${\cal S}_r$] \label{SWellDef}
 Let $r\in[0,r_0]$. Then:
  \begin{enumerate}[{\it (i)}]
   \item\label{PropSi} ${\cal S}_r$  is a well-defined continuous  linear operator in  $\Lpn$ for all $p\in[1,\infty]$. The corresponding operator 
   norm satisfies
  \begin{align}
  \|{\cal S}_r\|_{L(\Lpn)}\leq n.\label{Snorm}
 \end{align}
   \item\label{PropSii} Let $p,p^*\in[1,\infty]$ be such that $\frac{1}{p}+\frac{1}{p^*}=1$. For all $w_1\in \left(\Lp\right)^n$ and 
   $w_2\in \left(L^{p^*}(\Omega)\right)^n$ it holds:  
\begin{align}
 \int_{\Omega}({\cal S}_rw_1(x)\cdot w_2(x))\,dx=\int_{\Omega}(w_1(x)\cdot{\cal S}_rw_2(x) )\,dx.\nonumber%\label{sadjS}
\end{align}
\item\label{PropSiii} Let $p\in[1,\infty)$. For all $w\in 
%\cancel{(W^{1,p}(\Omega))^n} 
(L^p(\Omega))^n $ it holds that
 \begin{align}
  {\cal S}_r w\underset{r\rightarrow0}{\rightarrow} {\cal S}_0w=w\qquad \text{in }\Lpn.\nonumber
 \end{align}
 \item \label{PropSiv} For $p =2$ it holds that 
 \begin{align*}
\| {\cal S}_r\|_{L((L^2(\Omega))^n)} \underset{r\rightarrow0}{\rightarrow}  1.
\end{align*}
  \end{enumerate}
 \end{Lemma}
 \begin{proof}
 The proof almost repeats that of {\cref{TWellDef}}. Therefore, we only check \cref{Snorm} and omit further details.  Let $p\in[1,\infty)$ and $\frac{1}{p}+\frac{1}{p^*}=1$. 
 Using  H\"older's inequality, Fubini's theorem, and our convention that $w$ vanishes outside $\Omega$ we deduce for all $w\in \Lpn$ that
 \begin{align}
  \|{\cal S}_rw\|_{\Lpn}^p=&n^p\int_{\Omega}\left|\int_0^1\avint_{S_{1}} (w(x+rsy)\cdot y)y\,d_{S_1}(y)ds\right|^p\,dx\nonumber\\
  \leq &n^p\int_{\Omega}\int_0^1\avint_{S_{1}}|w(x+rsy)|^p\,d_{S_1}(y)dsdx\nonumber\\
  =&n^p\int_0^1\avint_{S_{1}}\int_{\Omega}|w(x+rsy)|^p\,dxd_{S_1}(y)ds,\nonumber\\
  \leq&n^p\int_0^1\avint_{S_{1}}\int_{\Omega}|w(z)|^p\,dzd_{S_1}(y)ds\nonumber\\
  =&n^p\|w\|_{\Lpn}^p\ ,\nonumber
 \end{align}
 which means that 
 \begin{align}
  \|{\cal S}_r\|_{L(\Lpn)}\leq n.\label{est7_}
 \end{align}
 The proof in the case $p=\infty$  follows the same steps, or, alternatively, one passes to the limit as $p\rightarrow\infty$ in \cref{est7_}. 

 \end{proof}
\begin{Remark}
The constants in \cref{Tbnd} for any  $n\geq1$ and in \cref{Snorm} for $n\geq2$ are not necessarily optimal. For $p \neq 2$ it remains open whether or not 
\begin{align}
 \underset{r\rightarrow0}{\lim\inf}\left\|{\cal T}_{r}\right\|_{ L((L^{ p}(\Omega))^n)}=1,\nonumber\\
 \underset{r\rightarrow0}{\lim\inf}\left\|{\cal S}_{r}\right\|_{ L((L^{ p}(\Omega))^n)}=1.\nonumber
\end{align}
The answer may depend upon $\Omega$ { and $p$}.
 
\end{Remark}

\section{Well-posedness for a class of  evolution equations involving \texorpdfstring{${\cal T}_r$ or ${\cal S}_r$}{}}\label{modelEq}
% \noindent {\co To show that the solution of the nonlocal IBVP in the next section is nonnegative,} 
In this Section we establish  the existence and uniqueness of solutions to a certain  class of single evolution equations involving ${\cal T}_r$ or ${\cal S}_r$. This result is an important ingredient for our analysis of nonlocal systems in \cref{model}. Thus, 
we consider the following initial boundary value problem:
\begin{subequations}\label{modelLin}
 \begin{alignat}{3}
 &\partial_t c_r=\nabla\cdot (\Cl[a]{a1}\nabla c_r-\Cl[a]{a2}{G_{\varepsilon}({\cal R}_{r}(\Cl[a]{a3}\nabla c_r))})
+f&&\qquad\text{ in }(0,T)\times\Omega,\label{modelLinc}\\
&(\Cr{a1}\nabla c_r-\Cr{a2}{\cal R}_{r}(\Cr{a3}\nabla c_r))\cdot\nu=0&&\qquad\text{ in }(0,T)\times\partial\Omega,\label{modelLinbc}\\
&c_r(0,\cdot)=c_0&&\qquad\text{ in }\Omega.
\end{alignat}
 \end{subequations}
 Here
 \begin{align*}
   {\cal R}_r\in\{{\cal T}_r,{\cal S}_r\},
 \end{align*}
 and for $\varepsilon \geq 0$ we set 
 \begin{align} \label{feps}
 G_{\varepsilon}: \R^n \rightarrow \R^n, \quad x \mapsto \frac{x}{1+ \varepsilon |x|}.
 \end{align}
 A standard calculation shows that $G_{\varepsilon}$ is globally Lipschitz with a Lipschitz constant $1$.
 \begin{Remark}Observe that for $\varepsilon=0$ equation \cref{modelLinc} is linear, whereas for $\varepsilon>0$ the nonlocal part of the flux is a priori bounded. The latter helps us to construct nonnegative solutions in \cref{model}.\end{Remark}
 \noindent
We make the following assumptions: 
% {\cy w\"are es vielleicht sinnvoll die dritte Annahme hier zu ändern zu $\left\|\Cr{a2}\right\|_{L^{\infty}(0,T;L^{\infty}(\Omega))}\left\|\Cr{a3}\right\|_{L^{\infty}(0,T;L^{\infty}(\Omega))}\left\|{\cal R}_{r}\right\|_{{ L((L^2(\Omega))^n)}} < \left\| \Cr{a1}\right\|_{L^{\infty}(0,T;L^{\infty}(\Omega))}$ wegen der veränderten  \cref{Assump1}?} { Lassen wir es lieber so, wie es ist. Wir sollten die Koeffizienten des linearen Problems nicht unnötig weiter einschränken.  }
% \begin{subequations}\label{acond}
\begin{align}
  &\Cr{a1},\Cr{a2},\Cr{a3}\in L^{\infty}(0,T;L^{\infty}(\Omega)),\label{acond1}\\
  &\Cr{a1}>0\text{ and } \Cr{a1}^{-1}\in L^{\infty}(0,T;L^{\infty}(\Omega)),\label{acond2}\\
  &\left\|\Cr{a1}^{-\frac{1}{2}}\Cr{a2}\right\|_{L^{\infty}(0,T;L^{\infty}(\Omega))}\left\|\Cr{a1}^{-\frac{1}{2}}\Cr{a3}\right\|_{L^{\infty}(0,T;L^{\infty}(\Omega))}\left\|{\cal R}_{r}\right\|_{{ L((L^2(\Omega))^n)}}<1,\label{acond3}\\
  &f\in L^2(0,T;(H^1(\Omega))^*),\\
  &c_0\in L^2(\Omega).\label{acondc0}
\end{align}
% \end{subequations}
To shorten the notation, we introduce a pair of constants
\begin{align*}
 \alpha_r:=&\|\Cr{a1}^{-1}\|_{L^{\infty}(0,T;L^{\infty}(\Omega))}^{-1}\left(1-\left\|\Cr{a1}^{-\frac{1}{2}}\Cr{a2}\right\|_{L^{\infty}(0,T;L^{\infty}(\Omega))}\left\|\Cr{a1}^{-\frac{1}{2}}\Cr{a3}\right\|_{L^{\infty}(0,T;L^{\infty}(\Omega))}\left\|{\cal R}_{r}\right\|_{L((L^2(\Omega))^n)}\right),\\
 M_r:=&\|\Cr{a1}\|_{L^{\infty}(0,T;L^{\infty}(\Omega))}+\|\Cr{a2}\|_{L^{\infty}(0,T;L^{\infty}(\Omega))}\|\Cr{a3}\|_{L^{\infty}(0,T;L^{\infty}(\Omega))}\left\|{\cal R}_{r}\right\|_{L((L^2(\Omega))^n)}.
\end{align*}
Due to assumptions \cref{acond1,acond2,acond3} it is clear that 
\begin{align}
 {0<}\alpha_r, M_r{<\infty.}\nonumber
\end{align}
\noindent 
We introduce a family of operators
 \begin{align}
&\left<\mathcal{M}(t,u),\varphi\right>_{{(H^1(\Omega))^*,H^1(\Omega)}} := \int_{\Omega} \Cr{a1}(t,\cdot)\nabla u\cdot \nabla \varphi \, dx-\int_{\Omega} \Cr{a2}(t,\cdot) G_{\varepsilon}(\Cr{a3}(t,\cdot){\cal R}_{r}(\nabla u))\cdot\nabla \varphi \, dx,\nonumber                                                                                                                                                                                                                             \\
&\left<\mathcal{M}(u),\varphi\right>_{{ L^2(0,T;(H^1(\Omega))^*),L^2(0,T;H^1(\Omega))}} :=\int_0^T\left<\mathcal{M}(t,u),\varphi(t)\right>_{{(H^1(\Omega))^*,H^1(\Omega)}}\,dt.\nonumber
\end{align}
% The above assumptions imply that the integro-differential operator $\mathcal{M}(t,\cdot)$ is monotone for a.a. $t$. In fact, the following holds:
\begin{Lemma}\label{LemMon}
  Let \cref{acond1,acond2,acond3} be satisfied. Then:
  \begin{enumerate}[{\it (i)}]
\item\label{LemMoni} 
For a.a. $t\in[0,T]$ the operator \begin{align}
  \mathcal{M}(t,\cdot):H^1(\Omega) \rightarrow (H^1(\Omega))^*  \nonumber                                               \end{align}
 is well-defined, monotone, hemicontinuous, and satisfies for all $u\in H^1(\Omega)$ the bounds
 \begin{align} 
 &\left< {\cal M}(t,u),u \right>_{{(H^1(\Omega))^*,H^1(\Omega)}} \geq {\alpha_{r} ||\nabla u ||_{(L^2(\Omega))^n}^{{2}}},\label{estbelow}\\
 &||{\cal M}(t,u)||_{(H^1(\Omega))^*} \leq M_r ||\nabla u ||_{(L^2(\Omega))^n}. \label{estmtu}                                                                                              
\end{align}
Moreover, for all $u\in H^1(\Omega)$  the function $\mathcal{M}(\cdot,u)$ is measurable.
\item\label{LemMonii}  The operator
\begin{align}
  \mathcal{M}:L^2(0,T;H^1(\Omega)) \rightarrow L^2(0,T;(H^1(\Omega))^*                                                )\nonumber
  \end{align}
 is well-defined, monotone, hemicontinuous, and satisfies for all $u\in L^2(0,T;H^1(\Omega))$ the bounds
 \begin{align} 
 &\left< {\cal M}(u),u \right>_{{ L^2(0,T;(H^1(\Omega))^*),L^2(0,T;H^1(\Omega))}} \geq {\alpha_{r} ||\nabla u ||_{L^2(0,T;(L^2(\Omega))^n)}^{{ 2}}},\nonumber%\label{estbelowT}
 \\
 &||{\cal M}(u)||_{L^2(0,T;(H^1(\Omega))^*                                                )} \leq M_r ||\nabla u ||_{L^2(0,T;(L^2(\Omega))^n)}.\nonumber% \label{estmtuT}                                                                                              
\end{align}
\end{enumerate}

% $C_P$ being the Poincar\'e constant.
\end{Lemma}

\begin{proof}
 The assumptions on the coefficients $a_i$ together with the Lipschitz continuity of $G_{\varepsilon}$ readily imply   that  for a.a. $t\in[0,T]$  the operator ${\cal M}(t, \cdot)$ is well-defined and   satisfies \cref{estmtu}. Moreover, due to \cref{acond1} and $G_{\varepsilon}$ Lipschitz, it is also clear that ${\cal M}(\cdot ,u):[0,T]\rightarrow(H^1(\Omega))^*$ is measurable on $[0,T]$ for all $u\in H^1(\Omega)$, whereas for a.a. $t$ the mapping  $\lambda \mapsto\left< {\cal M}(t,u + \lambda v), w\right>_{{(H^1(\Omega))^*,H^1(\Omega)}}$ is continuous on $\R$, the latter meaning that $\mathcal{M}(t,\cdot)$ is hemicontinuous. 
Using  H\"older's inequality, the fact that $G_{\varepsilon}$ is Lipschitz with Lipschitz constant $1$, the assumptions on the $a_i$'s, and the properties of ${\cal R}_r$, we compute that
\begin{align}
&\left< {\cal M}(t,u)-{\cal M}(t,v), u-v\right>_{{(H^1(\Omega))^*,H^1(\Omega)}} \nonumber\\
 =&\int_{\Omega}\nabla (u-v)\cdot \Cr{a1}(t,\cdot)\nabla (u-v) \, dx-\int_{\Omega}\left(G_{\varepsilon}({\cal R}_{r}(\Cr{a3}(t,\cdot)\nabla u))-G_{\varepsilon}({\cal R}_r(\Cr{a3}(t,\cdot)\nabla v))\right)\cdot \Cr{a2}(t,\cdot) \nabla (u-v)\,dx\nonumber\\
 \geq&\left\|\Cr{a1}^{\frac{1}{2}}\nabla (u-v)\right\|^2_{(L^2(\Om))^n}
 -\int_{\Omega} \left|{\cal R}_{r}\left(\Cr{a1}^{-\frac{1}{2}}\Cr{a3}(t,\cdot)\left(\Cr{a1}^{\frac{1}{2}}\nabla (u-v)\right)\right)\right|\left| \Cr{a1}^{-\frac{1}{2}}\Cr{a2}(t,\cdot)\left(\Cr{a1}^{\frac{1}{2}} \nabla (u-v)\right)\right| \,dx\nonumber\\
 \geq&\left(1-\left\|\Cr{a1}^{-\frac{1}{2}}\Cr{a2}\right\|_{L^{\infty}(0,T;L^{\infty}(\Omega))}\left\|\Cr{a1}^{-\frac{1}{2}}\Cr{a3}\right\|_{L^{\infty}(0,T;L^{\infty}(\Omega))}\left\|{\cal R}_{r}\right\|_{{L((L^2(\Omega))^n)}}\right)\left\|a_1^{\frac{1}{2}}\nabla (u-v)\right\|^2_{(L^2(\Om))^n}\nonumber\\
 \geq& \alpha_r \left\|\nabla (u-v)\right\|^2_{(L^2(\Om))^n}\label{unicoera}
\end{align}
for  $u,v\in H^1(\Omega)$, which proves monotonicity. Further, taking $v = 0$ in \cref{unicoera} and using ${\cal M}(t,0) = 0$ yields \cref{estbelow}. Part {\it\cref{LemMoni}} is thus proved. A proof of  {\it\cref{LemMonii}} can be done similarly; we omit further details.  

\end{proof}

\noindent
Using the properties of the averaging operators proved in {\cref{OperProp}} we can define weak solutions to \cref{modelLin} in a manner very similar to that for  the classical, purely local case (i.e., when $\Cr{a2}\equiv 0$):
\begin{Definition}\label{DefmodelW}
 Let \cref{acond1}-\cref{acondc0} hold. We call the function $c_r:[0,T]\times \overline{\Omega}\rightarrow\R$ a weak solution of \cref{modelLin} if:
 \begin{enumerate}[{\it (i)}]
  \item $c_r\in L^2(0,T;H^1(\Omega))\cap C([0,T];L^2(\Omega))$, $\partial_t c_r\in L^2(0,T;(H^1(\Omega))^*)$;
  \item $c_r$ satisfies \cref{modelLinc}-\cref{modelLinbc} in the following sense: for all $\varphi\in H^1(\Omega)$ { and  a.a. $t\in(0,T)$}
  \begin{align}
   \left<\partial_t c_r,\varphi\right>_{{(H^1(\Omega))^*,H^1(\Omega)}}= -\int_{\Omega}{\Cr{a1}}\nabla c_r\cdot  \nabla \varphi \, dx +\int_{\Omega} \Cr{a2} G_{\varepsilon}(\Cr{a3}{\cal R}_{r}(\nabla c_r))\cdot\nabla \varphi\,dx+\left<f,\varphi\right>_{{(H^1(\Omega))^*,H^1(\Omega)}};\label{weakF}
  \end{align}
  \item $c_r(0,\cdot)=c_0$ in $L^2(\Omega)$.
 \end{enumerate}
\end{Definition}

\noindent
Using  standard theory %(see, e.g., \cite[Chapter III Proposition 4.1]{Showalter})
%\cite[Chapter II \S 3.4 Theorem 3.4]{Temam}), 
one readily proves the following existence result:
\begin{Lemma}\label{LemWPLin}
 Let \cref{acond1}-\cref{acondc0} hold. Then there exists a unique weak solution to \cref{modelLin} in terms of {\cref{DefmodelW}}.
 The solution satisfies the following estimates:
 \begin{align}
  &\|c_r\|_{C({ [}0,T{ ]};L^2(\Omega))}^2+\alpha_r\|\nabla c_r\|_{L^2(0,T;(L^2(\Omega))^n)}^2\leq \Cl{Ccr}(\alpha_r,T)\left(\|c_0\|_{L^2(\Omega)}^2+\|f\|_{L^2(0,T;(H^1(\Omega))^*)}^2\right),\label{estbas}\\
   &\|\partial_t c_r\|_{L^2(0,T;(H^1(\Omega))^*)}^2\leq \Cl{Cptcr}(\alpha_r,M_r,T)\left(\|c_0\|_{L^2(\Omega)}^2+\|f\|_{L^2(0,T;(H^1(\Omega))^*)}^2\right).\label{estdtc}
 \end{align}

\end{Lemma}
% The proof is standard, we provide it for the convenience of the reader.
\begin{proof}
The existence of a unique weak solution to \cref{modelLin} is a direct consequence of  \cref{LemMon}{\it\cref{LemMoni}} and the standard theory of evolution equations with monotone operators, see, e.g. \cite[Chapter III Proposition 4.1]{Showalter}. It remains to check the bounds \cref{estbas,estdtc}.
Taking $\varphi:=c_r$ in the weak formulation \cref{weakF} and using \cite[Chapter III Lemma 1.2]{temam2001navier}, \cref{estbelow}, 
and the Young inequality, we obtain that
\begin{align}
 \frac{1}{2}\frac{d}{dt}\|c_r\|_{L^2(\Omega)}^2
 \leq &-\alpha_r\|\nabla c_r\|_{(L^2(\Omega))^n}^2+\|c_r\|_{H^1(\Omega)}\|f\|_{(H^1(\Omega))^*}\nonumber\\
 =&-\alpha_r\|c_r\|_{H^1(\Omega)}^2+\alpha_r\|c_r\|_{L^2(\Omega)}^2+\|c_r\|_{H^1(\Omega)}\|f\|_{(H^1(\Omega))^*}\nonumber\\
 \leq& -\frac{1}{2}\alpha_r\|c_r\|_{H^1(\Omega)}^2+\alpha_r\|c_r\|_{L^2(\Omega)}^2+\frac{1}{2}\alpha_r^{-1}\|f\|_{(H^1(\Omega))^*}^2,\nonumber
\end{align}
which yields \cref{estbas} due to the Gronwall lemma. 
Finally, using \cref{estmtu}, we obtain from the weak formulation \cref{weakF} that
\begin{align}
 \|\partial_t c_r\|_{L^2(0,T;(H^1(\Omega))^*)}^2\leq &2M_r^2\|\nabla c_r\|_{L^2(0,T;(L^2(\Om))^n)}^2+2\|f\|_{L^2(0,T;(H^1(\Omega))^*)}^2.\nonumber
\end{align}
Together with \cref{estbas} this implies  \cref{estdtc}.

\end{proof}

\section{Nonlocal models involving averaging operators \texorpdfstring{${\cal T}_r$ and ${\cal S}_r$}{}}\label{model}
In this section we 
study 
% apply the  ideas from {\cref{Aver}} to
the following model IBVP:
 \begin{subequations}\label{nlProto}
 \begin{alignat}{3}
  &\partial_t c_{r}=\nabla\cdot\left(D_c(c_{r},v_{r})\nabla c_{r
  }-c_{r}\chi(c_{r},v_{r}) {\cal R}_{r}(\nabla g(c_r,v_r))\right)+f_c(c_{r},v_{r})&&\text{ in }\R^+\times\Omega,\label{nlProtoc}\\
  &\partial_t v_{r}=D_v\Delta v_r+f_v(c_{r},v_{r})&&\text{ in }\R^+\times\Omega,\label{nlProtov}\\
  &D_c(c_{r},v_{r})\partial_{\nu} c_{r
  }-c_{r}\chi(c_{r},v_{r}) {\cal R}_{r}(\nabla g(c_r,v_r))\cdot \nu=D_v \partial_{\nu}v_{r}=0&&\text{ in }\R^+\times\partial\Omega,\label{nlProtobc}\\
  &c_{r}(0,\cdot)=c_0,\ v_{r}(0,\cdot)=v_0&&\text{ in }\Omega.
 \end{alignat}
 \end{subequations}
 Here, as in the previous section, ${\cal R}_r$ stands for any of the two averaging operators:
 \begin{align}%\label{eq:ops}
{ {\cal R}_r\in\{{\cal T}_r,{\cal S}_r\}.}\nonumber
      \end{align}
 We assume that the diffusion coefficient $D_v$ is either a positive number, or it is zero.
 
 \noindent
 Equations \cref{nlProtoc}-\cref{nlProtov} are closely related to \cref{nlHapto} and \cref{nlChemo} in {\cref{intro}}, the difference being that the terms involving the adhesion 
velocity/non-local gradient are now replaced by those including the averaging operators ${\cal T}_{r}$/${\cal S}_{r}$ from {\cref{AverOper}}. Our motivation for introducing this change is twofold. First of all, due to \cref{ArTr} 
and \cref{DrSr} it affects the points in the boundary layer $\Omega\backslash \Omega_r$, at the most. On the other hand,  {\cref{BndL}} indicates that 
including, e.g.,  ${\cal A}_r$ can lead to limits with  unexpected  blow-ups on  the boundary of $\Omega$. 

%\noindent
%Equations \cref{nlProtoc}-\cref{nlProtov} are supplemented with initial and no-flux boundary conditions...

\noindent
System \cref{nlProto} is a non-local version of the hapto-/chemotaxis system
\begin{subequations}\label{lProto}
 \begin{alignat}{3}
  &\partial_t c=\nabla\cdot\left(D_c(c,v)\nabla c-c\chi(c,v) \nabla g(c,v)\right)+f_c(c,v)&&\qquad\text{ in }\R^+\times\Omega,\label{lProtoc}\\
  &\partial_t v=D_v\Delta v+f_v(c,v)&&\qquad\text{ in }\R^+\times\Omega,\label{lProtov}\\
  &D_c(c,v)\partial_{\nu} c_{r
  }-c\chi(c,v) \partial_{\nu} g(c,v)=D_v \partial_{\nu}v=0&&\qquad\text{ in }\R^+\times\partial\Omega,\\
  &c(0,\cdot)=c_0,\ v(0,\cdot)=v_0&&\qquad\text{ in }\Omega.
 \end{alignat}
 \end{subequations}
 In this case, the actual diffusion and haptotactic sensitivity coefficients are
 \begin{align}
  &{\widetilde D}_c(c,v)=D_c(c,v)-c\chi (c,v)\partial_cg(c,v),\nonumber%\label{Dtil}
  \\
  &{\widetilde \chi}(c,v)=\chi(c,v)\partial_vg(c,v),\nonumber%
  \label{chitil}
 \end{align}
 so that 
 in the classical formulation   
\cref{lProtoc} takes the form  
\begin{equation}
\partial_t c=\nabla\cdot\left({\widetilde D}_c(c,v)\nabla c-c{\widetilde \chi}(c,v) \nabla v\right)+f_c(c,v).\qquad\text{ in }\R^+\times\Omega.\nonumber%\label{lProtoCl}
\end{equation} 
The main goal of this Section is to establish, under suitable assumptions on the system coefficients which are introduced in \cref{Assupm}, a rigorous convergence as $r\rightarrow0$ of solutions of the nonlocal model family \cref{nlProto} to those of the local model \cref{lProto}, see \cref{limitr}. This is accomplished in the final  \cref{SecConv}. Since we are  dealing here with a new type of nonlocal system,  we  establish for \cref{nlProto} the existence  of nonnegative solutions in \cref{SecEx,apriori}.

\subsection{Problem setting  and main result of the section} \label{Assupm}

We begin with several general assumptions  about the coefficients of system \cref{nlProto}.
\begin{Assumptions}\label{Assump1} Let $D_v\in \R^+_0$, $D_c,\chi\in C_b(\R^+_0\times\R^+_0)$, and $g,f_c,f_v\in C^1(\R^+_0\times\R^+_0)$ satisfy for some $s \geq 0$:
\begin{alignat}{3}
 & {\Cl{Dcmin}\leq D_c\leq \Cl{Dcmax}}& { \qquad \text{in }\R^+_0\times\R^+_0}&\qquad {\text{for some }\Cr{Dcmin},\Cr{Dcmax}>0,}\nonumber\\
 & {\nabla_{(c,v)}g,\ } {\nabla_{(c,v)} f_v\in (L^{\infty}(\R^+_0\times\R^+_0))^2,}&& \nonumber%\label{bngnf} 
 \\
 &f_c(0,\cdot)\equiv0,&&\nonumber%\label{fc0}
 \\
 &f_v(\cdot,0)\equiv0.&&\nonumber%\label{fv0}
\end{alignat}
Assume that the coefficients satisfy the following bounds:
% \begin{enumerate}[(1)]
%  \item  
\begin{align}
 &\Cr{Q1}:= {\sup_{c,v\geq0}c|\chi(c,v)|}<\infty, \label{cchi}\\
 &\Cr{Q2}:={\sup_{c,v\geq0}|\partial_cg(c,v)|}<\infty. \label{pcgb}
\end{align}

\noindent
{Further, we assume that the initial values satisfy 
\begin{align}
 &0\leq c_0 \in L^2(\Omega),\nonumber\\
 &{0\leq v_0} \in H^1(\Omega).\label{iniv}
\end{align}
%  $$ and $$ {\co and $0\leq c_0$ and $0 \leq v_0 \leq K_v$ for some constant $K_v > 0$}.
 }
\end{Assumptions}
\begin{Remark}
       If $D_v>0$, then assumption \cref{iniv} can be replaced by a weaker one, such as
       \begin{align}
        v_0\in L^2(\Omega).\nonumber
       \end{align}
We keep \cref{iniv} in order to simplify the exposition. 
      \end{Remark}
\noindent
In addition, we will later choose one of the following assumptions on $f_c$ and the nonlocal operator: 
\begin{Assumptions}[Further assumptions on $f_c$] \label{Assump2}
One of the following { conditions} holds:
\begin{enumerate}[{\it (a)}] 
\item \label{Assump2a} $$\nabla_{(c,v)}f_c \in \left( L^{\infty}(\R_0^+\times \R_0^+)\right)^2$$
\item \label{Assump2b} %or 
\begin{alignat}{3}
&|f_c(c,v)| \leq \Cl{Nem}(1+|c|^s) \qquad&&\text{in }\R^+_0\times\R^+_0\qquad \text{for some }\Cr{Nem}\geq0, \label{ineqNem}
\\
& cf_c(c,v) \leq \Cl{Dis1}-\Cl{Dis2} c^{s+1} \qquad &&\text{in }\R^+_0\times\R^+_0\qquad \text{for some }\Cr{Dis1}\geq0,\ \Cr{Dis2}>0. \nonumber%\label{ineqDis} 
\end{alignat}
\end{enumerate}
\end{Assumptions}

\begin{Assumptions}[Assumptions on $\mathcal{R}_r$] \label{Assump3}
% The operator norm of $\mathcal{R}_r$ satisfies 
One of the following holds:
\begin{enumerate}[{\it (a)}]
\item for a given fixed $r\in(0,r_0]$
 $$\Cl{Cgstr}(\|\mathcal{R}_r \|):=1-\frac{\Cr{Q1}\Cr{Q2}}{ \Cr{Dcmin}}\left\|{\cal R}_{r}\right\|_{L((L^2(\Omega))^n)}>0$$ \label{Assump3a}
\item 
\begin{align}
  \Cl{Cgst}:=\frac{\Cl{Q1}\Cl{Q2}}{\Cr{Dcmin} }%{\cb \cancel{\underset{r\rightarrow0}{\lim\sup}\left\|{\cal R}_{r}\right\|_{L((L^2(\Omega))^n)}}}
  \,<1.\label{condC*}
 \end{align} \label{Assump3b}
\end{enumerate}
\end{Assumptions}

% \begin{Remark}\label{rcsup}
%{\cy Conditions \cref{ineqNem} imply 
%\begin{align}\label{csup}
%\Cl{Qsup} := \sup_{c,v\geq 0} \frac{f_c(c,v)}{c} < \infty.
%\end{align}}
% \end{Remark}

\begin{Example}\label{Bsp} Let  
\begin{align*}
&D_v=0,\\
&F_r(\rho):=(n+1)e^{-r\rho},\\
 &g(c,v):=\frac{S_{cc}c+S_{cv}v}{1+c+v}\qquad\text{ for some constants}\qquad S_{cc},S_{cv}>0,\\
 %&D_c(c,v)=\left(\frac{a(1+c)}{1+cv}\right)^2,\quad  a>0,\\
 &D_c(c,v) := \frac{1+c}{1+c+v},\\
 &\chi(c,v):=\frac{b}{1+c+v},\quad b>0,\\
 &f_c(c,v):=\mu_c{\frac{c}{1+c^2}}(K_c-c-\eta_c v)\qquad\text{ for some constants}\qquad
 \qquad K_c,\eta_c>0,\ \mu_c {>} 0,\\
  &f_v(c,v):=\mu_vv(K_v-v)-\lambda_vv{\frac{c}{1+c}}\qquad\text{ for some constants}\qquad K_v,\lambda_v>0,\  \mu_v\ge 0,
\end{align*}
and assume that
\begin{align}
 0\leq v_0\leq K_v.\nonumber
\end{align}
Then, it holds a priori that
\begin{align}
 0\leq v\leq K_v\nonumber
\end{align}
for any $v$ which solves \cref{nlProtov}. Therefore it suffices to consider the coefficient functions in $\R_0^+ \times [0,K_v]$.

\noindent
%Observe that
%\begin{align}
%  &\Cr{CP2_}(2)=(n+1)
%  \left(\frac{n}{n+2}\right)^{\frac{1}{2}},\nonumber
%\end{align}
%so that 
%\begin{align}
% \underset{r\rightarrow0}{\lim\sup}\left\|{\cal R}_{r}\right\|_{L((L^2(\Omega))^n)}\leq \Cl{TS1}:=\begin{cases}(n+1)
%  \left(\frac{n}{n+2}\right)^{\frac{1}{2}}&\text{for }{ {\cal R}_r={\cal T}_r,}\\
%  n&\text{for }{{\cal R}_r={\cal S}_r.}
%  \end{cases}\nonumber
%\end{align}
For $D_c$ it holds on $\R_0^+ \times [0,K_v]$ that
\begin{align}
D_c(c,v) \geq \frac{1+c}{1+c+K_v} \geq \frac{1}{1+ K_v} =: \Cr{Dcmin}\nonumber
\end{align}
and
\begin{align}
D_c(c,v) \leq 1 =: \Cr{Dcmax}.\nonumber
\end{align}
Moreover, $\nabla_{(c,v)}g,\ \nabla_{(c,v)} f_v \in (L^{\infty}(\R^+_0\times\R^+_0))^2$,  due to 
\begin{align}
\Cr{Q2} = \sup_{c,v\geq0}|\partial_cg(c,v)| =& \max_{0\leq v\leq K_v}\max_{c\geq0}\frac{|S_{cc}(1+v)-S_{cv}v|}{(1+c+v)^2} \nonumber\\
=& \max\left\{S_{cc},\left|\frac{S_{cc}}{1+K_v}-\frac{S_{cv}K_v}{(1+K_v)^2}\right| \right\},\nonumber\\
\sup_{c,v\geq0}|\partial_vg(c,v)|=&\max_{0\leq v\leq K_v}\max_{c\geq0}\frac{|S_{cv}(1+c)-S_{cc}c|}{(1+c+v)^2}\nonumber\\
=&\max_{c\geq0}\frac{|S_{cv}(1+c)-S_{cc}c|}{(1+c)^2}<\infty,\nonumber\\
\sup_{c,v \geq 0} |\partial_c f_v(c,v)| =& \lambda_v K_v \nonumber
\end{align}
and
\begin{align}
\sup_{c,v \geq 0} |\partial_v f_v(c,v)| = \sup_{c,v \geq 0}\left|\mu_v(K_v-2v)-\lambda_v \frac{c}{1+c} \right| < \infty.\nonumber
\end{align}
For $\Cr{Nem} {:=} \mu_c(K_c + 1 +\eta_c K_v)$, $\Cr{Dis1}{:=} \mu_c(K_c + 1)$ and $\Cr{Dis2}{:=} \mu_c $ we can estimate on $\R^+_0\times\R^+_0$ that
\begin{align}
&|f_c(c,v)| \leq \Cr{Nem}, \nonumber\\
&cf_c(c,v) \leq \mu_c \left( K_c + \frac{c}{1+c^2}-c\right) \leq \Cr{Dis1}-\Cr{Dis2} c.\nonumber
\end{align}
Further, 
\begin{align}
\Cr{Q1} = \underset{c\geq0}{ \sup} \frac{bc}{1+c} = b\nonumber
\end{align}
holds.

% \begin{align}
% &\Cr{Q1}=\underset{c\geq0}{ \sup}\frac{bc}{a(1+c)}=\frac{b}{a},
% \end{align}
% \begin{align}
% \Cr{Q2}=&\max_{0\leq v\leq K_v}\max_{c\geq0}\frac{|S_{cc}(1+v)-S_{cv}v|(1+cv)}{a(1+c)(1+c+v)^2}\nonumber\\
% =&\frac{1}{a}\max_{0\leq v\leq K_v}|S_{cc}(1+v)-S_{cv}v|\nonumber\\
% =&\frac{1}{a}\max\left\{S_{cc},|S_{cc}(1+K_v)-S_{cv}K_v|\right\},
%\end{align}
%\begin{align}
%\Cr{Q5}=&\max_{0\leq v\leq K_v}\max_{c\geq0}\frac{|S_{cv}(1+c)-S_{cc}c|}{(1+c+v)^2}\nonumber\\
%=&\max_{c\geq0}\frac{|S_{cv}(1+c)-S_{cc}c|}{(1+c)^2}<\infty,\end{align}
%\begin{align}
%\Cr{Q34}=&\max_{c,v\geq0}(\mu_c(K_c-c-\eta_cv))=\mu_cK_c,
%\end{align}
%\begin{align}
%\Cr{Q4}=\max_{0\leq v\leq K_v}\max_{c\geq0}({ \mu_vK_v}-2\mu_vv-\lambda_v c) =\mu_vK_v,
%\end{align}
%and
%\begin{align}
% \Cr{Q3}=&\max_{0\leq v\leq K_v}\max_{c\geq0}{ \lambda _v}v\frac{1+cv}{a(1+c)}\nonumber\\
% =&  \frac{ \lambda _v}{a}\max_{0\leq v\leq K_v}v\max_{c\geq0}\left(v+\frac{1-v}{1+c}\Big %)=\frac{ \lambda _v}{a}K_v.
%\end{align}

\noindent
Thus,  \textit{Assumptions} \ref{Assump1}, \ref{Assump2}{\eqref{Assump2b}} and \ref{Assump3}{ \eqref{Assump3b}} are fulfilled if 
\begin{align*}%\label{cond-assumption}
% \frac{b}{a^2}\Cr{TS1}
%\max\left\{S_{cc},|S_{cc}(1+K_v)-S_{cv}K_v|\right\}<1, \qquad a, b>0.
 (1+K_v)b\max\left\{S_{cc},\left|\frac{S_{cc}}{1+K_v}-\frac{S_{cv}K_v}{(1+K_v)^2}\right| \right\} < 1.
\end{align*}

\noindent
This choice of coefficient functions can be used to describe a population of cancer cells which interact among themselves and with the surrounding extracellular matrix (ECM) tissue. Both interaction types are due to adhesion, whether to other cells (cell-cell adhesion) or to the matrix (cell-matrix adhesion). The interaction force $F_r(\rho)$ is taken to diminish with increasing interaction range $\rho $ and/or of the sensing radius $r$: cells too far apart/out of reach hardly interact in a direct way. Function $g(c,v)$ characterises effective interactions.
% , mutually between cells and with the neighboring tissue; 
Here the coefficients $S_{cc}$ and $S_{cv}$ represent cell-cell and cell-matrix adhesion strengths, respectively. Our choice of $g$ accounts for some adhesiveness limitation imposed by high local cell and tissue densities. It is motivated by the fact that overcrowding may preclude further adhesive bonds, e.g. due to saturation of receptors. The diffusion coefficient $D_c(c,v)$ is chosen to be everywhere positive and increase with a growing population density, thus enhancing diffusivity under population pressure, but, further, limited by excessive cell-tissue interaction. The latter also applies to the choice of the sensitivity function $\chi$. Indeed, there is evidence that tight packing of cells and ECM limits diffusivity and the advective effects of haptotaxis \cite{lu12}. Thereby the constant $b>0$ is assumed to be rather small. Finally, $f_c$ and $f_v$ describe growth of cells and tissue limited by concurrence for resources. 
%%%%%%%%Numerics%%%%%%
% Simulations for the situation described by these functions are presented in {\cref{sec:numerics}}. 

\end{Example}

\noindent
Next, we introduce weak-strong solutions to our problem. The definition is as follows:  
\begin{Definition}\label{DefSol}
  Let  \cref{Assump1} hold. Let $r\in [0,r_0]$. We call a pair of functions $(c_r,v_r):\R^+_0\times \overline{\Omega}\rightarrow\R^+_0\times\R^+_0$ a global \underline{weak-strong} solution of \cref{nlProto} if for all $T>0$:
 \begin{enumerate}[{\it (i)}]
  \item\label{regcr} $c_r\in L^2(0,T;H^1(\Omega))\cap C_w([0,T];L^2(\Omega))$, $\partial_t c_r\in L^1(0,T;(W^{1,\infty}(\Omega))^*)$;
  \item  $v_r\in C([0,T];H^1(\Omega))$, $\partial_t v_r\in L^2(0,T;L^2(\Omega))$,  $D_v v_r\in L^2(0,T;H^2(\Omega))$; 
  \item $f_c(c_r,v_r)\in L^1(0,T;L^1(\Omega))$, $f_v(c_r,v_r)\in L^{ 2}(0,T;L^{ 2}(\Omega))$;
  \item $(c_r,v_r)$ satisfies \cref{nlProto} in the following weak-strong sense: for all $\varphi\in C^1(\overline{\Omega})$ and a.a. $t\in(0,T)$  
  \begin{subequations}\label{weakcvr}
  \begin{align}
   \left<\partial_t c_r,\varphi\right>_{{(W^{1,\infty}(\Omega))^*,W^{1,\infty}(\Omega)}}=&-\int_{\Omega}\left(D_c(c_{r},v_{r})\nabla c_{r
  }-c_{r}\chi(c_{r},v_{r}) {\cal R}_{r}(\nabla g(c_r,v_r))\right)\cdot\nabla \varphi  \, dx\nonumber\\
  &+ \int _{\Om}f_c(c_r,v_r)\varphi\, dx,\label{weakcr}
\end{align}

\begin{align}
 c_{{r}}(0,\cdot)=c_0\qquad\text{in }L^2(\Omega),
\end{align}
and
\begin{alignat}{3}
  &{\partial_t v_r=D_v\Delta v_r+f_v(c_r,v_r)}\qquad&&{ \text{a.e. in }(0,T)\times\Omega},\label{weakvr}\\
  &{ D_v \partial_{\nu} v_r=0\qquad}&&{ \text{a.e. in }(0,T)\times\partial\Omega},\label{weakbcv}\\
  &v_{{r}}(0,\cdot)=v_0\qquad&&\text{in }H^1(\Omega).                                                      \label{initval} 
  \end{alignat}
\end{subequations}
 \end{enumerate}
\end{Definition}
% { \begin{Remark}
%        The  weak reformulation \cref{weakcr} is obtained using partial integration. The term $\nabla g(c_r,v_r)$ is well-defined and belongs to $L^2(0,T;L^2(\Omega))$.
%       \end{Remark}
% }
\begin{Remark}
 Observe that for $r=0$ we obtain a  corresponding solution definition for the local system \cref{lProto}. 
\end{Remark}
\noindent
Our main result now reads:
\begin{Theorem}  \label{limitr}
Let \cref{Assf,Assump1,Assump2,Assump3}{\it\cref{Assump3b}} hold. 
Then, there exists a sequence $r_m\rightarrow0$ as $m\rightarrow\infty$ and  solutions $(c_{r_m},v_{r_m})$ and $(c,v)$ in terms of \cref{DefSol} corresponding to $r=r_m$ and $r=0$, respectively, s.t. 
 \begin{align} 
 c_{r_m} &\underset{m \rightarrow \infty}{\rightarrow} c \qquad \text{in } L^2(0,T;L^2(\Om)) ,\nonumber\\%\label{convcml2} \\
 v_{r_m} &\underset{m \rightarrow \infty}{\rightarrow} v \qquad \text{in } L^2(0,T;L^2(\Om)).\nonumber%\label{convvml2} 
 \end{align}
\end{Theorem}
\noindent
This Theorem is proved in \cref{SecConv}. 
\begin{Notation}
 Dependencies upon such parameters as the space dimension $n$, domain $\Omega$, function $c$,  the norms of the initial data $c_0$ and $v_0$, norms and bounds for the coefficient functions are mostly 
{\bf not} indicated in an explicit way.
\end{Notation}

\subsection{Global existence of solutions to  \texorpdfstring{\cref{nlProto}}{}:  the case of \texorpdfstring{$f_c$}{} Lipschitz}\label{SecEx}
In this Subsection we address the existence of solutions to the nonlocal model \cref{nlProto} for the case when $f_c$ satisfies \cref{Assump2}{\it\cref{Assump2a}}. 
The main result of the Subsection is as follows:
\begin{Theorem} \label{thmmainEx}
Let  \cref{Assf,Assump1,Assump2}{\it\cref{Assump2a}} hold and let $r$ satisfy \cref{Assump3}{\it\cref{Assump3a}}.
Then there exists a global weak-strong solution to \cref{nlProto} in terms of \cref{DefSol} with $\partial_t c_r\in L^2(0,T;(H^1(\Omega))^*)$. 
\end{Theorem}
\noindent
Since we aim at constructing nonnegative solutions, it turns out to be helpful to consider first the following family of  approximating problems:
 \begin{subequations}\label{ApprProto}
 \begin{alignat}{3}
  &\partial_t \cre=\nabla\cdot\left(D_c(\cre,\vre)\nabla \cre
  -\cre\chi(\cre,\vre)\Big ( G_{\varepsilon}( {\cal R}_{r}(\partial_c g(\cre,\vre) \nabla \cre))\right.\nonumber\\
  &\left.\phantom{\partial_t \cre=} + G_{\varepsilon}( {\cal R}_{r}(\partial_v g(\cre,\vre) \nabla \vre))\Big )\right)+f_c(\cre,\vre)&&\text{ in }\R^+\times\Omega,\label{ApprProtoc}\\
  &\partial_t \vre=D_v\Delta \vre+f_v(\cre,\vre)&&\text{ in }\R^+\times\Omega,\label{ApprProtov}\\
  &D_c(\cre,\vre)\nabla \cre
  -\cre\chi(\cre,\vre)\Big ( G_{\varepsilon}( {\cal R}_{r}(\partial_c g(\cre,\vre) \nabla \cre)) \nonumber \\
  &\phantom{\partial_t \cre=}+ G_{\varepsilon}( {\cal R}_{r}(\partial_v g(\cre,\vre) \nabla \vre))\Big )\cdot \nu=D_v \partial_{\nu}\vre=0&&\text{ in }\R^+\times\partial\Omega,\label{ApprProtobc}\\
  &\cre(0,\cdot)=c_0,\ \vre(0,\cdot)=v_0&&\text{ in }\Omega,
 \end{alignat}
 \end{subequations}
where $G_{\varepsilon}$ was defined in \cref{feps}. In order to obtain existence for the original problem, i.e., for $\varepsilon=0$, we first prove existence of nonnegative solutions for the cases when $\varepsilon, D_c>0$. This corresponds to a chemotaxis problem with a nonlocal flux-limited drift. Weak-strong solutions to \cref{ApprProto} are understood as in \cref{DefSol}, with the obvious modification of the weak formulation, which now reads:
\begin{align}
 \left<\partial_t \cre,\varphi\right>_{{(H^1(\Omega))^*,H^1(\Omega)}}=&-\int_{\Omega}{D_c(\cre,\vre)}\nabla \cre\cdot \nabla \varphi \, dx\nonumber\\
 &+ \int_{\Omega}{ \cre\chi(\cre,\vre)} G_{\varepsilon}({\cal R}_{r}(\partial_cg(\cre,\vre)\nabla \cre))\cdot\nabla \varphi\,dx\nonumber\\
   &+\int_{\Omega}{\cre\chi(\cre,\vre)}G_{\varepsilon}({\cal R}_{r} (\partial_vg(\cre,\vre)\nabla\vre)) \cdot \nabla \varphi+f_c(\cre,\vre)\varphi\, dx.\label{weakapprcr}
\end{align}
\begin{Lemma} \label{thmexr}
 Let Assumptions of \cref{thmmainEx} be satisfied.  Assume further that
 \begin{align}
 \varepsilon, D_v>0.\nonumber
  \end{align}
Then there exists a global weak-strong solution to \cref{ApprProto} with  $\partial_t\cre\in L^2(0,T;(H^1(\Omega))^*)$.
 \end{Lemma}

\begin{proof}
To begin with, we extend the coefficients:
\begin{align*}
&D_c(c,v) := D_c(-c,v), \quad (\chi,g,f_c,f_v)(c,v) := -(\chi,g,f_c,f_v) (-c,v)\qquad \text{for }c<0.
\end{align*}
These coefficients still satisfy   \cref{Assump1}, {\it\ref{Assump2}}{\it\cref{Assump2a}}, and {\it\ref{Assump3}}{\it\cref{Assump3a}} if we consider all suprema over $c \in \R$ instead of $c \in \R_0^+$. %Moreover, it can be shown that $\chi$ is bounded as $v$ is a priori bounded.

\noindent
Our approach to proving existence is based on the classical Leray-Schauder principle \cite[Chapter 6, \S 6.8, Theorem 6.A]{ZeidlerNFA1}. In order to apply this theorem we first 'freeze' $c_{r\varepsilon}$ in the system coefficients of \cref{ApprProto}, replacing it by  $\bar c_{r\varepsilon}$. Correspondingly, we obtain the following weak formulation in place of \cref{weakapprcr}:  
For all $\varphi\in H^1(\Omega)$ and a.a. $t>0$ 
\begin{subequations}\label{weakcbarc}
\begin{align}
   \left<\partial_t \cre,\varphi\right>_{{(H^1(\Omega))^*,H^1(\Omega)}}=&-\int_{\Omega}D_c(\bar c_{r\varepsilon},\vre)\nabla \cre\cdot \nabla \varphi \, dx\nonumber\\
   &+ \int_{\Omega}\bar c_{r\varepsilon}\chi(\bar c_{r\varepsilon},\vre) G_{\varepsilon}({\cal R}_{r}(\partial_cg(\bar c_{r\varepsilon},\vre)\nabla \cre))\cdot \nabla \varphi\,dx\nonumber\\
   &+\int_{\Omega}\bar c_{r\varepsilon}\chi(\bar c_{r\varepsilon},\vre)G_{\varepsilon}({\cal R}_{r} (\partial_vg(\bar c_{r\varepsilon},\vre)\nabla\vre)) \cdot \nabla \varphi+f_c(\bar c_{r\varepsilon},\vre)\varphi\, dx,\label{weakcrbar}
\end{align}
\begin{align}
 c_{{r\varepsilon}}(0,\cdot)=c_0\qquad\text{in }L^2(\Omega)
\end{align}
and
\begin{alignat}{3}
   &\partial_t \vre=D_v\Delta \vre+f_v(\bar c_{r\varepsilon},\vre)\qquad &&\text{a.e. in }(0,T)\times\Omega,\label{weakvrbar}\\
  &D_v \partial_{\nu} \vre=0&&\text{a.e. in }(0,T)\times\partial\Omega,\\
  &\vre(0,\cdot)=v_0\qquad&&\text{in }H^1(\Omega).  \label{decini}                                                     \end{alignat}
\end{subequations}
Let $T>0$ and let $\bar c_{r\varepsilon}\in L^2(0,T;L^2(\Omega))$. 
% {\cy We first consider the case $D_v>0$.}
Since $f_v$ is assumed to be Lipschitz, we can make use of the standard  theory \cite{LSU} which implies that the semilinear parabolic initial boundary value problem \eqref{weakvrbar}-\eqref{decini} 
% with the corresponding weak formulation \cref{weakvrbar} 
possesses a unique global strong solution $0\leq  \vre \in L^2(0,T;H^2(\Omega))$ with $\partial_t  \vre \in L^2(0,T;L^2(\Omega))$,  and satisfying the estimate
  \begin{align}
   \|{ \vre}\|_{L^{\infty}(0,T;{ H^1}(\Omega))}^2+{\|{ \vre}\|_{L^2(0,T;H^2(\Omega))}^2+\|\partial_t { \vre}\|_{L^2(0,T;L^2(\Omega))}^2}
   \leq &{\Cl{C13}(T)}\left(\|v_0\|_{{ H^1}(\Omega)}^2+\| { \bar c_{r\varepsilon}}\|_{L^2(0,T;{ L^2}(\Omega))}^2\right).\label{apriori1}
  \end{align}
  Here and further in the proof we omit the dependence of constants upon $D_v$. 
Set
\begin{align} %\label{coeffsmon}
 &\Cr{a1}:=D_c( \bar c_{r\varepsilon},\vre),\qquad
 \Cr{a2}:=\bar c_{r\varepsilon}\chi( \bar c_{r\varepsilon},\vre),\qquad
 \Cr{a3}:=\partial_cg( \bar c_{r\varepsilon},\vre),\nonumber\\
 &\left<f,\varphi\right>_{{(H^1(\Omega))^*,H^1(\Omega)}}:=\int_{\Omega}{\bar c_{r\varepsilon}\chi(\bar c_{r\varepsilon},\vre)} G_{\varepsilon}({\cal R}_{r} (\partial_vg(\bar c_{r\varepsilon},\vre)\nabla\vre)) \cdot \nabla \varphi+f_c(\bar c_{r\varepsilon},\vre)\varphi\, dx.\nonumber
\end{align}
Due to our assumptions about $D_c,\chi,g$, and $f_c$, { these coefficients   $a_i$ and $f$}  satisfy the requirements of 
 \cref{LemMon}. Consequently, there exists a unique global weak solution $c_{r\eps}$ to  problem \eqref{modelLin} with these coefficients.
%  in terms of its weak formulation \cref{weakcrbar}. 
 We  estimate for the corresponding constants $\alpha_r$ { and} $M_r$ introduced in {\cref{LemMon}}:
 \begin{align}
  \alpha_{r}\geq&\Cr{Dcmin}\Cr{Cgstr}{(r)}=:\Cl{alpha_}{(r)},\label{estalpha}\\
  M_r\leq&\Cr{Dcmax}+\Cr{Q1}\Cr{Q2}\left\|{\cal R}_{r}\right\|_{L((L^2(\Omega))^n)}=:\Cl{Mr_}{(r)}, \label{estmr}
 \end{align}
% { Ich glaube, dass $\Cr{Dcmax}$ hier schon n\"otig ist.} { Einverstanden, hatte die Nenner in $\Cr{Q1}$ und $\Cr{Q2}$ \"ubersehen}\\
and, due to \cref{apriori1}, 
\begin{align}
 \|f\|_{L^2(0,T;(H^1(\Omega))^*)}
 \leq &\|\nabla {\vre}\|_{L^2(0,T;(L^2(\Omega))^n)}{||\partial_v g||_{L^{\infty}(\R^+_0\times\R^+_0)}}\left\|{\cal R}_{r}\right\|_{ L(L^2(\Omega))^n)}\Cr{Q1}\nonumber\\
 &+\|\partial_c f_c\|_{L^{\infty}(\R^+_0\times\R^+_0)}\left({\|{\vre}\|_{L^2(0,T;L^2(\Omega))}}+\| \bar c_{r\varepsilon}\|_{L^2(0,T;L^2(\Omega))}\right)\nonumber\\
 {\leq}&\Cl{f_}({r,}T)\left(1+\| { \bar c_{r\varepsilon}}\|_{L^2(0,T;L^2 (\Omega))}\right).\label{estf}
\end{align}
 Combining \cref{estbas}-\cref{estdtc} and  \cref{estalpha}-\cref{estf}, we obtain the following bounds for $c_{ r\eps}$:
  \begin{align}
  &\|{\cre}\|_{ C({ [}0,T{ ]};L^2(\Omega))}^2+\alpha_{r}\|\nabla { \cre}\|_{L^2(0,T;L^2(\Omega))}^2\leq \C{({r,}T)}\left(1+\| { \bar c_{r\varepsilon}}\|_{L^2(0,T;L^2(\Omega))}^2\right),\label{estbas_}\\
  &\|\partial_t {\cre}\|_{L^2(0,T;(H^1(\Omega))^*)}^2\leq \C{({r,}T)}\left(1+\| { \bar c_{r\varepsilon}}\|_{L^2(0,T;L^2(\Omega))}^2\right).\label{estdtc_}
 \end{align}
 Now  consider the mapping
 \begin{align*}
 % \Phi( {\bar c_{r\varepsilon}}):={\cre}.\nonumber
 \Phi :{\bar c_{r\varepsilon}}\mapsto \cre.
 \end{align*}
% (Schlage vor, das so zu schreiben, sonst sieht es seltsam aus, wenn wir einerseits definieren $\Phi( {\bar c_{r\varepsilon}}):={\cre}$ und weiter unten genau das zeigen.)}
Thanks to \cref{estbas_} and \cref{estdtc_}, $\Phi$ is well-defined in  $L^2(0,T;L^2(\Omega))$ and 
 \begin{align}
  &\Phi:L^2(0,T;L^2(\Omega))\rightarrow \{ u\in L^2(0,T;H^1(\Omega)):\ \partial_t  u\in L^2(0,T;(H^1(\Omega))^*)\} \nonumber\\
  &\text{maps bounded sets on bounded sets}.\label{Pbndonbnd}
 \end{align}
Due to the Lions-Aubin lemma, \cref{Pbndonbnd} implies that \begin{align}
\Phi:L^2(0,T;L^2(\Omega))\rightarrow L^2(0,T;L^2(\Omega))\text{ maps bounded sets on precompact sets}. \label{Pcomp}                                                         \end{align}
Next, we verify that $\Phi$ is closed in $L^2(0,T;L^2(\Omega))$. Consider a sequence $\{ \bar c_{r\varepsilon m}\}\subset L^2(0,T;L^2(\Omega))$ s.t.
\begin{align}
 &{\bar c_{r\varepsilon m}}\underset{m\rightarrow\infty}{\rightarrow} {\bar c_{r\varepsilon}}\qquad \text{in }L^2(0,T;L^2(\Omega)),\label{convbarcm}\\
 \Phi({\bar c_{r\varepsilon m}})=:&{ \crem}\underset{m\rightarrow\infty}{\rightarrow}{\cre}\qquad \text{in }L^2(0,T;L^2(\Omega)).\label{convcm}
\end{align}
We need to check  that
\begin{align}
 \Phi( {\bar c_{r\varepsilon}}) = {\cre}.\nonumber
\end{align}
Due to  
\cref{convbarcm} we have  (by switching to a subsequence, if necessary) that
\begin{align}
 {\bar c_{r\varepsilon m}}\underset{m\rightarrow\infty}{\rightarrow} {\bar c_{r\varepsilon}}\qquad\text{a.e.}\label{convbarcmae}
\end{align}
Further, \cref{Pbndonbnd} and  \cref{convcm} together with the Banach-Alaoglu theorem imply that
\begin{alignat}{3}
 & {{ \crem}\underset{m\rightarrow\infty}{\rightharpoonup}  { \cre}}&&\qquad {\text{in }L^2(0,T;H^1(\Omega))}, \label{convnabbarcm} \\
 &\partial_t { \crem}\underset{m\rightarrow\infty}{\rightharpoonup} \partial_t{ \cre}&&\qquad \text{in }L^2(0,T;(H^1(\Omega))^*). \label{convnabtbarcm}
\end{alignat}
\noindent
By the definition of $\Phi$ we have that  $\bar c_{r\varepsilon m}$ and $\crem $ satisfy: for all $\varphi\in H^1(\Omega)$ and a.a. $t\in(0,T)$ 
\begin{subequations}\label{weakcbarcn}
\begin{align}
   \left<\partial_t {\crem},\varphi\right>_{{(H^1(\Omega))^*,H^1(\Omega)}}=&-\int_{\Omega}{ D_c({ \bar c_{r\varepsilon m}},{ \vrem})}\nabla {\crem}\cdot \nabla \varphi \, dx \nonumber\\
    &+ \int_{\Omega}{ \bar c_{r\varepsilon m}\chi(\bar c_{r\varepsilon m},\vrem )} { G_{\varepsilon}({\cal R}_{r}(\partial_cg(\bar c_{r\varepsilon m},\vrem)\nabla \crem))\cdot\nabla \varphi)}\,dx\nonumber\\
   &+\int_{\Omega}{ \bar c_{r\varepsilon m}\chi(\bar c_{r\varepsilon m},\vrem )}{ G_{\varepsilon}({\cal R}_{r} (\partial_vg(\bar c_{r\varepsilon m},\vrem)\nabla\vrem)) \cdot \nabla \varphi}\nonumber\\
   &+f_c({ \bar c_{r\varepsilon m},\vrem})\varphi\, dx,\label{weakcrbarn}
\end{align}
\begin{align}
 c_{{r\varepsilon m}}(0,\cdot)=c_0\qquad\text{in }L^2(\Omega)
\end{align}
and
\begin{alignat}{3}
  &{\partial_t \vrem=D_v\Delta \vrem+f_v( \bar c_{r\varepsilon m},\vrem)}\qquad&&{ \text{a.e. in }(0,T)\times\Omega},\label{weakvrbarn}\\
  &{ D_v \partial_{\nu} \vrem=0\qquad}&&{\text{a.e. in }(0,T)\times\partial\Omega},\\
&{\vrem}(0,\cdot)=v_0\qquad&&\text{in }H^1(\Omega).  \label{decinin}                                                     \end{alignat}
\end{subequations}
From \cref{apriori1} and \cref{convbarcm} we conclude that the sequence $\{ { \vrem} \}$ is uniformly bounded in $L^2(0,T;H^{2}(\Omega))$ and $\partial_t  \vrem \in L^2(0,T;(L^2(\Omega))$. Hence  
the Lions-Aubin lemma and the Banach-Alaoglu theorem imply that there exists $\vre$ s.t. (after switching to a subsequence, if necessary)
\begin{align}
{\vrem} &\underset{m\rightarrow\infty}{\rightharpoonup} {\vre} \qquad \text{in }L^2(0,T;{ H^2}(\Omega)),\nonumber\\% \label{convvh1} \\
\partial_t {\vrem}&\underset{m\rightarrow\infty}{\rightharpoonup} \partial_t{ \vre}\quad \text{in }L^2(0,T;{ L^2(\Omega)}),\nonumber\\
{\vrem}&\underset{m\rightarrow\infty}{\rightarrow}{ \vre} \qquad \text{in }L^2(0,T;{ H^1}(\Omega)) {\text{ and a.e. in } (0,T) \times \Omega}, \label{convvl2}
\end{align}
and this ${\vre}$ satisfies equation \cref{weakvrbar} for $\bar c_{r\varepsilon }$ as well as the initial and boundary conditions in the required sense.

\noindent
Further, due to \cref{convnabbarcm,convnabtbarcm} we have in the usual way that
\begin{align}
 \crem (t,\cdot) \underset{m \rightarrow \infty}{\rightharpoonup} \cre (t,\cdot)\qquad \text{in } L^2(\Omega)\qquad\text{ for all }t>0.\label{cret}
\end{align}
In particular,
\begin{align}
 \crem (0,\cdot)=c_0,\nonumber
\end{align}
i.e. the initial condition is satisfied.

\noindent
It remains now to pass to the limit in \cref{weakcrbarn}. For this purpose we use the Minty-Browder method. 
To shorten the notation, we introduce 
 for $m \in \N \cup \{\infty \}$
\begin{align*}
&\left< {\cal M}_m(u), \varphi \right>_{{ L^2(0,T;(H^1(\Omega))^*),L^2(0,T;H^1(\Omega))}}\nonumber\\
:=& \int_0^T\int_{\Omega} D_c(\bar c_{r\varepsilon m}, \vrem) \nabla u\cdot\nabla \varphi- G_{\varepsilon}({\cal R}_r(\partial_c g (\bar c_{r\varepsilon m}, \vrem ) \nabla u)) \bar c_{r\varepsilon m} \chi (\bar c_{r\varepsilon m}, \vrem)\cdot\nabla \varphi \, dxdt,%\label{notM}
\\
&\left< f_m, \varphi \right>_{{ L^2(0,T;(H^1(\Omega))^*),L^2(0,T;H^1(\Omega))}}\nonumber\\
:=&\int_0^T\int_{\Omega}{\bar c_{r\varepsilon}\chi(\bar c_{r\varepsilon m},\vrem )}{G_{\varepsilon}({\cal R}_{r} (\partial_vg(\bar c_{r\varepsilon m},\vrem)\nabla\vrem)) \cdot \nabla \psi}+f_c({ \bar c_{r\varepsilon m},\vrem})\psi\, dxdt,%\label{notF}
\end{align*}
where \begin{align}
\bar c_{r\varepsilon \infty} := \bar c_{r\varepsilon},\qquad \bar v_{r\varepsilon \infty} := \bar v_{r\varepsilon}. \nonumber                                                                                                                        \end{align}
% and denoting ${\cal M} := {\cal M}_{\infty}$.
Due to \cref{LemMon}{\it \cref{LemMonii}} and \cref{estmr} %(compare \cref{coeffsmon})
each operator ${\cal M}_m$ is monotone, hemicontinuous, and satisfies 
\begin{align*}
||{\cal M}_{m}(\crem)||_{L^2(0,T;(H^1(\Omega))^*)} \leq \Cr{Mr_}{(r)}\| \crem\|_{L^2(0,T;H^1(\Omega))} \leq \Cl{upcrem}{(r)}.
\end{align*} 
 Consequently, there is $\eta \in L^2(0,T;(H^1(\Omega))^*)$ s.t.
\begin{align} 
{\cal M}_m(\crem) \rightharpoonup \eta \text{ in } L^2(0,T;(H^1(\Omega))^*).\label{konvmm}
\end{align}
Next, from \cref{convbarcmae} and \cref{convvl2} we conclude using the boundedness and continuity of functions $G_{\varepsilon},\nabla g,\nabla f_c$, and $(c,v)\mapsto c\chi(c,v)$ over $\R\times\R^+_0$ and  of operator ${\cal R}_{r}$ in $L^2(\Omega)$  and the dominated convergence theorem that  
\begin{align} 
 f_m\underset{m \rightarrow \infty}{\rightarrow}f_{\infty}\qquad \text{in }L^2(0,T;(H^1(\Omega))^*).\label{konvt2}
\end{align}
 A similar argument  yields 
\begin{align}
  &{\cal M}_m(w)\underset{m \rightarrow \infty}{\rightarrow}{\cal M}_{\infty}(w),\qquad \text{in }L^2(0,T;(H^1(\Omega))^*)\nonumber
\end{align}
so that due to  \cref{convnabbarcm} and the  compensated compactness
\begin{align}
 &\left< {\cal M}_m(w),\crem\right>_{{ L^2(0,T;(H^1(\Omega))^*),L^2(0,T;H^1(\Omega))}}\underset{m \rightarrow \infty}{\rightarrow}\left< {\cal M}_{\infty}(w),\cre\right>_{{ L^2(0,T;(H^1(\Omega))^*),L^2(0,T;H^1(\Omega))}}.\nonumber%\label{konvt2__}
\end{align}
Observe that the weak formulation \cref{weakcrbarn} is equivalent to 
\begin{align}
 \partial_t \crem=  -{\cal M}_m(\crem)+f_m\qquad\text{in }(H^1(\Omega))^*. \label{weakcsimple}
\end{align}
Combining
\cref{convnabtbarcm,konvmm,konvt2}
we can pass to the weak limit in \cref{weakcsimple}
and obtain
\begin{align}
 \partial_t \cre 
 =  -\eta+f_{\infty}\qquad\text{in }(H^1(\Omega))^*. \label{eqetacre}
\end{align}
For $w \in L^2(0,T;H^1(\Omega))$ and $m\in\N$ we have due to the monotonicity of ${\cal M}_m$ that
\begin{align} 
X_m := \left< {\cal M}_m (\crem)-{\cal M}_m (w), \crem-w\right>_{{(H^1(\Omega))^*,H^1(\Omega)}}\geq 0.\label{xmg0}
\end{align}
Moreover, setting $\varphi = \crem$ in \cref{weakcbarcn} and inserting the obtained term into the definition of $X_m$, we conclude that
\begin{align} %\label{xmug}
X_m =& - \left< {\cal M}_m ( \crem), w\right>_{{ L^2(0,T;(H^1(\Omega))^*),L^2(0,T;H^1(\Omega))}} - \left< {\cal M}_m ( w), \crem-w\right>_{{ L^2(0,T;(H^1(\Omega))^*),L^2(0,T;H^1(\Omega))}}\nonumber\\
&+ \frac{1}{2}\| c_0 \|_{L^2(\Omega)}^2-\frac{1}{2}\| \crem (T) \|_{L^2(\Omega)}^2 +\left< f_m, \crem \right>_{{ L^2(0,T;(H^1(\Omega))^*),L^2(0,T;H^1(\Omega))}}.\label{konvmmw}
\end{align}
Combining \cref{cret} for $t=T$,  \cref{konvmm,xmg0,konvmmw,convcm,convnabbarcm}, we obtain
\begin{align*}
0 \leq \limsup_{m \rightarrow \infty} X_m \leq& -  \left< \eta, w\right>_{{ L^2(0,T;(H^1(\Omega))^*),L^2(0,T;H^1(\Omega))}} - \left< {\cal M}_{\infty} (w), \cre-w\right>_{{ L^2(0,T;(H^1(\Omega))^*),L^2(0,T;H^1(\Omega))}}\nonumber\\
&+ \frac{1}{2}\| c_0 \|_{L^2(\Omega)}^2-\frac{1}{2}\| \cre (T) \|_{L^2(\Omega)}^2+\left< f_{\infty}, \cre \right>_{{ L^2(0,T;(H^1(\Omega))^*),L^2(0,T;H^1(\Omega))}}.
\end{align*}
As $\cre$ satisfies \cref{eqetacre}, it follows from the last equation that for all $w \in L^2(0,T;H^1(\Omega))$ it holds that 
\begin{align*}
0 \leq \left< \eta-{\cal M}_{\infty}(w), \cre-w \right>_{{ L^2(0,T;(H^1(\Omega))^*),L^2(0,T;H^1(\Omega))}} .
\end{align*}
% { Es gilt doch $\left< f_{\infty}, \cre \right>_{{(H^1(\Omega))^*,H^1(\Omega)}}$=$\frac{1}{2}\frac{d}{dt}\|\cre \|_{L^2}^2+\left< \eta,\cre\right>_{{(H^1(\Omega))^*,H^1(\Omega)}} $, wie kriegen wir die Terme mit $\frac{1}{2}$ weg?}
% { Es steckt bereits die Integration über $[0,T]$ in $<\cdot,\cdot>$. Ich wollte es bloß nicht explizit hinschreiben (s. die entsprechende Stelle in Section 2). Meinst Du, man sollte es doch tun? } { Aha, danke. Ja, ich finde man sollte es schon machen.}

\noindent
Since ${\cal M}_{\infty}$ is monotone and hemicontinuous, Minty's lemma implies that it is maximal monotone. Consequently, $\eta={\cal M}_{\infty}(\cre)$.  

\noindent
Altogether, we conclude that $(\cre,\vre)$ satisfies \cref{weakcbarc} for $\bar c_{r\varepsilon}$, meaning that  $\Phi( \bar c_{r\varepsilon}) = \cre $ holds, i.e. $\Phi$ is a closed operator. Together with \cref{Pcomp}, this implies that
\begin{align}
\Phi: L^2(0,T; L^2(\Omega)) \rightarrow  L^2(0,T; L^2(\Omega)) \text{ is a compact operator.}
\end{align}  
Since we aim to apply the Leray-Schauder principle \cite[Chapter 6, \S 6.8, Theorem 6.A]{ZeidlerNFA1}, 
it is necessary to consider for $\lambda\in(0,1)$  the system which corresponds to $c_r=\lambda \Phi(c_r)$. The corresponding weak-strong formulation reads:
\begin{subequations}\label{FPweakcbarc}
\begin{align}
   \left<\partial_t \cre,\varphi\right>_{{(H^1(\Omega))^*,H^1(\Omega)}}=&-\int_{\Omega}D_c( c_{r\varepsilon},\vre)\nabla \cre\cdot \nabla \varphi \, dx\nonumber\\
   &+ \int_{\Omega} c_{r\varepsilon}\chi( c_{r\varepsilon},\vre)\lambda G_{\varepsilon}(\lambda^{-1}{\cal R}_{r}(\partial_cg( c_{r\varepsilon},\vre)\nabla \cre))\cdot  \nabla \varphi\,dx\nonumber\\
   &+\lambda\int_{\Omega}G_{\varepsilon}({\cal R}_{r} (\partial_vg( c_{r\varepsilon},\vre)\nabla\vre)) \cdot  c_{r\varepsilon}\chi( c_{r\varepsilon},\vre)\nabla \varphi+f_c( c_{r\varepsilon},\vre)\varphi\,\  dx,\label{FPweakcrbar}
  \end{align}
  \begin{align}
 c_{{r\varepsilon}}(0,\cdot)=\lambda c_0\qquad\text{in }L^2(\Omega)
\end{align}
and
 \begin{alignat}{3}
  &{\partial_t \vre=D_v\Delta \vre+f_v( c_{r\varepsilon},\vre)}\qquad &&{ \text{a.e. in }(0,T)\times\Omega},\label{FPweakvrbar}\\
  &{D_v \partial_{\nu} \vre=0} \qquad &&{ \text{a.e. in }(0,T)\times\partial\Omega},\\
&{\vre}(0,\cdot)=v_0\qquad &&\text{in }H^1(\Omega).  \label{FPdecini}                                                   \end{alignat}
\end{subequations}

\noindent
Taking $\varphi := \cre$ in \cref{FPweakcbarc} and estimating the right-hand side by using \cref{Assump1,Assump3}{\it\cref{Assump3a}}, the H\"older inequality, and the fact that $|G_{\varepsilon}(x)| \leq |x|$, we obtain that
\begin{align*}
&\frac{1}{2} \frac{d}{dt} ||\cre ||^2_{L^2(\Omega)} \nonumber\\
\leq&-{ \Cr{Dcmin}} \Cr{Cgstr}{ (\|\mathcal{R}_r \|)} \left\| \nabla { \cre }\right\|_{ (L^2(\Omega))^n}^ 2\\
&+ \lambda \left(\Cr{Q1} { ||\partial_vg||_{L^{\infty}(\R_0^+\times \R_0^+)}} \|\mathcal{R}_r\|_{L((L^2(\Omega))^n)} \left\| \nabla { \cre }\right\|_{(L^2(\Omega))^n} ||\nabla { \vre } ||_{(L^2(\Omega))^n}
{ + { \|\partial_c f_c\|_{L^{\infty}(\R_0^+\times \R_0^+)}} ||{ \cre }||^2_{L^2(\Omega)}} \right) \nonumber \\
\leq&-{ \Cr{Dcmin}} \Cr{Cgstr}{ (\|\mathcal{R}_r \|)} \left\| \nabla { \cre }\right\|_{(L^2(\Omega))^n}^2\\
&+ \Cr{Q1} { ||\partial_vg||_{L^{\infty}(\R_0^+\times \R_0^+)}} \|\mathcal{R}_r\|_{L((L^2(\Omega))^n)} \left\|\nabla { \cre }\right\|_{(L^2(\Omega))^n} ||\nabla { \vre } ||_{(L^2(\Omega))^n} 
{ + { \|\partial_c f_c\|_{L^{\infty}(\R_0^+\times \R_0^+)}}||{ \vre }||^2_{L^2(\Omega)}} 
\end{align*}
holds for a.e. $t\in (0,T)$. Further, performing estimates similar to the proof of  \cref{DthmmainEx} below and using \cref{apriori1}, we conclude that the set
\begin{align}
\left\{c_r \in L^2(0,T;L^2(\Omega)): \, c_r = \lambda \Phi(c_r) \text{ for } \lambda \in (0,1)\right\}\nonumber
\end{align}
is uniformly bounded. Consequently, for all $\varepsilon \in (0,1)$ the Leray-Schauder principle implies that $\Phi$ has a fixed point $\cre$,  which together with the corresponding $\vre$, satisfies \cref{ApprProto} in the weak-strong sense on the interval $[0,T]$. Since $T>0$ was arbitrary, the standard prolongation argument yields the existence of a global solution.

\noindent
It remains to check that $\cre$  is nonnegative. Taking $\varphi :=-(\cre )_- =\min \{\cre,0\}$ in \cref{weakapprcr} and using $f_c(0,\cdot)\equiv0$, the boundedness of $G_{\varepsilon},D_c,\partial_c f_c$, and $(c,v)\mapsto c\chi(c,v)$, along with the H\"older and Young inequalities, yields
\begin{align*}
   &\frac{1}{2}\frac{d}{dt}\|(\cre)_-\|^2_{L^2(\Omega)}\\
=&-\int_{\Omega} D_c(-(\cre)_-,\vre)\left|\nabla (\cre)_- \right|^2 \, dx  - \int_{\Omega} G_{\varepsilon}({\cal R}_{r}(\partial_cg(\cre,\vre)\nabla \cre))\cdot (\cre)_-\chi(-(\cre)_-,\vre)\nabla (\cre)_-\,dx\nonumber\\
   &-\int_{\Omega}G_{\varepsilon}({\cal R}_{r} (\partial_vg(\cre,\vre)\nabla\vre)) \cdot (\cre)_- \chi( -(\cre)_-,\vre)\nabla (\cre)_- \, dx+ \int _\Om f_c( -(\cre)_-,\vre)(\cre)_-\, dx\\
\leq&-\Cr{Dcmin} \|\nabla (\cre)_- \|_{(L^2(\Omega))^n}^2 + \frac{2}{\varepsilon}{\Cr{Q1}} \|(\cre)_-\|_{L^2(\Omega)} \|\nabla (\cre)_-\|_{(L^2(\Omega))^n} + \| \partial_c f_c\|_{L^{\infty}(\R_0^+ \times \R_0^+)} \| (\cre)_-\|_{L^2(\Omega)}^2\\
\leq& \Cl{conscm} \| (\cre)_-\|_{L^2(\Omega)}^2.
\end{align*}
Since $\cre(0,\cdot)=c_0 \geq 0$,
 the Gronwall inequality implies that $(\cre)_{-} = 0$, i.e. that $\cre\geq0$.
\end{proof}
\begin{Remark}
 Observe that $\cre$ cannot be replaced by $-(\cre)_-$ inside the nonlocal operator. This is why we introduced  the flux-limitation.
\end{Remark}

\noindent
Now we are ready to prove \cref{thmmainEx}.
\begin{proof} {\it (of \cref{thmmainEx}). }
We start with the case $$D_v >0.$$ \cref{thmexr} gives the existence of solutions $(\cre, \vre)$ to \cref{ApprProto}. Setting $\varphi = \cre$ in \cref{weakapprcr},  using the facts that $f_c$ is Lipschitz and $|G_{\varepsilon}(x)| \leq |x|$, we can estimate similarly to \cref{DthmmainEx} below and obtain upper bounds of the form \cref{boundcr}-\cref{boundintc1s},  which are independent from $\varepsilon$ (with $p=q=2$ there). Applying the Lions-Aubin lemma and the Banach-Alaoglu theorem, we conclude the existence of a pair of nonnegative functions $c_r$ and $v_r$ having the regularity stated in  \cref{DefSol} and such that for a sequence $\varepsilon_m \underset{m \rightarrow \infty}{\rightarrow} 0$ it holds that
\begin{align}
\crel &\underset{m \rightarrow \infty}{\rightarrow} c_r \text{ in } L^2(0,T;L^2(\Omega)) \text{ and a.e. in } (0,T) \times \Omega , \label{crell2ae}\\
\vrel &\underset{m \rightarrow \infty}{\rightarrow} v_r \text{ in } L^2(0,T;H^1(\Omega)) \text{ and a.e. in } (0,T) \times \Omega \label{vrell2ae}, \\
\crel &\underset{m \rightarrow \infty}{\rightharpoonup} c_r \text{ in } L^2(0,T;H^1(\Omega)). \label{crelh1}
\end{align}
\noindent
Consider an arbitrary measurable set $E \subset (0,T) \times \Omega$.  Using $G_{\varepsilon}(x)-x = -\varepsilon \frac{x|x|}{1+\varepsilon |x|}$, we can estimate for every component $i \in \{1,\dots,n\}$:
\begin{align*}
&\left| \int_E \left( { G}_{\varepsilon _m}({\cal R}_{r}(\partial_cg(\crel,\vrel)\nabla \crel))-{\cal R}_{r}(\partial_cg(\crel,\vrel)\nabla \crel) \right)_i  \, dx \, dt \right|\\
\leq& \varepsilon_m \int_0^T \int_{\Omega} \left| {\cal R}_{r}(\partial_cg(\crel,\vrel)\nabla \crel \right|^2 \, dx \, dt \\
\leq& \varepsilon_m \| {\cal R}_r \|_{L((L^2(\Omega))^n)}  \Cr{Q2} \|\nabla \crel  \|_{L^2(0,T;(L^2(\Omega))^n)}^2,
\end{align*}
where the last term tends to 0 as $\varepsilon _m \underset{m \rightarrow \infty}{\rightarrow} 0$. As the term inside the integral is moreover bounded in $L^2(0,T;L^2(\Omega))$ by a constant independent from $\varepsilon _m$, we conclude  by using a result from \cite[p. 6]{Evans1990}  that
\begin{align*}
{ G}_{\varepsilon _m}({\cal R}_{r}(\partial_cg(\crel,\vrel)\nabla \crel))-{\cal R}_{r}(\partial_cg(\crel,\vrel)\nabla \crel) \underset{m \rightarrow \infty}{\rightharpoonup} 0 \text{ in } L^2(0,T;(L^2(\Omega))^n).
\end{align*}
From this and the boundedness of $\|\nabla \crel \|_{L^2(0,T;(L^2(\Omega))^n)}$, \cref{crell2ae}-\cref{crelh1}, \cref{TWellDef} or \ref{SWellDef} {\it (i)} and {\it (ii)}, respectively, the fact that $|G_{\varepsilon}(x)| \leq |x|$, the continuity of $\partial_c g, \chi$, \cref{cchi}, \cref{pcgb}, compensated compactness, the dominated convergence theorem, and the H\"older inequality, we obtain that for all $\psi \in L^2(0,T;H^1(\Omega))$ it holds  that
\begin{align*}
&\int_0^T \int_{\Omega} { G}_{\varepsilon _m}({\cal R}_{r}(\partial_cg(\crel,\vrel)\nabla \crel))\cdot\crel\chi(\crel,\vrel) \nabla \psi\,dx \, dt\\
\underset{m \rightarrow \infty}{\rightarrow}& \int_0^T \int_{\Omega} {\cal R}_{r}(\partial_cg(c_r,v_r)\nabla c_r)\cdot c_r\chi(c_r,v_r) \nabla \psi\,dx \, dt.
\end{align*}
The convergence to the remaining terms in \cref{weakcr} and the rest of \cref{weakcvr} can be obtained in a way either completely analogous or very similar to the corresponding parts of the proof of   \cref{thmexr}.
% Applying \cref{TWellDef} or \ref{SWellDef} {\it (ii)}, respectively, we conclude that $(c_r, v_r)$ satisfies \cref{weakcvr}.

\noindent
In order to prove existence for  the case $$D_v = 0$$ consider a family of solutions  $(c_{rD_v}, v_{rD_v})$ corresponding to $D_v \in (0,1)$. Estimating similarly to the proof of \cref{DthmmainEx} below and performing a standard limit procedure based on the Banach-Alaoglu theorem, the dominated convergence theorem, the Lions lemma \cite[Lemma 1.3]{Lions}, and the compensated compactness, one readily obtains a solution $(c_{r0},v_{r0})$ for $D_v = 0$ in the sense of \cref{DefSol}. Observe that this time the gradient of $v$-component enters linearly, so that no strong convergence is required. We omit further details.
\end{proof}

\subsection{Global existence of solutions to  \texorpdfstring{\cref{nlProto}}{}:  the case of \texorpdfstring{$f_c$}{} dissipative}\label{apriori}

\noindent
In this Subsection we provide an extension of the existence \cref{thmmainEx} from \cref{SecEx}:
\begin{Theorem} \label{DthmmainEx}
Let  \cref{Assf,Assump1,Assump2}{\it\cref{Assump2b}} hold and let $r$ satisfy \cref{Assump3}{\it\cref{Assump3a}}. Set  \begin{align}
q := \min\left\{2, \frac{s+1}{s}\right\}.\qquad q^*:=\frac{q}{q-1}.  \label{qq}                                                                                                                                                                                                                        \end{align}
Then there exists a global weak-strong solution to \cref{nlProto} in terms of \cref{DefSol}, with \\
$\partial_t c_r\in L^q(0,T;(W^{1,q^*}(\Omega))^*)$ and satisfying the following estimates: For all $T>0$
 \begin{align}
||c_r||_{L^{\infty}(0,T;L^2(\Omega))} &\leq \Cl{Curnew}(T,\|\mathcal{R}_r\|_{L(\Ltwon)}), \label{boundcr}\\
||\nabla c_r||_{L^2(0,T;(L^2(\Omega))^n)} &\leq \Cr{Curnew}(T,\|\mathcal{R}_r\|_{L(\Ltwon)}),\label{boundncr}\\
||\partial_t c_r||_{L^q(0,T;(W^{1,q^*}(\Omega))^*)} &\leq \Cr{Curnew}(T,\|\mathcal{R}_r\|_{L(\Ltwon)}),\label{bndptcr}\\
||v_r||_{L^{\infty}(0,T;L^2(\Omega))} &\leq \Cr{Curnew}(T,\|\mathcal{R}_r\|_{L(\Ltwon)}),\label{estvr}\\
||\nabla v_r||_{L^{\infty}(0,T;(L^2(\Omega))^n)} &\leq \Cr{Curnew}(T,\|\mathcal{R}_r\|_{L(\Ltwon)}),\label{boundnvr}\\
||\partial_t v_r||_{L^2(0,T;L^2(\Omega))} &\leq \Cr{Curnew}(T,\|\mathcal{R}_r\|_{L(\Ltwon)}),\label{bounddtvr}\\
\| f_c(c_r, v_r)\|_{L^q(0,T;L^q(\Omega))}  &\leq \Cr{Curnew}(T,\|\mathcal{R}_r\|_{L(\Ltwon)}), \label{boundintc1s}\\
\| f_v(c_r, v_r)\|_{L^2(0,T;L^2(\Omega))} &\leq \Cr{Curnew}(T,\|\mathcal{R}_r\|_{L(\Ltwon)}).\label{bnfintf_v}
\end{align}
\end{Theorem}

\begin{proof}
 For $k \in \N$  set 
 $$f_{ck}(c,v) := f_c(c, v)\eta_k(c),$$ 
 where $\eta_k$ is a cut-off function:
\begin{align}
\eta_k \in C_0^{\infty}(B_k(0)) \quad\text{with}\quad \eta_k \equiv 1 \quad \text{ in } B_{k-1}(0) \quad\text{and}\quad 0 \leq \eta_k \leq 1.                                                                                                                                                                             \end{align}
Since $f_{ck}$ is Lipschitz, \cref{thmmainEx} implies the existence of a solution $(c_{rk},v_{rk})$ in terms of \cref{DefSol} with $\partial_t c_{rk}\in L^2(0,T;(H^1(\Omega))^*)$, which corresponds to $f_c=f_{{ c}k}$. Our next aim is to prove that $(c_{rk},v_{rk})$ satisfies the same bounds as in the statement of the Theorem with some constant $\Cr{Curnew}(T,\|\mathcal{R}_r\|_{L(\Ltwon)})$ which does not depend upon $k$.

\noindent
Set $$\Cl{NR}(\|\mathcal{R}_r\|) := \|\mathcal{R}_r\|_{L((L^2(\Omega))^n)}.$$ Taking   $\varphi:=c_{rk}$ in \cref{weakcr}  written for $c_{rk}$  and using  \cref{Assump1},  {\it\ref{Assump2}{\it\cref{Assump2b}}}, {\it\ref{Assump3}{\it\cref{Assump3a}}} and the H\"older and Young inequalities, we compute 
\begin{align}
&\frac{1}{2}\frac{d}{dt}\|c_{rk}\|_{L^2(\Omega)}^2\nonumber\\
=&\int_{\Omega}\Big(-\left(D_c(c_{rk},v_{rk})\nabla c_{rk
  }- c_{rk}\chi(c_{rk},v_{rk}) {\cal R}_{r}(\nabla g(c_{rk},v_{rk}))\right)\cdot \nabla c_{rk}+ c_{rk}f_{ck}(c_{rk},v_{rk})\Big )\,dx\nonumber\\
\leq&-\Cr{Dcmin}\left\|\nabla c_{rk}\right\|_{(L^2(\Omega))^n}^2+ \Cr{Q1}\left\|\nabla c_{rk}\right\|_{L^2(\Omega)} \left\| {\cal R}_{r}(\nabla g(c_{rk},v_{rk}))\right\|_{(L^2(\Omega))^n}+\int _{\Om }( \Cr{Dis1}-\Cr{Dis2} c_{rk}^{1+s}) \eta_k(c_{rk}) \, dx \nonumber\\
\leq&-\Cr{Dcmin}\left\|\nabla c_{rk}\right\|_{(L^2(\Omega))^n}^2 + \Cr{Q1} \Cr{NR}( \|\mathcal{R}_r\|) \left\|\nabla c_{rk}\right\|_{(L^2(\Omega))^n} \left\|\nabla g(c_{rk},v_{rk})\right\|_{(L^2(\Omega))^n} + \Cl{Dis1mu}-\Cr{Dis2} \int_{\Omega} c_{rk}^{1+s} \eta_k(c_{rk}) \, dx  \nonumber\\
\leq&-\Cr{Dcmin}\left\|\nabla c_{rk}\right\|_{(L^2(\Omega))^n}^2 + \Cr{Q1} \Cr{NR}(\|\mathcal{R}_r\|) \left\|\nabla c_{rk}\right\|_{(L^2(\Omega))^n} \left\|\partial_cg(c_{rk},v_{rk}) \nabla c_{rk} \right\|_{(L^2(\Omega))^n}\nonumber\\ 
&+  \Cr{Q1} \Cr{NR}( \|\mathcal{R}_r\|) \left\|\nabla c_{rk}\right\|_{(L^2(\Omega))^n} \left\|\partial_vg(c_{rk},v_{rk}) \nabla v_{rk} \right\|_{(L^2(\Omega))^n} + \Cr{Dis1mu}-\Cr{Dis2}  \int_{\Omega} c_{rk}^{1+s} \eta_k(c_{rk}) \, dx \nonumber\\
\leq&-\Cr{Dcmin} \Cr{Cgstr}(\|\mathcal{R}_r\|) \left\|\nabla c_{rk}\right\|_{(L^2(\Omega))^n}^2 +  \Cr{Q1} \Cr{NR}( \|\mathcal{R}_r\|) \left\| \partial_vg \right\|_{L^{\infty}(\R_0^+ \times \R_0^+)} \left\|\nabla c_{rk}\right\|_{(L^2(\Omega))^n} \left\| \nabla v_{rk} \right\|_{(L^2(\Omega))^n} \nonumber\\ 
&+ \Cr{Dis1mu}-\Cr{Dis2} \int_{\Omega} c_{rk}^{1+s} \eta_k(c_{rk}) \, dx \nonumber\\
\leq& -2\Cl{Z2}(\|\mathcal{R}_r\|) \left\|\nabla c_{rk}\right\|_{(L^2(\Omega))^n}^2 +\Cl{C15}(\|\mathcal{R}_r\|)\left\|\nabla v_{rk}\right\|_{(L^2(\Omega))^n}^2 + \Cr{Dis1mu}-\Cr{Dis2} \int_{\Omega} c_{rk}^{1+s} \eta_k(c_{rk})\,dx \label{estc1}.
\end{align}
Next, we estimate $v_{rk}$. If $D_v>0$, then standard theory \cite{LSU} yields that for all $0<t\leq T$

\begin{align}
   \|v_{rk}\|_{L^{\infty}(0,{ t};{ H^1}(\Omega))}^2+\|v_{rk}\|_{L^2(0,{ t};H^2(\Omega))}^2+\|\partial_t v_{rk}\|_{L^2(0,{ t};L^2(\Omega))}^2
   \leq &{\Cl{C13_}(T)}\left(\|v_0\|_{{ H^1}(\Omega)}^2+\|c_{rk}\|_{L^2(0,{ t};{ L^2}(\Omega))}^2\right).\label{apriorivrk1}
  \end{align}
Here and further in the proof we omit the dependence of constants upon $D_v$. 
If $D_v=0$, then we get the ODE
\begin{align}
\partial_t v_{rk}=&f_v(c_{rk},v_{rk}). \label{vrkODE}
\end{align}
Hence, the assumptions on $f_v$ and the solution components together with the chain rule imply that 
\begin{align}
\partial_t v_{rk}\in& L^2(0,T;H^1(\Omega)).\nonumber
\end{align}
 Computing the gradient on both sides of \cref{vrkODE}, multiplying by $\nabla v_{rk}$ throughout, integrating over $\Omega$, and using {\it  \cref{Assump1}} and the Young inequality, we obtain that
\begin{align}
 \frac{1}{2}\frac{d}{dt} \|\nabla v_{rk}\|_{(L^2(\Omega))^n}^2 
 =&\int_{\Omega}\left(\partial_vf_v(c_{rk},v_{rk})|\nabla v_{rk}|^2+\partial_cf_v(c_{rk},v_{rk})\nabla c_{rk}\cdot\nabla v_{rk}\right)\,dx\nonumber\\
\leq& \left\| \partial_v f_v\right\|_{L^{\infty}(\R_0^+ \times \R_0^+)} \|\nabla v_{rk}\|_{(L^2(\Omega))^n}^2 + \left\| \partial_c f_v\right\|_{L^{\infty}(\R_0^+ \times \R_0^+)} \|\nabla c_{rk}\|_{(L^2(\Omega))^n} \|\nabla v_{rk}\|_{(L^2(\Omega))^n} \nonumber\\
\leq& \Cl{dvfvy}\|\nabla v_{rk}\|_{(L^2(\Omega))^n}^2+\Cl{C17}\left\|\nabla c_{rk}\right\|_{(L^2(\Omega))^n}^2 .\label{estv1}
\end{align}
Applying the Gronwall inequality to \cref{estv1} yields
\begin{align}
   \|\nabla v_{rk}\|_{L^{\infty}(0,{ t};{ L^2}(\Omega))}^2
   \leq &{\Cl{C13_1}(T)}\left(\|\nabla v_0\|_{L^2(\Omega)}^2+\|\nabla c_{rk}\|_{L^2(0,{ t};{ L^2}(\Omega))}^2\right).\label{apriorinvrk}
  \end{align}
Multiplying \cref{vrkODE} by $v_{rk}$ we obtain in a similar fashion that
\begin{align}
   \|v_{rk}\|_{L^{\infty}(0,{ t};{ L^2}(\Omega))}^2
   \leq &{\Cr{C13_1}(T)}\left(\|v_0\|_{L^2(\Omega)}^2+\|c_{rk}\|_{L^2(0,{ t};{ L^2}(\Omega))}^2\right).\label{apriorivrk}
  \end{align}
  Adding \cref{apriorinvrk,apriorivrk} together yields
\begin{align}
   \|v_{rk}\|_{L^{\infty}(0,{ t};{ H^1}(\Omega))}^2
   \leq &{\Cr{C13_1}(T)}\left(\|v_0\|_{H^1(\Omega)}^2+\|c_{rk}\|_{L^2(0,{ t};{ H^1}(\Omega))}^2\right).\label{aprioriH1vrk}
  \end{align} 
  Estimating the right-hand side of \cref{vrkODE}  by using \cref{apriorivrk} implies 
  \begin{align}
   \|\partial_t v_{rk}\|_{L^2(0,T;L^2(\Omega))}^2\leq &{\Cr{C13_1}(T)}\left(\|v_0\|_{L^2(\Omega)}^2+\|c_{rk}\|_{L^2(0,T;{ L^2}(\Omega))}^2\right).\label{aprioriptvrk}
  \end{align} 
 Further, combining \cref{estc1} with \cref{apriorivrk1} if $D_v>0$ and with \cref{aprioriH1vrk} if $D_v=0$ and using the Gronwall inequality yields for $c_{rk}$ the same estimates as \cref{boundcr,boundncr}, and the estimate
 \begin{align} 
\int_0^T \int_{\Omega} c_{rk}^{1+s} \eta_k(c_{rk}) \, dx dt\leq \Cl{C18}(T,\|\mathcal{R}_r\|).\label{intcrks}
\end{align}
 From \cref{ineqNem,intcrks}, the embedding of Lebesgue spaces, and $\eta_k \in [0,1]$ we conclude that
\begin{align*}
\|f_{ck}(c_{rk}, v_{rk})\|_{L^q(0,T;L^q(\Omega))} \leq& \Cl{H1}(T) + \Cl{H2}||c_{rk}^s\eta_k(c_{rk})||_{L^{\frac{s+1}{s}}(0,T;L^{\frac{s+1}{s}}(\Omega))}\\
 \leq& \Cr{H1}(T) + \Cr{H2} \left( \int_0^T \int_{\Omega} c_{rk}^{1+s} \eta_k(c_{rk}) \, dx \, dt \right)^{\frac{s}{s+1}}\nonumber\\
 \leq &\C(\|\mathcal{R}_r\|,T).
\end{align*}
so that \cref{boundintc1s} holds for $f_{ck}(c_{rk},v_{rk})$. Combining \cref{boundcr,boundncr} for $c_{rk}$ with \cref{apriorivrk1} or \cref{aprioriH1vrk,aprioriptvrk} (depending on the sign of $D_v$) and using the equation for $v_{rk}$ yields  such bounds as \cref{estvr}-\cref{bounddtvr} and \cref{bnfintf_v} for $c_{rk}$ and $v_{rk}$. Finally, combining \cref{Assump1} with bounds on $\nabla c_{rk},\nabla v_{rk}$, and $f_{ck}(c_{rk},v_{rk})$, the weak formulation \cref{weakcr}, and estimating in a standard way yields \cref{bndptcr} for $\partial_t c_{rk}$.

\noindent
Since $(c_{rk},v_{rk})$ satisfy \cref{boundcr}-\cref{bnfintf_v} uniformly in  $k$, a standard limit procedure based on the Banach-Alaoglu theorem, the dominated convergence theorem, the Lions lemma, and the compensated compactness yields the existence of a weak-strong solution $(c_{r},v_{r})$ to \cref{weakcvr} which satisfies \cref{boundcr}-\cref{bnfintf_v}.

\end{proof}

\subsection{Limiting behaviour of the nonlocal model \texorpdfstring{\cref{nlProto}}{} as \texorpdfstring{$r\rightarrow0$}{}}\label{SecConv}
In this Subsection we finally prove our main result concerning convergence for $r\to 0$.
\begin{proof} (of \cref{limitr})
 Due to \cref{condC*} and  \cref{TWellDef} \eqref{PropTiv} or {\it\ref{SWellDef}} \eqref{PropTiv}, respectively, there exists a sequence $r_m\rightarrow0$ as $m\rightarrow\infty$ such that 
\begin{align*}
 \underset{m\in\N}{\sup}\left\|{\cal R}_{r_m}\right\|_{L((L^2(\Omega))^n)} < \frac{1}{\Cr{Cgst}}.
\end{align*}
Since  for each such $r_m$  the \cref{Assump3}{\it\cref{Assump3a}}  are satisfied, \cref{DthmmainEx} is applicable and yields the existence of solutions  $(c_{r_m},v_{r_m})$ which satisfy \cref{boundcr}-\cref{bnfintf_v}.   
 Replacing $\|\mathcal{R}_r \|$ by $\Cr{Cgst}$ in $\Cr{Curnew}(T,\|\mathcal{R}_r\|_{L(\Ltwon)})$ makes the constant in \cref{boundcr}-\cref{bnfintf_v} independent of $m$.
 Using the Lions-Aubin lemma and the Banach-Alaoglu theorem we conclude (by possibly switching to a subsequence) that
  \begin{alignat}{3}
 &c_{r_m} \underset{m \rightarrow \infty}{\rightarrow} c,\ \ v_{r_m} \underset{m \rightarrow \infty}{\rightarrow} v  \qquad &&\text{in } L^2(0,T;L^2(\Om)), \quad \text{a.e. in } (0,T)\times \Omega\label{convcmae} \\
&c_{r_m} \underset{m \rightarrow \infty}{\rightharpoonup} c,\ \  v_{r_m} \underset{m \rightarrow \infty}{\rightharpoonup} v \qquad &&\text{in } L^2(0,T;H^1(\Om)).\label{convcmh1}
\end{alignat}
Using standard arguments based on the Banach-Alaoglu theorem, the dominated convergence theorem, the Lions lemma, and assumptions on $\chi$ and $g$ we conclude from \cref{convcmae,convcmh1} that
\begin{alignat}{3}
 &c_{r_m}\chi (c_{r_m}, v_{r_m})\underset{m \rightarrow \infty}{\rightarrow} c\chi (c,v)  \qquad &&\text{in } L^2(0,T;L^2(\Om)),\label{convchi}\\
 &g(c_{r_m},v_{r_m}) \underset{m \rightarrow \infty}{\rightharpoonup} g(c,v) \qquad &&\text{in } L^2(0,T;H^1(\Om)).\label{convg}
\end{alignat}
Observe that for any $\psi \in L^{\infty}(0,T; W^{1, \infty}(\Omega))$ the following estimate holds:
\begin{align}
&\int_0^T \int_{\Omega} \left| \mathcal{R}_{r_m} (c_{r_m}\chi (c_{r_m}, v_{r_m}) \nabla  \psi )-c \chi(c,v) \nabla \psi\right|^2 \, dx \, dt\\
\leq& 2\left( \int_0^T \int_{\Omega} \left| \mathcal{R}_{r_m} (c_{r_m}\chi (c_{r_m}, v_{r_m}) \nabla \psi )-\mathcal{R}_{r_m} (c\chi (c, v) \nabla\psi ) \right|^2 \, dx \, dt\right.\nonumber\\
 &\left.+ \int_0^T \int_{\Omega} \left|\mathcal{R}_{r_m} (c\chi (c, v) \nabla \psi )-c \chi(c,v) \nabla \psi\right|^2 \, dx \, dt \right).\label{estexr}
\end{align}
Now, using \eqref{convchi}  together with \cref{TWellDef}{\it \cref{PropTi}} and {\it \cref{PropTiii}} { and \cref{A0T0}} or \cref{SWellDef}{\it \cref{PropSi}} and {\it \cref{PropSiii}} { and \cref{D0S0}}, respectively, we conclude that the right hand side of \cref{estexr} tends to zero, hence
\begin{align} 
{\cal R}_{r_m} (c_{r_m}\chi(c_{r_m},v_{r_m})\nabla \psi)  \underset{m \rightarrow \infty}{\rightarrow } c\chi(c,v)\nabla \psi \qquad \text{in } L^2(0,T;(L^2(\Omega))^n).\label{convrm}
\end{align}
Thus, using \cref{TWellDef}{\it \cref{PropTii}} or \cref{SWellDef}{\it \cref{PropSii}}, respectively,  and compensated compactness, we obtain from \cref{convg,convrm} that
 \begin{align}
  \int_0^T\int_{\Omega}c_{r_m}\chi(c_{r_m},v_{r_m}) {\cal R}_{r_m}(\nabla g(c_{r_m},v_{r_m}))\cdot \nabla\psi\,dx\, dt=&\int_0^T\int_{\Omega} \nabla g(c_{r_m},v_{r_m})\cdot {\cal R}_{r_m}(c_{r_m}\chi(c_{r_m},v_{r_m})\nabla\psi)\,dxdt \nonumber  \\
  \underset{m \rightarrow \infty}{\rightarrow }& \int_0^T\int_{\Omega} \nabla g(c,v)\cdot c\chi(c,v)\nabla\psi\,dxdt.\nonumber%\label{convpr} 
 \end{align}
\noindent 
The convergence in the remaining terms, equations, and conditions follows  by means of a standard limit procedure based on the Banach-Alaoglu theorem, the dominated convergence theorem, the Lions lemma, and the compensated compactness. We omit these details.

\end{proof} 

% %%%%%%%%%%%Numerics%%%%%%%%%%%
% \noindent
% The numerical simulations suggest...
% Comparison with previous adhesion models... Not just w.r.t. blow-up, but maybe also some comments about previous choices of coefficients (our models extend previous settings, in the sense that...), the 'new' nonlocal model class with the approximated gradient... 

\section{Numerical simulations in 1D}\label{sec:numerics}

We perform numerical simulations to investigate on the one hand the effect of differences between hitherto choices of nonlocal operators and our novel ones proposed in \cref{AverOper}, and on the other hand convergence between nonlocal and local formulations. 
For compactness, our current study restricts to the prototypical nonlocal model for cellular adhesion \cref{nlHapto}, its reformulation as \cref{nlProto}, and the corresponding local model \cref{lProto}. Thus, for \cref{nlProto} we take the operator form $\mathcal R_r=\mathcal T_r$, with $\mathcal T_r$ as in \cref{IrTr}. These models can be interpreted in the context of a population of cells
invading an adhesion-laden ECM/tissue environment and, with this in mind, we initially concentrate 
cells at the centre of a one-dimensional domain $\Omega = [0,L]$ and impose an initially 
homogeneous ECM. Specifically, we set for the ECM
\begin{equation} \label{iniv0}
v_0(x) = 1,\quad x \in \Om
\end{equation}
and consider for the cell population a Gaussian-shaped aggregate
\begin{equation} \label{c0}
c_0(x)=\exp \left(-\alpha (x-x_c)^2 \right),\quad  x \in \Om,
\end{equation}
where we set $x_c = L/2$ or $x_c = 0$.

\noindent
The numerical scheme follows that described in  \cite{Alf-num09}, which we refer to for 
details. Briefly, a Method of Lines approach is invoked whereby equations are first discretised 
in space (in conservative form, via a finite volume method) to yield a high-dimensional system of 
ODEs, which are subsequently integrated in time. Discretisation of advective terms follows a third 
order upwinding scheme, augmented by flux limiting to preserve positivity of solutions and the 
resulting scheme is (approximately) 
second-order accurate in space. Time integration has been performed with standard Matlab ODE solvers:
our default is ``ode45'' with absolute and relative error tolerances set at $10^{-6}$, but simulations have been compared for varying space discretisation step, ODE solver, and error tolerances. 
To measure the difference between two distinct solutions over time we define a 
distance function as follows:
\[
d(u_1(x,t),u_2(x,t)) (t)= \avint_{\Om} \left| u_1(x,t) - u_2(x,t)\right| dx\,,
\]
where $u_1$ and $u_2$ denote the two solutions that are being compared.

\subsection{Comparison of nonlocal operator representations}\label{SecNum1}

We first explore the correspondence between forms of nonlocal operator representation:
we choose the prototypical nonlocal model for cell/matrix adhesion \cref{nlHapto} and its
reformulation \cref{nlProto}, therefore taking for the latter the operator 
form $\mathcal R_r=\mathcal T_r$ with $\mathcal T_r$ as in \cref{IrTr}. In what follows, 
solutions to \cref{nlHapto} are denoted $c_{A}$ and $v_{A}$ and those for
\cref{nlProto} denoted $c_{T}$ and $v_{T}$. For simplicity we restrict in this section to 
a minimalist formulation in which $D_c =$ constant, $\chi = 1$, $f_c = 0$. Cell-matrix interactions 
are defined by $g(c,v) = S_{cc} c + S_{cv} v$ and $f_v(c,v) = -\mu c v$, where $S_{cc}$ and $S_{cv}$ 
respectively represent cell-to-cell and cell-to-matrix adhesion strengths and $f_v$ 
simplistically describes (direct) proteolytic degradation of matrix by cells 
parametrised by degradation rate $\mu$. 
\begin{figure}
\begin{center}
\includegraphics[width=\textwidth]{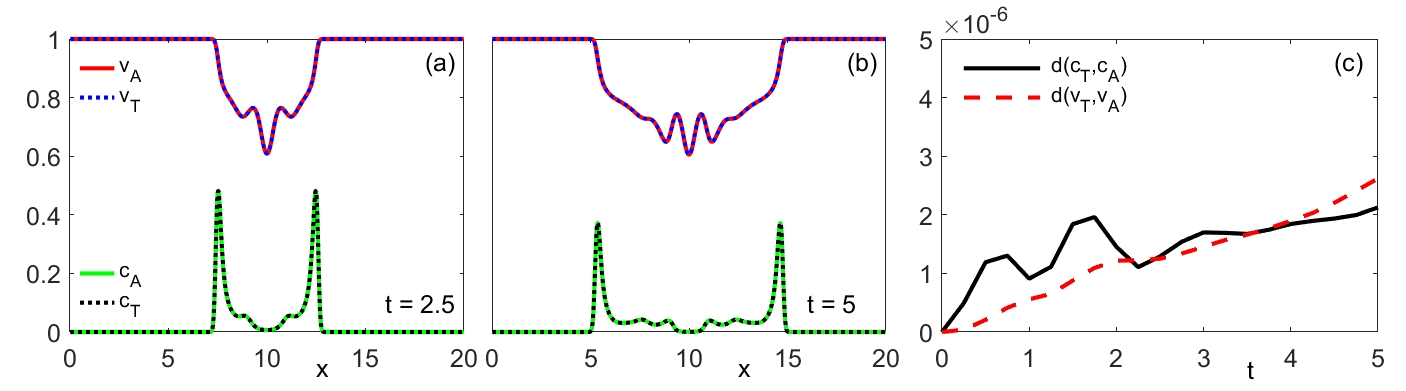}
\includegraphics[width=\textwidth]{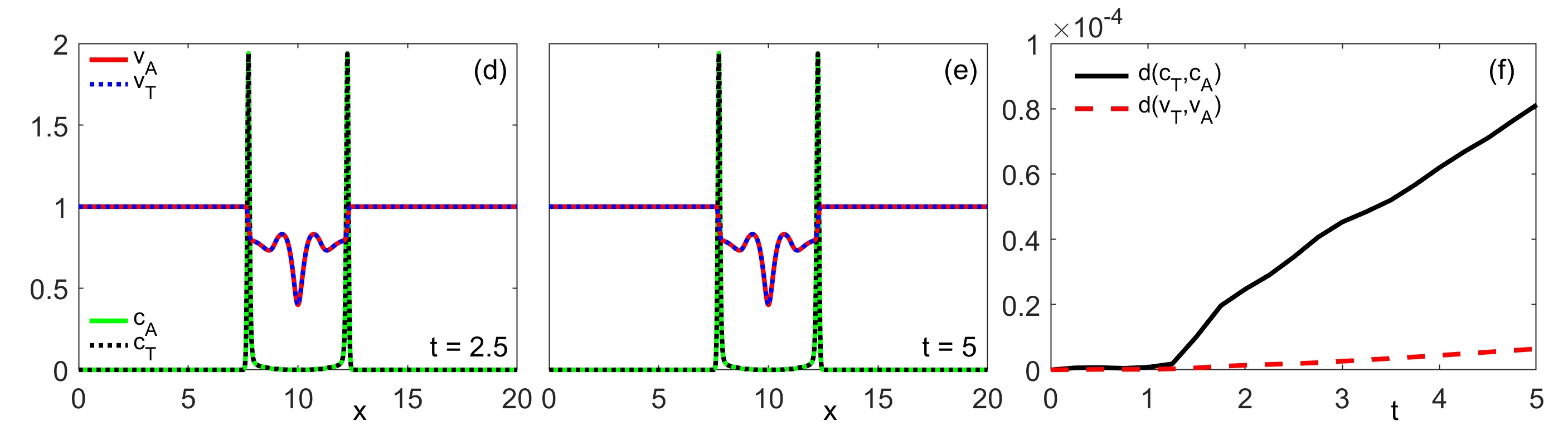}
\end{center}
\caption{Comparison between nonlocal formulations \cref{nlHapto} and \cref{nlProto}. (a-b) Cell and matrix densities for the models \cref{nlHapto} and \cref{nlProto} at $t=2.5$ and $t = 5$. (c) Difference between the solutions. For these simulations we take $\alpha = 10$, $r = 1$, $D_c = 0.01$, $\chi = 1$, $F_r = 2$, $f_c = 0$ and $f_v(c,v) = - c v$, along with (a-c) $g(c,v) = 10 v$, (d-f) $g(c,v) = 2.5c + 10 v$.} \label{figure1}
\end{figure}

\noindent
\cref{figure1} shows the computed solutions under (a-c) negligible cell-cell adhesion 
($S_{cc} = 0$) and (d-f) moderate cell-cell adhesion ($S_{cc} = S_{cv}/4$). The equivalence of the
two formulations is revealed through the negligible difference between solutions, with the
distance magnitude attributable to the subtly distinct numerical implementation. Both simulations 
describe an invasion/infiltration process, in which matrix degradation by the cells generates an 
adhesive gradient that pulls cells into the acellular surroundings. The impact of cell-cell adhesion is manifested in the compaction of cells at the leading edge into a tight aggregate.

\noindent
However,  as pointed out in \cref{AverOper}, differences in the nonlocal formulations can emerge in the 
vicinity of boundaries. To highlight this we consider an equivalent formulation to \cref{figure1} 
(a-c), but with the cells initially placed at the left boundary ($x_{c} = 0$ in \eqref{c0}), e.g. suggesting a tumor mass which is concentrated there and whose cells are expected to detach and migrate into the considered 1D domain, travelling from left to right. As 
stated earlier we impose zero-flux boundary conditions at $x=0$ (and $x=L$), and further suppose 
$c = v = 0$ and $\nabla c = \nabla v$ in the extradomain region ($\R \backslash \Om$). Representative simulations are shown in  \cref{figure2}. They are in agreement with our observation in \cref{BndL}. Indeed, for
this scenario, in the prototypical nonlocal model \cref{nlHapto}-\cref{DefAr}
% the cell population immediately 
% detaches from the boundary and moves rightwards: intuitively, precluding the formation of 
% adhesive contacts in the region $x<0$ results in a strong pull into the domain.
there is a very large adhesion velocity modulus at $x=0$; the cells are crowded within the tumor mass and their mutual interactions are maintained during the invasion process in a sufficiently strong manner to ensure a collective shift of the still concentrated cell aggregate, with a correspondingly strong tissue degradation in its wake. 
%evoking an unnaturally strong push of the population inside the domain leading to the tightening of the aggregate which quickly detaches and moves into the domain.} 
%\textcolor{blue}{KJP comment: I'm not sure ``unnatural'' is the right expression here. Specifically, the model is set-up in an inherently unnatural way to start with: for an adherent cell population the density could never (biologically) be expected to generate high densities at the boundary if it (and the extra-domain region) was unattractive to cells (i.e. not providing any anchors for attachment), unless there were other processes at work (e.g. pressure due to high cell densities within the domain, pushing cells towards the boundary). So it is only to be expected that unnatural results would be expected from an unnatural set-up. Also, I'm not sure we should describe the adhesion velocity as even nearly singular: it will be very much bounded and determined by the parameters of the system (the rather arbitrarily chosen parameters have been chosen in a way to exaggerate the dynamics, for demonstration purposes).} 
In the reformulation \cref{nlProto}-\cref{IrTr}, rather,
%the detachment is somewhat slower. 
the adhesion magnitude at $x=0$ is for the same initial condition much lower - suggesting a tumor whose cells are readier to detach and migrate individually. This results in a more diffusive spread, with accordingly less degradation of tissue, and with cell mass remaining available at the original site over a larger time span. The latter scenario is  different from the former one, but it seems nevertheless reasonable, as a tumor mass would  very often not move as a whole from its original location to another in a relatively short time; moreover, the active cells in a sufficiently large tumor (releasing substantial amounts of acidity) are known to preferentially adopt a migratory phenotype and perform EMT (epithelial-mesenchymal transition), see e.g., \cite{Gupta168,Peppicelli2014,PrietoGarca2017}, which supports the idea of cells moving in a loose way rather than in compact, highly aggregated assemblies \footnote{unless environmental influences dictate conversion to a collective type of motion}. As such, our simulations suggest that,  within this particular function- and parameter setting, choosing the adhesion operator in the form \cref{DefAr} instead of \cref{IrTr} might  possibly overestimate the tumor invasion speed and associated healthy tissue degradation, thereby predicting a spatially concentrated tumor and neglecting regions with lower cell densities which can nevertheless trigger tumor recurrence if untreated.

%\textcolor{blue}{KJP: There is sense in the above, but I still find this still a rather tenuous/speculative justification {\em based on the current simulations}. The argument that formulation \cref{DefAr} overestimates invasion speed can be true simply because the model has not been properly parametrised/set-up to begin with: since the size of the adhesive operator provides a fundamental bound on the adhesive velocity of the cell population, any parameter/functional/boundary choices that generate implausible invasion speeds is flawed. To me {\em in terms of the current state of the paper}, one can only robustly justify argue one choice over based on its analytical advantages. Justifying it robustly from a ``biological realism'' perspective can only be made if models were formulated, parametrised and compared according to experimental datasets: at the moment, we have a completely abstract set-up {\em that has been tuned to emphasise the differences}. It may turn out that one model form can be subsequently justified over the other, but I don't think we have the evidence in this paper to argue that. Summarising, I think we have to be much more cautious and not overplay the ``biological realism'' side: that would require a very different type of study (including a solid motivation of the reformulation from a modelling perspective).}

\begin{figure}
\begin{center}
\includegraphics[width=\textwidth]{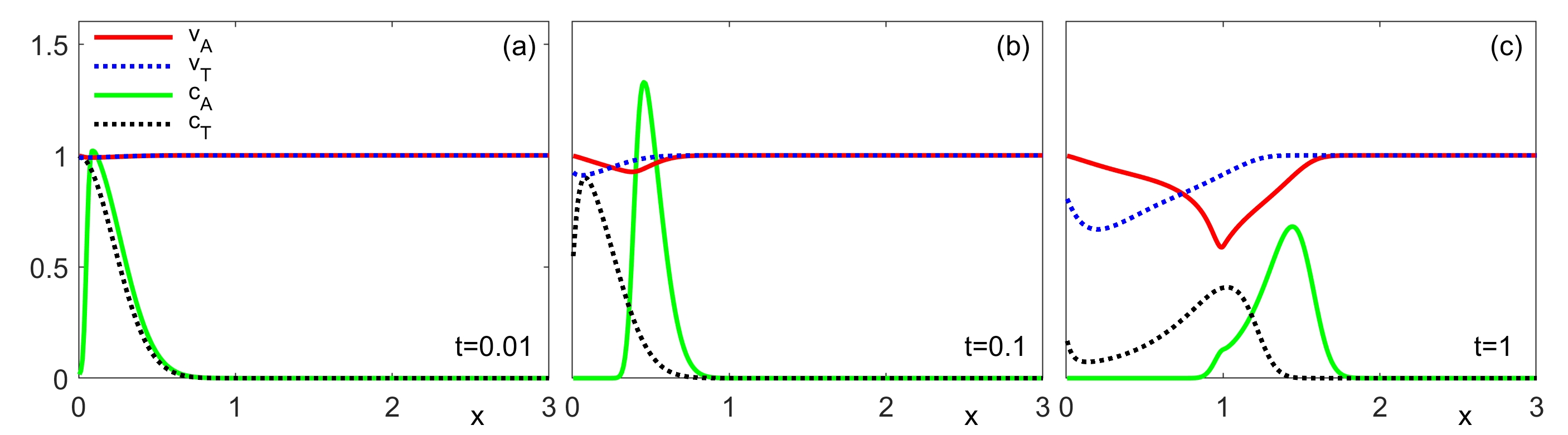}
\includegraphics[width=\textwidth]{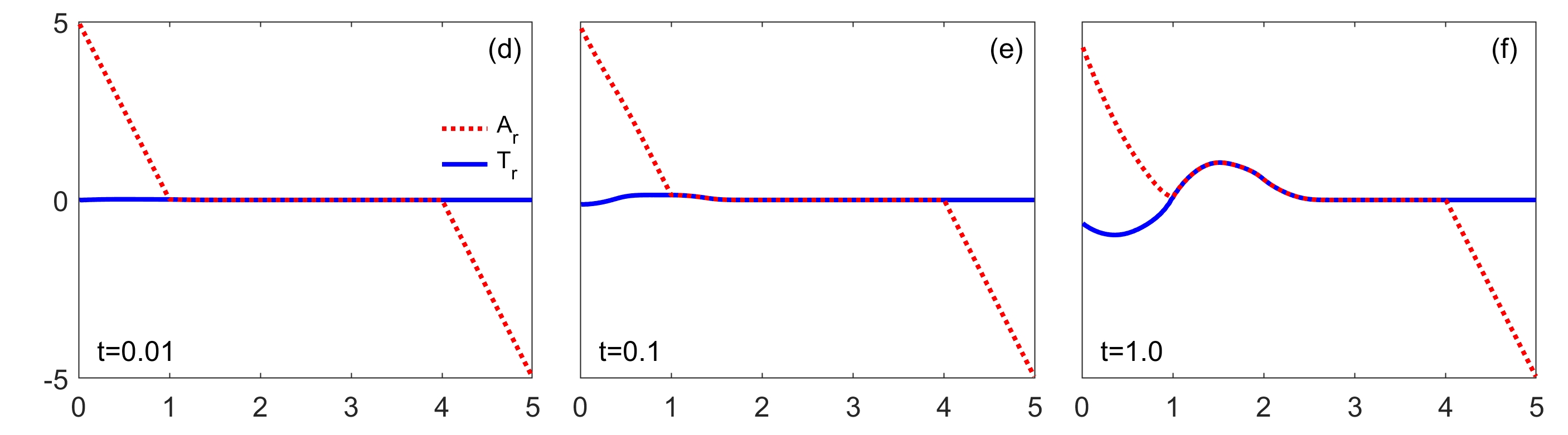}
\end{center}
\caption{(a-c) Comparison between nonlocal formulations \cref{nlHapto} and \cref{nlProto} near boundaries. Model as in \cref{figure1} (a-c), but with the cells initially concentrated at the boundary.
(d-f) Comparison of the two forms of nonlocal operator corresponding to the simulations represented in (a-c). The operators are practically identical sufficiently far from the boundary, but can diverge significantly for distances $<r$ from the boundaries.} \label{figure2}
\end{figure}

\subsection{Comparison between nonlocal and local formulation}\label{SecNum2}

Having %\cmg{studied} the equivalence of the two
compared together the original, \cref{nlHapto}, and the new, \cref{nlProto},  nonlocal formulations, we next 
consider the extent to which their dynamics can be captured by the classical local 
formulation \cref{lProto}. Note that for nonlocal model simulations we will 
restrict to the original formulation \cref{nlHapto}, %due to its more efficient (computational-time) form. 
so that we can avail ourselves of an already well-established efficient (in terms of computational time) numerical scheme \cite{Alf-num09}.
Here we use $c_{L}$ and $v_{L}$ to denote
solutions to the local formulation and  $c_{Ar}$ and $v_{Ar}$ to denote
solutions to the nonlocal model with sensing radius $r$. We remark that a large 
number of related local and nonlocal models have been numerically studied 
to describe the 
invasion-type process considered here (e.g. 
\cite{perumpanani1996,anderson2000,gerisch2008,PAS10}): here the specific focus is to explore the convergence of nonlocal to local form as $r\rightarrow 0$, which, as far as we are aware, has not been systematically investigated.

\noindent
As in the first test we use the initial values \cref{iniv0} and \cref{c0}, choosing $x_c=L/2$, $\alpha=10$ in the latter, and consider the coefficients and functions as proposed in \cref{Bsp}. 
Under these choices the resultant nonlinear diffusion coefficient for the $c$-equation 
in the classical local formulation (compare \cref{lProtoc}) becomes
\begin{equation} \label{dtilde}
\tilde D_c(c,v)=%D_c(c,v)-c\chi (c,v)\partial_cg(c,v)=
\frac{a^2(1+c)^2(1+c+v)^2-bc(1+cv)(S_{cc}+(S_{cc}-S_{cv})v)}{(1+cv)^2(1+c+v)^2}.
\end{equation}
Notably, this potentially becomes negative under an injudicious combination of adhesive 
strengths $S_{cc}$, $S_{cv}$, and of $a,b$. Likewise, the 
actual haptotaxis sensitivity function takes the form
\begin{equation}
\tilde \chi(c,v)=%\chi(c,v)\partial_vg(c,v)=
b\frac{S_{cv}+(S_{cv}-S_{cc})c}{(1+cv)(1+c+v)^2}.
\end{equation}
Again, depending on the relationship between $S_{cc}$ and $S_{cv}$, this can become negative, 
which would lead to repellent haptotaxis: cells effectively moving away from regions with large ECM gradients, a rather unexpected behaviour. This suggests that cell-tissue adhesions should dominate over cell-cell adhesions,\footnote{An analogous behaviour was suggested by the two-scale structured population model with adhesion introduced in \cite{ESuSt2016}.} as 'usual' haptotaxis, i.e. towards the increasing tissue gradient, is known to be an essential component of cell migration, this applying to several types of cells moving through the ECM (tumor cells, mesenchymal stem cells, fibroblasts, endothelial cells, etc.) see e.g. \cite{lamal,pickup,wen} and references therein.\\

\begin{figure}
\begin{center}
\includegraphics[width=\textwidth]{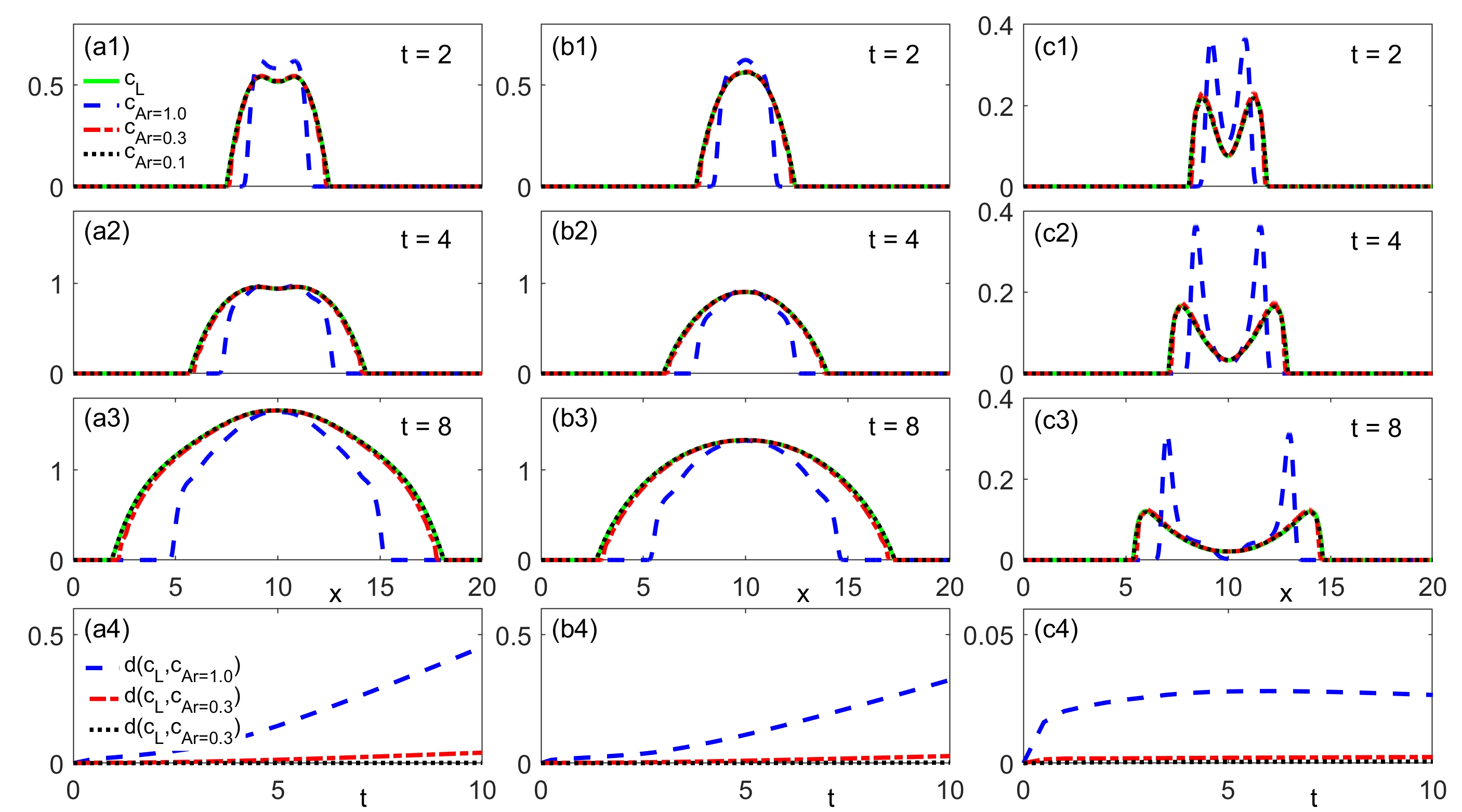}
\end{center}
\caption{Convergence between nonlocal and local/classical formulations under negligible cell-cell adhesion, $S_{cc} = 0$, $S_{cv} = 10$. Functional forms as proposed in \cref{Bsp}, with modifications specified in the subfigures. (a) Solutions for $r = 0.1, 0.3, 1.0$ at (a1) $t=2$, (a2) $t=4$ and (a3) $t=8$; (a4) Distance 
between local/nonlocal solutions as a function of time. For these simulations, we take $a = 0.01$, $b = 1$, $\mu_c = 0.01$, $K_c = 2$, $\eta_c = 1$, $\mu_v = 0$, $\lambda_v = 1$. (b) Solutions for $r = 0.1, 0.3, 1.0$ at (b1) $t=2$, (b2) $t=4$ and (b3) $t=8$; (b4) Distance 
between local/nonlocal solutions as a function of time. Parameters as in (a) except $\mu_v = 1$, $K_v = 1$. (c) Solutions for $f_c = 0$ and $f_v(c,v)= - c v$, with the other parameters as in (a).} \label{figure3}
\end{figure}

\begin{figure}
	\begin{center}
		\includegraphics[width=\textwidth]{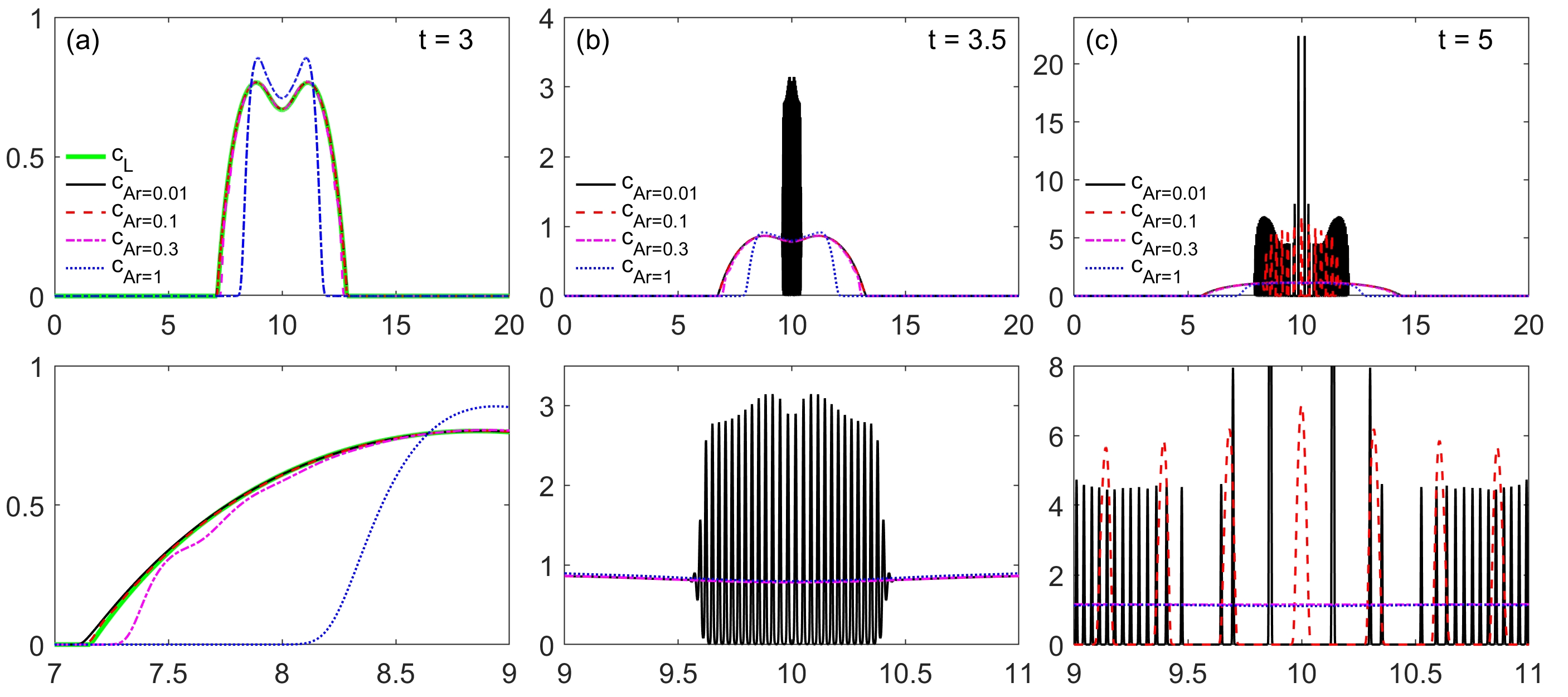}
	\end{center}
	\caption{Time restricted convergence under moderate cell-cell adhesion, $S_{cc} = 2.5$, $S_{cv} = 10$. Top row shows solutions across the full spatial region ($[0,20]$), the bottom row magnifies a relevant portion for clarity. Solutions to local and nonlocal models under the functional forms proposed in \cref{Bsp} for $r = 0.01, 0.1, 0.3, 1.0$ at (a) $t=3$, (b) $t=3.5$ and (c) $t=5$. In (a) solutions to the local model continue to exist and we observe convergence between local and nonlocal formulations. In (b-c) the solutions to the local model are noncomputable. Nonlocal models, however, can destabilise into a pattern of aggregates. %For these simulations, we take the functional forms proposed in \cref{Bsp} with 
	Parameters: $a = 0.01$, $b = 1$, $\mu_c = 0.01$, $K_c = 2$, $\eta_c = 1$, $\mu_v = 0$, $\lambda_v = 1$ and adhesion parameters as above.} \label{figure4}
\end{figure}

\noindent
Simulations are plotted in \cref{figure3} where we show cell densities 
for the local model ($c_L$) and nonlocal model under three 
sensing radii ($c_{Ar=0.1}, c_{Ar=0.3}, c_{Ar=1.0}$). In this first set of simulations 
we assume negligible cell-cell adhesion ($S_{cc} = 0$), which automatically 
ensures positivity for the diffusion coefficient of the equivalent local model,
$\tilde D_c(c,v)$. We note that matrix renewal is absent ($\mu_{v} = 0$)
in the left-hand column and present ($\mu_{v} > 0$) in the central column. In the right-hand
column we show the greater generality of the results under vastly simplified kinetics,
specifically setting $f_c(c,v) = 0$ and $f_v(c,v) = - c v$ (with the other functional forms as
in \cref{Bsp}). Simulations highlight the convergence between local and nonlocal models
as $r\rightarrow 0$: for $r = 0.1$, the solution differences become negligible. However,
distinctions emerge for large $r$, where we can expect significant discrepancy 
between the solutions. This suggests that the local model fails to accurately
predict the behaviour in cases where cells sample over relatively large regions of their
local environment.

\noindent
Next, we extend to include a degree of cell-cell adhesion, setting functions
and parameters as in \cref{figure3}, except now $S_{cc} > 0$. Notably this raises
the possibility of a negative diffusion coefficient in the classical formulation and 
subsequent illposedness. Solutions under a representative set of parameters are shown in Figure 
\ref{figure4}. For $t$ below some critical time we observe convergence as before, with the
nonlocal formulation converging to solutions of the local model as $r\rightarrow 0$. 
However, continued matrix degradation further depletes $v$, with the result that 
\eqref{dtilde} can become negative. At this point (in this case $t\approx 3.2\ldots$)
the local model becomes illposed and its solutions become incomputable (implying nonexistence of 
solutions). However, the nonlocal formulation appears to preserve wellposedness, consistent with 
previous theoretical studies where extending to a nonlocal formulation regularises a
singular local model (e.g. \cite{HilPSchm}). Solutions to the nonlocal model instead destabilise
into a quasi-periodic pattern of cell aggregations, maintained through the cell-cell adhesion, and
with a wavelength shrinking as $r \rightarrow 0$. 

%\begin{figure}
%\begin{center}
%\includegraphics[width=\textwidth]{Figure4.jpg}
%\end{center}
%\caption{Time restricted convergence under moderate cell-cell adhesion, $S_{cc} = 2.5$, $S_{cv} = 10$. Top row shows solutions across full spatial region ($[0,\red{20}]$), bottom row magnify a relevant portion for clarity. Solutions to local and nonlocal models under the functional forms proposed in \cref{Bsp} for $r = 0.01, 0.1, 0.3, 1.0$ at (a) $t=3$, (b) $t=3.5$ and (c) $t=5$. In (a) solutions to the local model continue to exist and we observe convergence between local and nonlocal formulations. In (b-c) the solutions to the local model are noncomputable. Nonlocal models, however, can destabilise into a pattern of aggregates. For these simulations, we take the functional forms proposed in \cref{Bsp} with $a = 0.01$, $b = 1$, $\mu_c = 0.01$, $K_c = 2$, $\eta_c = 1$, $\mu_v = 0$, $\lambda_v = 1$ and adhesion parameters as above.} \label{figure4}
%\end{figure}

\noindent
Finally, we remark that convergence of solutions extends beyond the specific functional forms and, 
as a representative example, we consider a minimalist setting based on linear/constant forms. 
Specifically, we set $D_c = a$ (constant), $\chi = 1$, $f_c = 0$, $g(c,v) = S_{cc} c + S_{cv} v$ 
and $f_v(c,v) = -\mu c v$. In this scenario, the diffusion and haptotaxis coefficients for the classical 
local formulation \cref{lProto} reduce to 
\begin{equation} 
\tilde D_c(c,v)=a - S_{cc} c \quad \mbox{and} \quad \tilde \chi(c,v)= S_{cv}.
\end{equation}
Positivity is only guaranteed under appropriate parameter selection. Such a case is
illustrated in \cref{figure5} (a) where we assume negligible cell-cell adhesion ($S_{cc} = 0$). Clearly, we observe convergence between the nonlocal and local formulations
as $r\rightarrow 0$. Inappropriate parameter selection, however, generates backward diffusion 
in the local model and solutions are consequently incomputable. Under such scenarios, 
however, solutions to the nonlocal model appear to exist: \cref{figure5} (b) plots
the behaviour under shrinking $r$. 
In all cases considered in this test the cells do not reach the boundary region where the difference between the nonlocal formulations \cref{nlHapto,nlProto} can play a role. Thus, we expect the same solution if reformulation \cref{nlProto} is applied instead.
\begin{figure}
\begin{center}
\includegraphics[width=\textwidth]{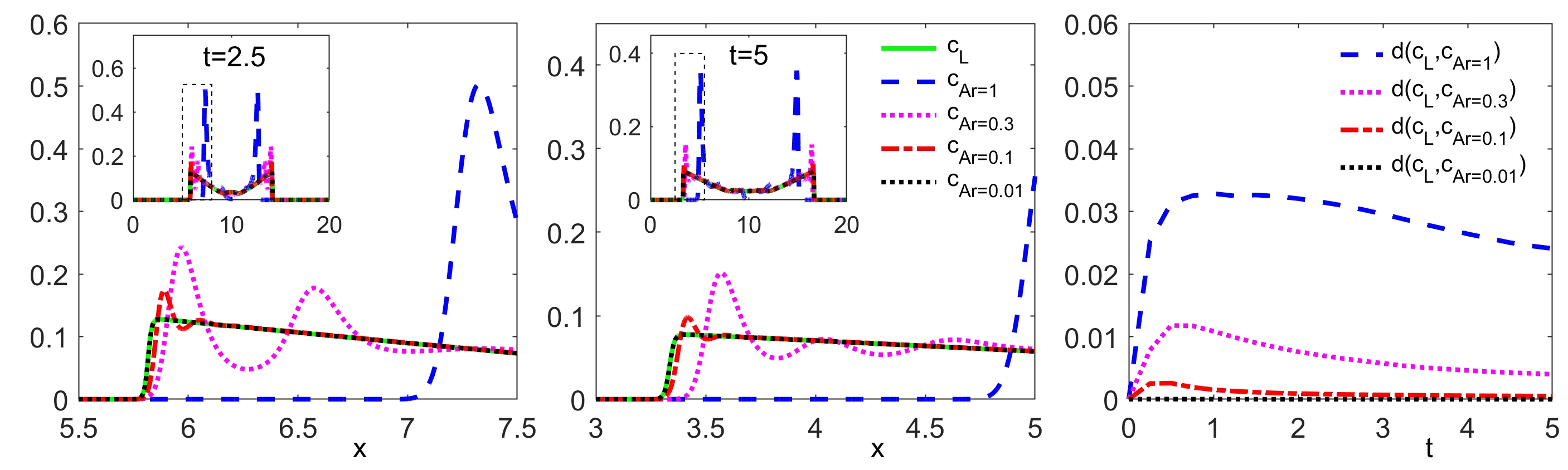}
\end{center}
\caption{Convergence between nonlocal and local/classical formulations under a set of minimalistic 
linear functional forms ($D_c = 0.01, \chi = 1, f_c = 0, g(c,v) = S_{cc} c + S_{cv} v, f_v(c,v) = -\mu c v$).
(a) Negligible cell-cell adhesion, $S_{cc}=0, S_{cv}=10$: solutions shown at (left) $t=2.5$ and (middle) $t=5$, with the distance between solutions to the nonlocal and local model shown in the right panel.}\label{figure5}
\end{figure}

\section{Discussion}\label{sec:discussions}

In this work we provide a rigorous limit procedure which links  nonlocal models involving adhesion or a nonlocal form of chemotaxis gradient to their local counterparts featuring haptotaxis, respectively chemotaxis in the usual sense. As such, our paper closes a gap in the existing literature. Moreover, it offers a unified treatment of the two types of models and extends the previous mathematical framework to settings allowing for more general, solution dependent,  coefficient functions (diffusion, tactic sensitivity, adhesion velocity, nonlocal taxis  gradient, etc.). Finally, we provide simulations  illustrating some of our theoretical findings in 1D.

\noindent
Our  reformulations in terms of  $\mathcal{T}_r$ and $\mathcal{S}_r$ reveal  the tight relationship between the nonlocal operators $\mathcal A_r$ and $\mathring{\nabla}_{r}$  and the (local) gradient. This suggests that both nonlocal descriptions (adhesion, chemotaxis) actually encompass the dependence on the signal gradients rather than on the signal concentration/density itself, which is in line with the biological phenomenon. Indeed, through their transmembrane elements (e.g. receptors, ion channels etc.) the cells are mainly able to perceive and respond to differences in the signal at various locations or within more or less confined areas rather than measure effective signal concentrations. Along with the mentioned solution dependency of the nonlocal model coefficients, the influence of the gradient possibly reflects into contributions of the adhesion/nonlocal chemotaxis to the (nonlinear) diffusion in the local setting obtained through the limiting procedure. 

\noindent
The set $\Om _r$ (as introduced in \textit{\cref{not}}) can be regarded as the 'domain of restricted sensing', meaning that there cells a priori sense only what happens inside $\Omega$, the domain of interest. The measure of this subdomain is a decreasing function of the sensing radius $r$. When $r\to 0$ the set $\Om _r$ tends to cover the whole domain $\Om $, whereas as $r$ increases the cells can sense at increasingly larger distances;   correspondingly, $\Om _r$ shrinks. For $r>\diam(\Omega)$ the restricted sensing domain is  empty: everywhere in $\Omega$  the cells  can perceive signals not only  from any point within $\Omega$ but potentially also from the outside. In this paper, however, we look at models with no-flux boundary conditions. This corresponds, e.g., to the impenetrability of the  walls of a Petri dish or that of  comparatively hard barriers limiting the areas populated by migrating cells, e.g. bones or cartilage material. As a result, the cells  in the boundary layer $\Om \setminus \Om_r$ have a much reduced ability to stretch their protrusions outside $\Om $ and thus gain little  information from without. To simplify matters, we assume in this work that there is no such information or it is insufficient to trigger any change in their behaviour. In the definitions of $\mathcal{T}_r$ and $\mathcal{S}_r$ this corresponds to the integrands being set to zero in $\Omega\backslash\overline{\Omega}_r$. 

\noindent
It is important to note that for points $x\in\Omega\backslash\overline{\Omega}_r$ the influence of a signal $p$ in a direction $y\in S_1$  is not taken into account by $\mathring{\nabla}_{r}$ at all if $x+ry\not\in\overline{\Omega}$. If $\mathcal{S}_r$ is used instead, then its contribution to the average is given by  
\begin{align*}
\tilde{y}:=n \left(\int_0^1\chi_{\Omega}\nabla p(x+rsy)\,ds\cdot y\right)y.
\end{align*}
Thus, thanks to integration w.r.t. $s$, the resulting vector $\tilde{y}$ assembles the impact of those parts of the segment connecting $x$ and $x+ry$ which are contained in $\Omega$. It is parallel to $y$, and it may have the same or  the opposite orientation.
% \footnote{ Hier sollte man den Unterschied zwischen 'direction' und 'orientation' im Hinterkopf behalten. Wir verwendeten 'direction' f\"ur die Richtung des Vektors, also schlage ich vor, 'orientation' f\"ur die positive/negative Ausrichtung zu verwenden.}.
In particular this means that  although for a certain range of directions large parts of the sensing region of a cell are actually outside $\Omega$,  this may still strongly influence the speed and actual direction of the drift.  The effect of integration w.r.t. $s$ in $\mathcal T_r$ is less obvious, since in this case the average w.r.t. $y$ is computed over the ball $B_1$. This already achieves the covering of the whole sensing region  by allowing a cell  to gather information about the signal not only in any direction $y/{|y|}$, but also at any distance less than $r$.  The additional integration over the path  $x+rsy$, $s\in[0,1]$, appears to mean that  cells at $x\in\Omega_r$ are able to measure the average of the signal gradient all along such line segment rather than its value directly at the ending  point. Indeed, from a biological viewpoint this description seems to make more sense, as cells do not jump from one position to another, nor do they send out their protrusions in a discontinuous way bypassing certain space points along a chosen direction. Averages over cell paths are then averaged w.r.t. $y$, which finally determines the direction of  population movement.  \cref{Ex2} indicates that the effect of even an extremely concentrated signal gradient is mollified by averaging. This agrees with our expectations from using non-locality. In higher dimensions $n\geq2$,
the    two-stage averaging in ${\cal T}_r$ (w.r.t. $s$ and $y$) produces a direction field which is smooth away from the concentration point and also  weakens but still keeps the singularity there.  In contrast, averaging only w.r.t. $y$ leads instead to jump discontinuities at a unit distance from the accumulation point.  Moreover, we remark that without  integrating w.r.t. $s$ in ${\cal T}_r(\nabla\cdot)$ one cannot regain ${\cal A}_r$. 

\noindent
The effect observed in  \cref{BndL} further supports the conjecture that the nonlocal operators which act directly on the signal gradients might actually be a more appropriate modelling tool. While inside the subdomain $\Omega_r$  there is no difference (recall \cref{LemAdVel,LemNGR}), inside the boundary layer $\Omega\backslash\overline{\Omega}_r$ the limiting behaviour as $r\rightarrow 0$ is qualitatively distinct. Indeed,  \cref{BndL} shows that using, e.g., $\mathcal A_r$, leads, for $r\rightarrow0$, to unnatural sharp singularities at the boundary of $\Omega$ even in the absence of  signal gradients, whereas this does not happen if $\mathcal{T}_r$ is used instead. Simulations in \cref{SecNum1} (see \cref{figure2}) confirm our theoretical findings and show a substantial difference between the solutions obtained with the two nonlocal  formulations involving \cref{DefAr} and \cref{IrTr}, respectively.  The choice \cref{IrTr} is motivated  above all from a mathematical viewpoint (as it enables a rigorous, well-justified passage to the limit for $r\to 0$),  but it also seems to make sense biologically, as our above comments and the simulations performed for the particular setting in \cref{SecNum1} suggest.

\section*{Acknowledgement}

MK and CS acknowledge the partial support of the research initiative \textit{Mathematics Applied to Real World Challenges} (\textit{MathApp}) of the TU Kaiserslautern. CS also acknowledges funding by the Federal Ministry of Education and Research (BMBF) in the project \textit{GlioMaTh}.

\phantomsection


\begin{thebibliography}{10}

\bibitem{anderson2000}
A.~R.~A. Anderson, M.~A.~J. Chaplain, E.~L. Newman, R.~J.~C. Steele, and
  M.~Thompson, A.
\newblock Mathematical modelling of tumour invasion and metastasis.
\newblock {\em Computational and mathematical methods in medicine},
  2(2):129--154, 2000.

\bibitem{Armstrong2006}
N.~J. Armstrong, K.~J. Painter, and J.~A. Sherratt.
\newblock A continuum approach to modelling cell-cell adhesion.
\newblock {\em J. Theoret. Biol.}, 243(1):98--113, 2006.

\bibitem{Bell}
G.~Bell.
\newblock Models for the specific adhesion of cells to cells.
\newblock {\em Science}, 200(4342):618--627, 1978.

\bibitem{BDB}
G.~Bell, M.~Dembo, and P.~Bongrand.
\newblock Cell adhesion. competition between nonspecific repulsion and specific
  bonding.
\newblock {\em Biophysical Journal}, 45(6):1051 -- 1064, 1984.

\bibitem{BBTW}
N.~Bellomo, A.~Bellouquid, Y.~Tao, and M.~Winkler.
\newblock Toward a mathematical theory of keller-segel models of pattern
  formation in biological tissues.
\newblock {\em Math. Models Meth. Appl. Sci.}, 25(9):1663--1763, 2015.

\bibitem{BTCE}
V.~Bitsouni, D.~Trucu, M.~A.~J. Chaplain, and R.~Eftimie.
\newblock {Aggregation and travelling wave dynamics in a two-population model
  of cancer cell growth and invasion}.
\newblock {\em Mathematical Medicine and Biology: A Journal of the IMA},
  35(4):541--577, 01 2018.

\bibitem{HillButten2019}
A.~Buttensch\"{o}n and T.~Hillen.
\newblock Nonlocal adhesion models for microorganisms on bounded domains,
  preprint.

\bibitem{Butten}
A.~Buttensch\"{o}n, T.~Hillen, A.~Gerisch, and K.~J. Painter.
\newblock A {\it space-jump} derivation for non-local models of cell-cell
  adhesion and non-local chemotaxis.
\newblock {\em J. Math. Biol.}, 76(1-2):429--456, 2018.

\bibitem{Carillo_etal19}
J.~Carrillo, H.~Murakawa, M.~Sato, H.~Togashi, and O.~Trush.
\newblock A population dynamics model of cell-cell adhesion incorporating
  population pressure and density saturation.
\newblock {\em J. Theor. Biol.}, 474:14--24, 2019.

\bibitem{CLSW}
M.~Chaplain, M.~Lachowicz, Z.~Szymanska, and D.~Wrzosek.
\newblock Mathematical modelling of cancer invasion: The importance of
  cell-cell adhesion and cell-matrix adhesion.
\newblock {\em Mathematical Models and Methods in Applied Sciences},
  21(04):719--743, 2011.

\bibitem{CPSZ}
L.~Chen, K.~Painter, C.~Surulescu, and A.~Zhigun.
\newblock Mathematical models for cell migration: a nonlocal perspective, 2019.

\bibitem{DT93}
R.~B. Dickinson and R.~T. Tranquillo.
\newblock A stochastic model for adhesion-mediated cell random motility and
  haptotaxis.
\newblock {\em Journal of Mathematical Biology}, 31(6):563--600, 1993.

\bibitem{DiMilla}
P.~DiMilla, K.~Barbee, and D.~Lauffenburger.
\newblock Mathematical model for the effects of adhesion and mechanics on cell
  migration speed.
\newblock {\em Biophysical Journal}, 60(1):15 -- 37, 1991.

\bibitem{DTGC14}
P.~Domschke, D.~Trucu, A.~Gerisch, and M.~A.~J. Chaplain.
\newblock Mathematical modelling of cancer invasion: implications of cell
  adhesion variability for tumour infiltrative growth patterns.
\newblock {\em J. Theoret. Biol.}, 361:41--60, 2014.

\bibitem{DTGC17}
P.~Domschke, D.~Trucu, A.~Gerisch, and M.~A.~J. Chaplain.
\newblock Structured models of cell migration incorporating molecular binding
  processes.
\newblock {\em J. Math. Biol.}, 75(6-7):1517--1561, 2017.

\bibitem{dyson13}
J.~Dyson, S.~Gourley, and G.~Webb.
\newblock A non-local evolution equation model of cell-cell adhesion in higher
  dimensional space.
\newblock {\em Journal of biological dynamics}, 7:68--87, 2013.

\bibitem{dyson10}
J.~Dyson, S.~A. Gourley, R.~Villella-Bressan, and G.~F. Webb.
\newblock Existence and asymptotic properties of solutions of a nonlocal
  evolution equation modeling cell-cell adhesion.
\newblock {\em SIAM J. Math. Anal.}, 42(4):1784--1804, 2010.

\bibitem{eftimie}
R.~Eftimie.
\newblock {\em Hyperbolic and Kinetic Models for Self-organised Biological
  Aggregations. A Modelling and Pattern Formation Approach}.
\newblock Springer, Cham, 2018.

\bibitem{ESuSt2016}
C.~Engwer, C.~Stinner, and C.~Surulescu.
\newblock On a structured multiscale model for acid-mediated tumor invasion:
  the effects of adhesion and proliferation.
\newblock {\em Math. Models Methods Appl. Sci.}, 27(7):1355--1390, 2017.

\bibitem{eom}
D.~S. Eom and D.~M. Parichy.
\newblock A macrophage relay for long-distance signaling during postembryonic
  tissue remodeling.
\newblock {\em Science}, 355(6331):1317--1320, 2017.

\bibitem{Evans1990}
L.~C. Evans.
\newblock {\em Weak convergence methods for nonlinear partial differential
  equations}, volume~74 of {\em CBMS Regional Conference Series in
  Mathematics}.
\newblock Published for the Conference Board of the Mathematical Sciences,
  Washington, DC; by the American Mathematical Society, Providence, RI, 1990.

\bibitem{GaPa08}
G.~Garcia and C.~Parent.
\newblock Signal relay during chemotaxis.
\newblock {\em Journal of Microscopy}, 231(3):529--534, 2008.

\bibitem{Alf-num09}
A.~Gerisch.
\newblock {On the approximation and efficient evaluation of integral terms in
  PDE models of cell adhesion}.
\newblock {\em IMA Journal of Numerical Analysis}, 30(1):173--194, 01 2010.

\bibitem{gerisch2008}
A.~Gerisch and M.~A.~J. Chaplain.
\newblock Mathematical modelling of cancer cell invasion of tissue: local and
  non-local models and the effect of adhesion.
\newblock {\em J. Theoret. Biol.}, 250(4):684--704, 2008.

\bibitem{GePa10}
A.~Gerisch and K.~J. Painter.
\newblock Mathematical modeling of cell adhesion and its applications to
  developmental biology and cancer invasion.
\newblock 2010.

\bibitem{GM17}
L.~Gonz\'alez-M\'endez, I.~Seijo-Barandiar\'an, and I.~Guerrero.
\newblock Cytoneme-mediated cell-cell contacts for hedgehog reception.
\newblock {\em eLife}, 6:e24045, 2017.

\bibitem{Gupta168}
S.~C. Gupta and Y.-Y. Mo.
\newblock Abstract 168: Malat1 is crucial for epithelial-mesenchymal transition
  of breast cancer cells in acidic microenvironment.
\newblock {\em Cancer Research}, 75(15 Supplement):168--168, 2015.

\bibitem{Hillen2007}
T.~Hillen.
\newblock A classification of spikes and plateaus.
\newblock {\em SIAM Rev.}, 49(1):35--51, 2007.

\bibitem{HilPSchm}
T.~Hillen, K.~Painter, and C.~Schmeiser.
\newblock Global existence for chemotaxis with finite sampling radius.
\newblock {\em Discrete Contin. Dyn. Syst. Ser. B}, 7(1):125--144.

\bibitem{WHP16}
T.~Hillen, K.~Painter, and M.~Winkler.
\newblock Global solvability and explicit bounds for non-local adhesion models.
\newblock {\em European Journal of Applied Mathematics}, 29(04):1--40, 2018.

\bibitem{kav-suzuki}
N.~I. Kavallaris and T.~Suzuki.
\newblock {\em Non-local partial differential equations for engineering and
  biology}, volume~31 of {\em Mathematics for Industry (Tokyo)}.
\newblock Springer, Cham, 2018.
\newblock Mathematical modeling and analysis.

\bibitem{Kuusela-Alt}
E.~Kuusela and W.~Alt.
\newblock Continuum model of cell adhesion and migration.
\newblock {\em Journal of Mathematical Biology}, 58(1):135, 2008.

\bibitem{LSU}
O.~{Ladyzhenskaya}, V.~{Solonnikov}, and N.~{Ural'tseva}.
\newblock {Linear and quasi-linear equations of parabolic type. Translated from
  the Russian by S. Smith.}
\newblock {Translations of Mathematical Monographs. 23. Providence, RI:
  American Mathematical Society (AMS). XI, 648 p. (1968).}, 1968.

\bibitem{lamal}
L.~Lamalice, F.~Le~Boeuf, and J.~Huot.
\newblock Endothelial cell migration during angiogenesis.
\newblock {\em Circulation Research}, 100:782 -- 794, 2007.

\bibitem{Lions}
J.-L. Lions.
\newblock {\em Quelques m{\'e}thodes de r{\'e}solution des probl{\`e}mes aux
  limites non lin{\'e}aires}.
\newblock Dunod, Gauthier-Villars, Paris, 1969.

\bibitem{loy-preziosi}
N.~Loy and L.~Preziosi.
\newblock Kinetic models with non-local sensing determining cell polarization
  and speed according to independent cues.
\newblock arXiv:1906.11039, 2019.

\bibitem{lu12}
P.~Lu, V.~Weaver, and Z.~Werb.
\newblock The extracellular matrix: A dynamic niche in cancer progression.
\newblock {\em J. Cell Biol.}, 196:395--406, 2012.

\bibitem{15MuTo}
H.~Murakawa and H.~Togashi.
\newblock Continuous models for cell-cell adhesion.
\newblock {\em Journal of Theoretical Biology}, 374:1--12, 2015.

\bibitem{othmer-hillen2}
H.~Othmer and T.~Hillen.
\newblock The diffusion limit of transport equations ii: chemotaxis equations.
\newblock {\em SIAM J. Appl. Math.}, 62:1122--1250, 2002.

\bibitem{PAS10}
K.~J. Painter, N.~J. Armstrong, and J.~A. Sherratt.
\newblock The impact of adhesion on cellular invasion processes in cancer and
  development.
\newblock {\em Journal of Theoretical Biology}, 264(3):1057 -- 1067, 2010.

\bibitem{Peppicelli2014}
S.~Peppicelli, F.~Bianchini, E.~Torre, and L.~Calorini.
\newblock Contribution of acidic melanoma cells undergoing
  epithelial-to-mesenchymal transition to aggressiveness of non-acidic melanoma
  cells.
\newblock {\em Clinical {\&} Experimental Metastasis}, 31(4):423--433, Jan.
  2014.

\bibitem{perumpanani1996}
A.~J. Perumpanani, J.~A. Sherratt, J.~Norbury, and H.~M. Byrne.
\newblock {{B}iological inferences from a mathematical model for malignant
  invasion}.
\newblock {\em Invasion Metastasis}, 16(4-5):209--221, 1996.

\bibitem{pickup}
M.~Pickup, J.~Mouw, and V.~Weaver.
\newblock The extracellular matrix modulates the hallmarks of cancer.
\newblock {\em EMBO reports}, 15:1243--1253, 2014.

\bibitem{PrietoGarca2017}
E.~Prieto-Garc{\'{\i}}a, C.~V. D{\'{\i}}az-Garc{\'{\i}}a,
  I.~Garc{\'{\i}}a-Ruiz, and M.~T. Agull{\'{o}}-Ortu{\~{n}}o.
\newblock Epithelial-to-mesenchymal transition in tumor progression.
\newblock {\em Medical Oncology}, 34(7), May 2017.

\bibitem{SSM17}
I.~S\'aenz-de Santa-Mar\'ia, C.~Bernardo-Casti\~neira, E.~Enciso,
  I.~Garc\'ia-Moreno, J.~Chiara, C.~Suarez, and M.-D. Chiara.
\newblock Control of long-distance cell-to-cell communication and autophagosome
  transfer in squamous cell carcinoma via tunneling nanotubes.
\newblock {\em Oncotarget}, 8:20939--20960, 2017.

\bibitem{SGAP09}
J.~Sherratt, S.~Gourley, N.~Armstrong, and K.~Painter.
\newblock Boundedness of solutions of a non-local reaction-diffusion model for
  adhesion in cell aggregation and cancer invasion.
\newblock {\em European Journal of Applied Mathematics}, 20(1):123--144, 2009.

\bibitem{Showalter}
R.~E. Showalter.
\newblock {\em Monotone Operators in Banach Space and Nonlinear Partial
  Differential Equations}, volume~49 of {\em Mathematical Surveys and
  Monographs}.
\newblock American Mathematical Society, 1997.

\bibitem{temam2001navier}
R.~Temam.
\newblock {\em Navier-Stokes equations: theory and numerical analysis}, volume
  343.
\newblock American Mathematical Soc., 2001.

\bibitem{Aydar}
A.~Uatay.
\newblock {\em Multiscale Mathematical Modeling of Cell Migration: From Single
  Cells to Populations}.
\newblock PhD thesis, TU Kaiserslautern, 2019.

\bibitem{ward}
M.~Ward and D.~Hammer.
\newblock A theoretical analysis for the effect of focal contact formation on
  cell-substrate attachment strength.
\newblock {\em Biophysical Journal}, 64(3):936 -- 959, 1993.

\bibitem{wen}
J.~H. Wen, O.~Choi, H.~Taylor-Weiner, A.~Fuhrmann, J.~V. Karpiak, A.~Almutairi,
  and A.~J. Engler.
\newblock Haptotaxis is cell type specific and limited by substrate
  adhesiveness.
\newblock {\em Cell Mol. Bioeng.}, 8(4):530 -- 542, 2015.

\bibitem{ZeidlerNFA1}
E.~Zeidler.
\newblock {\em Nonlinear functional analysis and its applications. {I}}.
\newblock Springer-Verlag, New York, 1986.
\newblock Fixed-point theorems, Translated from the German by Peter R. Wadsack.

\end{thebibliography}
\end{document}